\documentclass[11pt,a4paper]{article}

\usepackage[a4paper]{geometry}
\usepackage{amssymb,latexsym,mathtools,amsfonts,amsthm}
\usepackage{graphicx}
\usepackage{amsbsy}
\usepackage{hyperref}
\usepackage{fullpage}
\usepackage{enumitem}
\usepackage{graphicx}
\usepackage{color}

\newcommand{\Ai}{\mathrm{Ai}}
\newcommand{\Mei}{\mathrm{Mei}}
\newcommand{\gfrac}[2]{\genfrac{}{}{0pt}{}{#1}{#2}}
\newcommand{\MeijerG}[5]{G^{#1}_{#2} \left( \left. \gfrac{#3}{#4} \right\rvert {#5} \right)}
\newcommand{\compC}{\mathbb{C}}
\newcommand{\realR}{\mathbb{R}}
\newcommand{\intZ}{\mathbb{Z}}
\newcommand{\halfH}[1][\theta]{\mathbb{H}_{#1}}
\newcommand{\bigO}{\mathcal{O}}
\newcommand{\g}{\widetilde{g}}
\newcommand{\E}{\mathbb{E}}
\newcommand{\I}{\widetilde{I}}
\newcommand{\G}{\widetilde\Psi^{(\Mei)}}
\newcommand{\Y}{\widetilde{Y}}
\newcommand{\T}{\widetilde{T}}
\newcommand{\St}{\widetilde{S}}
\newcommand{\J}{\widetilde{J}_c}
\newcommand{\V}{\widetilde{V}}
\newcommand{\Pb}{\widetilde P^{(b)}}
\newcommand{\Vb}{\widetilde V^{(b)}}
\renewcommand{\v}{\widetilde{v}}
\renewcommand{\c}{\widehat{c}}
\newcommand{\N}{\widetilde{N}}
\newcommand{\U}{\widetilde{U}}

\newcommand{\f}{\widetilde{f}}
\renewcommand{\gg}{\widetilde{g}}

\renewcommand{\P}{\widetilde{P}}
\newcommand{\Pmodeltilde}{\widetilde{\mathsf{P}}}
\newcommand{\C}{\widetilde{C}}
\newcommand{\R}{\widetilde{R}}
\newcommand{\Q}{\widetilde{Q}}
\newcommand{\Pinfty}{\widetilde P^{(\infty)}}
\newcommand{\s}{\widetilde{s}}
\newcommand{\n}{\widetilde{n}}
\renewcommand{\b}{\widetilde{b}}
\renewcommand{\a}{\widetilde{a}}

\newcommand{\tP}{\mathcal{P}}
\newcommand{\tR}{\mathcal{R}}

\newcommand{\tRhat}{\widehat{\mathcal{R}}}
\newcommand{\Shat}{\widehat{S}}

\newcommand{\Yhat}{\widehat{Y}}

\newcommand{\SigmaR}{\pmb{\mathit{\Sigma}}}

\newcommand{\DeltaR}{\Delta_{\pmb{\mathit{\Sigma}}}}
\newcommand{\DeltatildeR}{\Delta_{\widetilde{\pmb{\mathit{\Sigma}}}}}

\newcommand{\Z}{\widetilde{Z}}
\newcommand{\Psitilde}{\widetilde{\Psi}}

\newcommand{\Mcyclic}{M_{\mathrm{cyclic}}}

\DeclareMathOperator{\supp}{supp}
\DeclareMathOperator{\diag}{diag}

\newtheorem{theorem}{Theorem}[section]
\newtheorem{lemma}[theorem]{Lemma}
\newtheorem{prop}[theorem]{Proposition}

\newtheorem{RHP}[theorem]{RH problem}

\theoremstyle{remark}
\newtheorem{rmk}[theorem]{Remark}

\theoremstyle{definition}

\numberwithin{equation}{section}
\begin{document}

\title{A vector Riemann-Hilbert approach to the Muttalib-Borodin ensembles}
\author{Dong Wang\footnotemark[1] ~ and \, Lun Zhang\footnotemark[2]}
\maketitle
\renewcommand{\thefootnote}{\fnsymbol{footnote}}
\footnotetext[1]{
  School of Mathematical Sciences, University of Chinese Academy of Sciences, Beijing 100049, P.~R.~China. E-mail: wangdong\symbol{'100}wangd-math.xyz}
\footnotetext[2]{School of Mathematical Sciences and Shanghai Key Laboratory for Contemporary Applied Mathematics, Fudan University, Shanghai 200433, P.~R.~China. E-mail: lunzhang\symbol{'100}fudan.edu.cn}
\begin{abstract}
In this paper, we consider the Muttalib-Borodin ensemble of Laguerre type, a determinantal point process on $[0,\infty)$ which depends on the varying weights $x^{\alpha}e^{-nV(x)}$, $\alpha>-1$, and a parameter $\theta$. For $\theta$ being a positive integer, we derive asymptotics of the associated biorthogonal polynomials near the origin for a large class of potential functions $V$ as $n\to \infty$. This further allows us to establish the hard edge scaling limit of the correlation kernel, which is previously only known in special cases and conjectured to be universal. Our proof is based on a Deift/Zhou nonlinear steepest descent analysis of two $1 \times 2$ vector-valued Riemann-Hilbert problems that characterize the biorthogonal polynomials and the explicit construction of $(\theta+1)\times(\theta+1)$-dimensional local parametrices near the origin in terms of Meijer G-functions.
\end{abstract}

\setcounter{tocdepth}{2} \tableofcontents

\section{Introduction and statement of main results}
\subsection{The model}
The Muttalib-Borodin ensemble refers to $n$ particles $x_1< \dotsb < x_n$ distributed on $[0,\infty)$, following a probability density function of the form
\begin{equation} \label{eq:bioLag}
  \frac{1}{\mathcal{Z}_n}\Delta(x_1,\ldots,x_n)\Delta(x_1^\theta,\ldots,x_n^\theta)\prod_{j=1}^nw(x_j), \qquad \theta >0,
\end{equation}
where $w$ is a weight function over the positive real axis, $\mathcal{Z}_n$ is the normalization constant, and $\Delta(x_1,\ldots,x_n)=\prod_{1\leq i <j \leq n}(x_j-x_i)$ is the standard Vandermonde determinant. As a generalization of the classical unitary invariant ensemble \cite{Anderson-Guionnet-Zeitouni10,Forrester10}, which corresponds to $\theta=1$, this ensemble was first introduced by Muttalib as a toy model in the studies of quasi-$1$ dimensional disordered conductors \cite{Muttalib95}. Because of the two body interaction term $\Delta(x_1^\theta,\ldots,x_n^\theta)$, it gives a more effective description of the disorder phenomena than the classical random matrix theory; cf. \cite{Beenakker97,Lueck-Sommers-Zirnbauer06,Katori-Takahashi12} for some concrete physical models related to \eqref{eq:bioLag} with specified $\theta$ and $w$. Since this ensemble was further studied by Borodin under the general framework of biorthogonal ensembles \cite{Borodin99,Desrosiers-Forrester08}, we call this class of joint probability density function \eqref{eq:bioLag} the Muttalib-Borodin ensemble, following the terminology initiated in \cite{Forrester-Wang15}.

The Muttalib-Borodin ensemble has attracted great interest nowadays, due to its close connections with various stochastic models. For the classical Laguerre or Jacobi weights, it has been shown in \cite{Adler-van_Moerbeke-Wang11,Cheliotis14,Forrester-Wang15} that \eqref{eq:bioLag} can be realized as the eigenvalues of certain random matrices, while the significant progresses achieved recently on random matrix products models \cite{Akemann-Ipsen15,Akemann-Ipsen-Kieburg13,Kuijlaars-Zhang14} have unveiled deep relations between these two ensembles in both finite $n$ and limiting cases; cf. \cite{Forrester-Liu14,Kuijlaars-Stivigny14,Zhang15}. We also refer to \cite{Betea-Occelli20,Betea-Occelli20a} and \cite{Gautie-Le_Doussal-Majumdar-Schehr19,Grela-Majumdar-Schehr21} for the emergences of Muttalib-Borodin ensembles in plane partitions and in the Dyson's Brownian motion.

Apart from wide applications in mathematical physics aforementioned, the Muttalib-Borodin ensemble itself exhibits rich mathematical structures as well, which provides the other reason of intensive studies. To this end, it is helpful to see  that the particles in \eqref{eq:bioLag} form a determinantal point process \cite{Soshnikov00,Tracy-Widom98}. This means that there exits a correlation kernel $K_n(x,y)$ such that the density function can be rewritten in the following determinantal form:
\begin{equation} \label{eq:kernel_intro}
  \frac{1}{n!}\det\left(K_n(x_i,x_j)\right)_{i,j=1}^n.
\end{equation}
For the classical Laguerre and Jacobi weights, it has been shown in Borodin's pioneering work \cite{Borodin99} that the scaling limit of $K_n$ near the origin (also known as the hard edge) converges to a new family of limiting kernels depending on the parameter $\theta$ as $n \to \infty$. This new family of limiting kernels, which describes the limiting distribution of the smallest particles in \eqref{eq:bioLag}, is a natural generalization of the classical Bessel kernel \cite{Forrester93,Tracy-Widom94b}, and reduces to the Meijer G-kernels encountered mainly in the products of random matrices and related models (cf. \cite{Akemann-Strahov16,Bertola-Bothner14,Bertola-Gekhtman-Szmigielski14,Kuijlaars-Zhang14,Silva-Zhang20}) if $\theta$ or $1/\theta$ is a positive integer, as observed in \cite{Kuijlaars-Stivigny14}. The corresponding gap probabilities near the origin are investigated in \cite{Charlier-Claeys20,Charlier-Lenells-Mauersberger19,Claeys-Girotti-Stivigny19,Zhang17} and the local limits of $K_n$ away from the origin for the classical weights are given in \cite{Forrester-Wang15} and \cite{Zhang15}. The limiting mean distribution of the particles in Muttalib-Borodin ensemble is formulated as the minimizer of a (vector) equilibrium problem in \cite{Claeys-Romano14,Kuijlaars16}, and the large deviation results can be found in \cite{Bloom-Levenberg-Totik-Wielonsky17,Butez17,Credner-Eichelsbacher15,Eichelsbacher-Sommerauer-Stolz11}; see also \cite{Charlier21,Lambert18,Alam-Muttalib-Wang-Yadav20} for other investigations and extensions of the Muttalib-Borodin ensemble.

In this paper, we are concerned with the Muttalib-Borodin ensemble of Laguerre type, i.e.,
\begin{equation} \label{def:weight}
  w(x)=x^{\alpha}e^{-nV(x)}, \quad \alpha>-1, \quad x>0,
\end{equation}
where $V$ is a potential function (also known as the external field) independent of $n$ and which satisfies the growth condition
\begin{equation}\label{eq:loggrowth}
\lim_{x\to +\infty}\frac{V(x)}{\log x}= + \infty.
\end{equation}
Inspired by the principle of \emph{universality}, a fundamental issue in random matrix theory, it is generally conjectured that, as $n\to\infty$, the local behaviours of the associated biorthogonal polynomials and correlation kernel are the same as those derived for classical weights for a large class of $V$ with the parameter $\theta>0$ fixed. An attempt to justify this conjecture is given in recent works \cite{Kuijlaars-Molag19,Molag20}, where local universality of the correlation kernel near the origin is established for $1/\theta \in \mathbb{N}=\{1,2,\ldots\}$ under weak conditions on $V$. It is the aim of the present work to prove a parallel universality result for a general class of $V$ and $\theta\in \mathbb{N}$. Moreover, our approach is different from that used in \cite{Kuijlaars-Molag19,Molag20}, which might be of independent interest and serves as the other main contribution of this paper. We next state our main results.

\subsection{Statement of main results} \label{subsec:statements}

Before stating the main results, we need to impose some conditions on the potential function $V$ and introduce some other functions. We start with the fact that if $V$ is continuous on $[0, +\infty)$ and satisfies \eqref{eq:loggrowth}, the limiting empirical measure of the particles in \eqref{eq:bioLag} with Laguerre type weight functions \eqref{def:weight} exists as $n\to \infty$, and it is the unique probability measure over $[0,+\infty)$ that minimizes the energy functional
\begin{equation} \label{eq:equilibriump}
  I^{(V)}(\nu):=\frac{1}{2}\iint \log \frac{1}{\lvert x-y \rvert} d\nu(x) d \nu(y)+\frac{1}{2}\iint \log \frac{1}{\lvert x^\theta-y^\theta \rvert} d\nu(x) d \nu(y) +\int V(x) d\nu(x);
\end{equation}
see \cite[Theorem 2.1 and Corollary 2.2]{Eichelsbacher-Sommerauer-Stolz11} and \cite[Theorem 1.2 and Corollary 1.4]{Butez17}.
Moreover, the equilibrium measure $\mu=\mu^{(V)}$ is characterized by the following Euler-Lagrange conditions:
\begin{align}
  \int \log \lvert x-y \rvert d\mu(y)+ \int \log \lvert x^\theta-y^\theta \rvert d\mu(y)-V(x)&=\ell, \quad x \in \supp(\mu), \label{eq:EL1}
  \\
  \int \log \lvert x-y \rvert d\mu(y)+ \int \log \lvert x^\theta-y^\theta \rvert d\mu(y)-V(x)&\leq \ell, \quad x \in [0,+\infty), \label{eq:EL2}
\end{align}
where $\ell$ is some real constant.

Following \cite{Kuijlaars-Molag19,Molag20}, we require the potential $V$ to be \emph{one-cut regular}, in the sense that
\begin{itemize}
\item the equilibrium measure $\mu$ is supported on one interval $[0, b]$ with a continuous density function $\psi=\psi^{(V)}$ for some $b=b^{(V)}>0$, that is,
\begin{equation}\label{eq:dmu}
d\mu(x) = \psi (x) dx,\qquad x \in (0, b);
\end{equation}
\item
  $\psi(x) > 0$ on $(0, b)$, and there exist two positive numbers $d_1=d^{(V)}_1$ and $d_1=d^{(V)}_1$ such that
  \begin{equation}\label{eq:psiasy}
  \psi(x) = \left\{
                      \begin{array}{ll}
                        d_1 x^{-\frac{1}{\theta + 1}}(1 + o(1)), & \hbox{$x \to 0_+$,} \\
                        d_2(b - x)^{\frac12}(1 + o(1)), & \hbox{$x\to  b_-$;}
                      \end{array}
                    \right.
\end{equation}
\item the inequality \eqref{eq:EL2} holds strictly for $x\in (b,+\infty)$.
\end{itemize}
An explicit expression of $\psi$ is given in \cite{Claeys-Romano14}; see also Section \ref{sec:equmeasure} below. We then introduce the $g$-functions  (cf.~\cite[Equations (4.4)--(4.5)]{Claeys-Romano14})
\begin{align}
  g(z) := {}& \int^b_0 \log(z - y) \psi(y) dy, & z \in {}& \compC \setminus (-\infty, b], \label{def:g} \\
 \g(z) := {}& \int^b_0 \log(z^{\theta} - y^{\theta}) \psi(y) dy, & z \in {}& \mathbb{H}_{\theta} \setminus [0, b], \label{def:tildeg}
\end{align}
where the branch cut of $\log z$ is taken along the negative real axis and
\begin{equation}\label{def:Htheta}
  \mathbb{H}_\theta:=\left\{ z\in \mathbb{C} \mid z = 0 \text{ or } -\frac{\pi}{\theta} < \arg z <  \frac{\pi}{\theta} \right\}.
\end{equation}

Our first result is about asymptotics of the biorthogonal polynomials associated with the model. They have their own probability meaning and are also important in our study of the correlation kernel of the Muttalib-Borodin ensemble; see \eqref{eq:p_q_meaning} and \eqref{eq:sum_form_K} below. These polynomials are two sequences of monic polynomials $\{p_j(x)=p^{(V)}_{n, j}(x) \}_{j=0}^{\infty}$ and $\{q_k(x)=q^{(V)}_{n, k}(x)\}_{k=0}^\infty$ that satisfy the biorthogonal conditions
\begin{equation}\label{eq:pqbioOP}
  \int^{\infty}_0 p_{j}(x) q_{k}(x^\theta) x^{\alpha} e^{-nV(x)} dx = \kappa_j\delta_{j, k},
\end{equation}
where $\kappa_j=\kappa_{n,j}^{(V)}>0$. The functions $p_n(z)$ and $q_n(z^{\theta})$ can be interpreted as the averages over the Muttalib-Borodin ensemble. Indeed, similar to the proofs of \cite[Proposition 2.1]{Bleher-Kuijlaars04a} and \cite[Proposition 1]{Claeys-Wang11}, it is easily seen that
\begin{equation} \label{eq:p_q_meaning}
    p_n(z) = \E \left( \prod^n_{j = 1} (z - x_j) \right), \quad q_n(z^{\theta}) = \E \left( \prod^n_{j = 1} (z^{\theta} - x^{\theta}_j) \right),
\end{equation}
where $\E$ denotes the expectation with respect to \eqref{eq:bioLag} and \eqref{def:weight}.

Let $\MeijerG{m, n}{p, q}{a_1,\ldots,a_p}{b_1,\ldots,b_q}{z}$ be the Meijer G-function (see Appendix \ref{appendix:Meijer} for a brief introduction), and define
\begin{equation} \label{eq:defn_rho}
c =  c^{(V)} =  b \theta(1+\theta)^{-1-\frac1\theta},
\qquad
\rho = \rho^{(V)} = \left. \theta^{-1} d_1 \pi \middle/ \sin\left(\frac{\pi}{1 + \theta}\right), \right.
\end{equation}
with $b$ and $d_1$ given in \eqref{eq:dmu} and \eqref{eq:psiasy}, respectively, we have the following asymptotic behaviours of the biorthogonal polynomials $p_n = p^{(V)}_{n, n}$ and $q_n = q^{(V)}_{n, n}$ defined by \eqref{eq:pqbioOP} near the origin.
\begin{theorem} \label{thm:pqkappa}
Suppose the potential function $V$ is real analytic on $[0,\infty)$ and one-cut regular. As $n\to \infty$, we have, for $\theta \in \mathbb{N}$,
  \begin{equation}
    p_n \left( \frac{z}{(\rho n)^{1 + 1/\theta}} \right) = (-1)^n C_n \left( z^{\theta - \alpha - 1} \MeijerG{\theta, 0}{0, \theta + 1}{-}{\frac{\alpha - \theta + 1}{\theta}, \frac{\alpha - \theta + 2}{\theta}, \dotsc, \frac{\alpha - 1}{\theta}, \frac{\alpha}{\theta}, 0}{z^{\theta}} + \bigO ( n^{-\frac{1}{2\theta + 1}} ) \right), \label{eq:p_thm}
\end{equation}
and
\begin{equation}
    q_n \left( \frac{z^{\theta}}{(\rho n)^{\theta + 1}} \right) = (-1)^n \C_n \left( \MeijerG{1, 0}{0, \theta + 1}{-}{0, -\frac{\alpha}{\theta}, \frac{1 - \alpha}{\theta}, \dotsc, \frac{\theta - 1 - \alpha}{\theta}}{z^{\theta}} + \bigO ( n^{-\frac{1}{2\theta + 1}} )\right), \label{eq:q_thm}
  \end{equation}
uniformly for $z$ in any compact subset of $\compC$ and $\halfH \cup \{ \text{boundary of $\halfH$} \}$, respectively, where $\halfH$ is defined in \eqref{def:Htheta},
  \begin{equation}\label{def:Cn}
      C_n = (2\pi)^{1 - \frac{\theta}{2}} \sqrt{\theta} c^{\frac{2(\alpha+1)-\theta}{2(1+\theta)}} (\rho n)^{\frac{\alpha + 1}{\theta} - \frac{1}{2}} e^{n \Re g(0)}, \quad \C_n = (2\pi)^{\frac{\theta}{2}} c^{\frac{(\alpha + 1/2)\theta}{\theta + 1}} (\rho n)^{\frac{1}{2} + \alpha} e^{n \Re \g(0)}, \end{equation}
with the constants $c$, $\rho$, and the functions $g$, $\g$ given in \eqref{eq:defn_rho}, \eqref{def:g} and  \eqref{def:tildeg}.
\end{theorem}

\begin{rmk}
  If $\theta=1$, it follows from \cite[Formula 10.9.23]{Boisvert-Clark-Lozier-Olver10} that $z^{-\alpha}\MeijerG{1, 0}{0, 2}{-}{\alpha, 0}{z}=\MeijerG{1, 0}{0, 2}{-}{0,-\alpha}{z}=z^{-\alpha/2}J_{\alpha}(2\sqrt{z})$ with $J_\alpha$ being the Bessel function of the first kind of order $\alpha$, which appears in the strong asymptotics of Laguerre type orthogonal polynomials; cf. \cite{Vanlessen07}. Our asymptotic analysis, leading to the proof of the above theorem, also enables us to obtain asymptotics of $p_n$ and $q_n$ in other parts of the complex plane. All the ingredients for obtaining such results are presented, but we will not write the details down neither comment them any further.
\end{rmk}

With the biorthogonal polynomials given in \eqref{eq:pqbioOP}, it is known that the correlation kernel in \eqref{eq:kernel_intro} can be written as \cite{Borodin99}
\begin{equation} \label{eq:sum_form_K}
  K_n(x, y) = x^{\alpha} e^{-nV(x)}\sum^{n - 1}_{j = 0} \frac{p_{j}(x)q_{j}(y^{\theta})}{\kappa_j},
\end{equation}
where again we emphasize that $p_j = p^{(V)}_{n, j}$, $q_j = q^{(V)}_{n, j}$ and $\kappa_j = \kappa^{(V)}_{n, j}$ depend on $n$ and the potential function $V$. This, together with Theorem \ref{thm:pqkappa}, allows us to establish the hard edge scaling limit of the correlation kernel, which proves the universality conjecture in the underlying case. We note that unless $\theta = 1$, there is no easy way to make $K_n(x, y)$ symmetric.

To state the result, we define a family of kernels
\begin{multline}\label{eq:defn_kernel_integral}
K^{(\alpha, \theta)}(x, y)=\theta^2 x^{\theta-1} \int_0^1 u^{\theta - 1} \MeijerG{\theta, 0}{0, \theta + 1}{-}{\frac{\alpha - \theta + 1}{\theta}, \frac{\alpha - \theta + 2}{\theta}, \dotsc, \frac{\alpha - 1}{\theta}, \frac{\alpha}{\theta}, 0}{(ux)^{\theta}}
\\
\times \MeijerG{1, 0}{0, \theta + 1}{-}{0, -\frac{\alpha}{\theta}, \frac{1 - \alpha}{\theta}, \dotsc, \frac{\theta - 1 - \alpha}{\theta}}{(uy)^{\theta}}du,
\end{multline}
for $\alpha>-1$ and $\theta\in\mathbb{N}$, which coincide with the Meijer G-kernels given in \cite[Theorem 5.3]{Kuijlaars-Zhang14} after proper scaling. Due to a technical reason, we also assume that
\begin{equation} \label{eq:one-cut_regular}
  V''(x)x+V'(x)>0, \qquad x > 0.
\end{equation}
It comes out that the conditions \eqref{eq:loggrowth} and \eqref{eq:one-cut_regular} ensure the potential function $V$ is one-cut regular, as shown in \cite[Proposition 3.6]{Kuijlaars-Molag19} for $\theta = 1/2$, and later confirmed in \cite{Molag20} for all rational $\theta$. Here we note that the argument in \cite{Kuijlaars-Molag19} works for all $\theta > 0$; see also \cite[Theorem 1.8]{Claeys-Romano14} for a slightly weaker sufficient condition for the one-cut regular property. The conditions \eqref{eq:loggrowth} and \eqref{eq:one-cut_regular} will also imply the potential functions $V_{\tau}$ defined in \eqref{def:vtau} are all one-cut regular for all $\tau \in (0, 1]$, which plays an important role in the proof of our universal result; see more discussions in Section \ref{subsec:proof_kernel_formula} below.

Our second result is the following scaling limit of the correlation kernel $K_n(x,y)$ near the origin.
\begin{theorem} \label{thm:kernel}
For real analytic $V$ that satisfies the conditions \eqref{eq:loggrowth} and \eqref{eq:one-cut_regular}, we have, for $\theta\in \mathbb{N}$,
  \begin{equation} \label{eq:thm:kernel}
    \lim_{n \to \infty} (\rho n)^{-(1 + \frac{1}{\theta})} K_n \left( \frac{x}{(\rho n)^{1 + 1/\theta}}, \frac{y}{(\rho n)^{1 + 1/\theta}} \right) = K^{(\alpha, \theta)}(x, y),
  \end{equation}
uniformly for $x, y$ in compact subsets of $(0,\infty)$, where $\rho$ and $K^{(\alpha, \theta)}$ are given in \eqref{eq:defn_rho} and \eqref{eq:defn_kernel_integral}.
\end{theorem}

\subsection{Proof strategy and methodological novelty} \label{sec:proof}

To prove the results, we adopt the Riemann-Hilbert (RH) approach \cite{Deift99,Deift-Kriecherbauer-McLaughlin-Venakides-Zhou99}, whose development had an important influence to nonlinear mathematical physics. It has been shown in \cite{Claeys-Romano14} that the biorthogonal polynomials underlying the Muttalib-Borodin ensemble can also be characterized by a $1 \times 2$ vector-valued RH problem for any $\theta \geq 1$. Although vector RH problems have appeared in the context of integrable systems for a long time (cf. \cite{Vekua67} and Percy Deift's lecture notes \cite{Deift17a} on inverse scattering theory in 1991), they have not received much attention in asymptotic analysis. Claeys and the first-named author have firstly succeeded in performing a Deift/Zhou steepest descent analysis of a vector-valued RH problem in the context of random matrices with equispaced external source \cite{Claeys-Wang11}. Almost simultaneously to the present work, Charlier investigated asymptotics of Muttalib-Borodin determinants with Fisher-Hartwig singularities \cite{Charlier21}, following similar idea in \cite{Claeys-Wang11}. Although we are also inspired from \cite{Claeys-Wang11}, the situation encountered in this work is essentially different from that in \cite{Charlier21} and is much more challenging. More precisely, due to the fact that there does not exist an open disk centred at the origin and lying in $\compC \cap \halfH$, the greatest technical obstacle of our analysis lies in the construction of a local approximation therein. To solve this issue for integer $\theta$, we convert the $1\times 2$ vector RH problem locally to an equivalent one of size $1\times(\theta+1)$, but defined in a small disk of $\compC$. With the aid of a $(\theta+1)\times(\theta+1)$ model RH problem solved in terms of the Meijer G-functions, we are able to complete the construction of a local parametrix near the origin. Conceptually, in our paper the model RH problem serves as an operator to trivialize the original RH problem, rather than an approximated local solution to the original one. This point of view may be also used in the general real $\theta$ case.

It is also worthwhile to emphasize that the set-up of the vector RH problem, the major part of our asymptotic analysis on $p_n$ and $q_n$, and our proof of the local universality of the correlation kernel (which relies on the summation formula \eqref{eq:sum_form_K} but not the Christoffel-Darboux formula that is available for rational $\theta$ \cite[Theorem 1.1]{Claeys-Romano14}) are applicable to all real parameter $\theta$. The integral property of $\theta$ only plays a role in the construction of a local parametrix near the origin. As a consequence, to generalize our current work to the general real $\theta$ case, we only need to solve the corresponding local paramatrix problem at the origin. It is still a challenge, but now \emph{localized} to tackle, and will be the goal of a follow-up paper.

Before our work, the RH problem has also been applied in asymptotic analysis of the Muttalib-Borodin ensembles of Laguerre type by Kuijlaars and Molag in a different approach for the special $\theta = 1/r$ with $r \in \mathbb{N}$ \cite{Kuijlaars-Molag19,Molag20}. For such $\theta$, the biorthogonal polynomials $p_j, q_k$ in \eqref{eq:pqbioOP} can be viewed as multiple orthogonal polynomials \cite{Kuijlaars10a}. The scaling limit of correlation kernel then follows from an asymptotic analysis of the associated $(r+1)\times(r+1)$ RH problem \cite{Geronimo-Kuijlaars-Van_Assche00}, which generalizes the classical $2 \times 2$ RH problem for orthogonal polynomials \cite{Fokas-Its-Kitaev92}. Seemingly, by the transformation $x\to x^{1/\theta}$, the result for $\theta = 1/r$ covers that for $\theta = r$, but the comments after \cite[Theorem 1.1]{Molag20} explain the limitations. Up to now, it is not clear how to generalize this approach to other rational $\theta$, and more unlikely to irrational $\theta$.

In conclusion, the present work provides a new perspective on resolving the local universality conjecture in the Muttalib-Borodin ensemble for general weight functions and real $\theta$, with only the local obstacle to be conquered.

\paragraph{Outline}

The rest of this paper is organized as follows. We present some preliminary results in Section \ref{sec:preli} to facilitate further analysis, which particularly includes the vector-valued RH problems characterizing the biorthogonal polynomials for Muttalib-Borodin ensemble. We carry out a Deift-Zhou steepest descent analysis on the RH problems in Sections \ref{sec:asyY} and \ref{sec:asytildeY}. After the RH asymptotic analysis is finished, the proofs of our main results, i.e., Theorems \ref{thm:pqkappa} and \ref{thm:kernel}, are given in Section \ref{sec:provmainresult}.

\paragraph{Notations}
Throughout this paper, the following notations are frequently used. We denote by $\compC_{\pm}=\{z\in\compC \mid \pm \Im z>0\}$, by $D(z_0, \delta)$ the open disc centred at $z_0$ with radius $\delta>0$, i.e.,
\begin{equation} \label{eq:size_U_b}
  D(z_0,\delta) := \{z\in \mathbb{C} ~|~ \lvert z - z_0 \rvert < \delta \},
\end{equation}
and by $I_m$, $m\in \mathbb{N}$, the identity matrix of size $m$. If $A$ and $B$ are two square matrices of sizes $m$ and $n$, then $A \oplus B$ is the $(m + n) \times (m + n)$ block matrix $
\begin{pmatrix}
  A & 0 \\
  0 & B
\end{pmatrix}
$.
For $\theta\in \mathbb{N}$, the $(\theta + 1) \times (\theta + 1)$ matrix $\Mcyclic$ is defined by
\begin{equation} \label{eq:defn_Mcyclic}
  \Mcyclic = I_1 \oplus
  \begin{pmatrix}
    0 & 0 & \cdots & 0 & 1 \\
    1 & 0 & \cdots & 0 & 0\\
    0 & 1 & \ddots & 0 & 0 \\
    \vdots& \ddots & \ddots & \ddots & \vdots\\
    0& \cdots & 0 & 1 & 0
  \end{pmatrix}
  =
  \begin{pmatrix}
    1 & 0 & 0 & \cdots & 0 & 0 \\
    0 & 0 & 0 & \cdots & 0 & 1 \\
    0 & 1 & 0 & \cdots & 0 & 0\\
    0 & 0 & 1 & \ddots & 0 & 0 \\
    \vdots& \vdots& \ddots & \ddots & \ddots & \vdots\\
    0& 0& \cdots & 0 & 1 & 0
  \end{pmatrix}.
\end{equation}
It is easily seen that
\begin{equation}
\det ( \Mcyclic)=(-1)^{\theta-1},\qquad \Mcyclic^{-1}=\begin{pmatrix}
      1 & 0 & 0 & 0 & \cdots & 0 \\
      0 & 0 & 1 & 0 & \cdots & 0 \\
      0 & 0 & 0 & 1 & \ddots & 0 \\
      \vdots & \vdots & \vdots & \ddots & \ddots & 0 \\
      0 & 0 & 0 & 0 & \ddots & 1  \\
      0 & 1 & 0 & 0 & \cdots & 0
    \end{pmatrix}.
\end{equation}
Finally, when we say ``$X = (X_1, X_2)$ is analytic in $(D_1, D_2)$'' for some $D_1,D_2 \subseteq \mathbb{C}$, it means that $X_1$ and $X_2$ are analytic in $D_1$ and $D_2$, respectively.

\section{Preliminaries}\label{sec:preli}

In this section, we review the precise description of the equilibrium measure and the vector-valued RH problems associated with the biorthogonal polynomials. The properties of the $g$-functions are also collected for later use.

\subsection{Integral representation of the equilibrium measure and properties of the $g$-functions} \label{sec:equmeasure}

If the real analytic potential function $V$ is one-cut regular and satisfies the growth condition \eqref{eq:loggrowth}, it follows from
\cite[Theorems 1.8 and 1.11]{Claeys-Romano14} that the density function $\psi$ of the equilibrium measure that minimizes the energy functional \eqref{eq:equilibriump} can be constructed explicitly via an integral formula. To state the relevant results, we need a mapping defined by
\begin{equation}\label{def:Jcs}
  J_c(s)=c(s + 1) \left( \frac{s + 1}{s} \right)^{1/\theta}, \quad s\in\mathbb{C} \setminus [-1,0],
\end{equation}
where the constant $c$ is determined through \eqref{eq:defn_c} below, and the branch of the $1/\theta$ power function is chosen such that $J_c(s) \sim c s$ as $s \to \infty$. It has two critical points $-1$ and $s_b=1/\theta$, which are mapped respectively to $0$ and
\begin{equation}\label{def:b}
  b = c \theta^{-1} (1+\theta)^{1+\frac{1}{\theta}}.
\end{equation}
Note that $J_{c}(s)$ is real for $s\in (-\infty,-1]\cup [s_b, +\infty)$, while the two complex conjugate curves $\gamma_1$ and $\gamma_2$, shown in Figure \ref{fig:mapJ}, going from $-1$ to $s_b$ are mapped to the interval $[0,b]$. Let $\gamma := (-\gamma_1) \cup \gamma_2$ and denote by $D$ the region bounded by $\gamma$. We have from \cite[Lemma 4.1]{Claeys-Romano14} that $J_c$ maps $\mathbb{C} \setminus \overline{D}$ bijectively to $\mathbb{C} \setminus [0,b]$, and maps $D\setminus [-1,0]$ bijectively to $\mathbb{H}_\theta \setminus [0,b]$, where $\halfH$ is defined in \eqref{def:Htheta}; see Figure \ref{fig:mapJ} for an illustration.
\begin{figure}
  \begin{minipage}[c]{5cm}
    \includegraphics{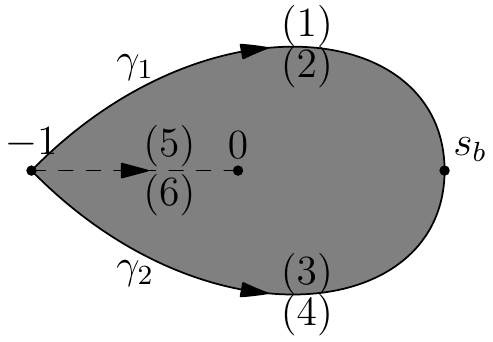}
  \end{minipage}
  \hspace{\stretch{1}}
  $\begin{aligned}
    & \phantom{\int} \\
    J_c: {}& D \setminus [-1, 0] \to \halfH \setminus [0, b] \\
    & \phantom{\int} \\
    & \phantom{\int} \\
    & \phantom{\int} \\
    J_c: {}& \compC \setminus \overline{D} \to \compC \setminus [0, b]
  \end{aligned}$
  \hspace{\stretch{1}}
  \begin{minipage}[c]{3cm}
    \includegraphics{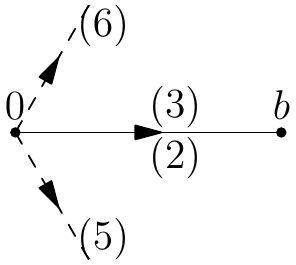}
    \linebreak
    \vspace{1cm}
    \linebreak
    \includegraphics{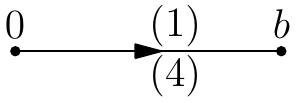}
  \end{minipage}
 \caption{The mapping properties of $J_c$ and the region $D$ (shaded area).}
  \label{fig:mapJ}
\end{figure}

Under the aforementioned conditions imposed on $V$, one has (see \cite[Theorems 1.8 and 1.11]{Claeys-Romano14}) the equilibrium measure $\mu$ is supported on $[0,b]$, where $b$ is given by \eqref{def:b} with $c$ determined uniquely by the equation
\begin{equation} \label{eq:defn_c}
  \frac{1}{2\pi i}\oint_{\gamma}\frac{V'(J_{c} (s))J_{c}(s)}{s}ds=1+\theta.
\end{equation}
Note that in \eqref{eq:defn_rho}, we define the constant $c$ through $b$ by \eqref{def:b}. The density function $\psi$ then admits the following integral representation:
\begin{equation}\label{def:psi}
  \psi(x)=\frac{1}{2 \pi^2 x}\int_0^b(V''(y)y + V'(y))\log \left| \frac{I_+(y)-I_-(x)}{I_+(y)-I_+(x)}\right|dy, \quad x\in[0,b],
\end{equation}
where $I_+(x)$ and $I_{-}(x)$ stand for the inverse images of $x\in[0,b]$ belonging to the curves $\gamma_1$ and $\gamma_2$, respectively. Moreover, by \cite[Equations (1.33)--(1.36)]{Claeys-Romano14}, we have
\begin{align}
  d_1 = & -\frac{1}{\pi^2}c^{-\frac{\theta}{1+\theta}}\sin \left(\frac{\pi}{1+\theta}\right)\int_0^b(V''(y)y+V'(y))\Im \left(\frac{1}{1+I_+(y)}\right)dy, \label{def:d1} \\
  d_2 = & -\frac{1}{\pi^2 b}\sqrt{\frac{2}{J''_c(s_b)}}\int_0^b(V''(y)y+V'(y))\Im \left( \frac{1}{I_+(y)-s_b}\right)dy, \label{eq:defn_d2}
\end{align}
where $d_1$ and $d_2$ are the two constants given in \eqref{eq:psiasy}. Thus, the regularity conditions therein
can be sharpened as
\begin{equation}\label{eq:localb}
  \psi(x)=\left\{
           \begin{array}{ll}
             d_1x^{-\frac{1}{1+\theta}}+ \bigO(x^{\frac{\theta-1}{\theta+1}}), & \hbox{$x\to 0_+$,} \\
             d_2(b-x)^{\frac12} f(x), & \hbox{$x \to b_-$,}
           \end{array}
         \right.
\end{equation}
for some real analytic function $f$ satisfying $f(b) = 1$.

We conclude this subsection with some properties of the $g$-functions.
\begin{prop} \label{prop:propg}
  The $g$-functions defined in \eqref{def:g} and \eqref{def:tildeg} have the following properties.
  \begin{enumerate}[label=\emph{(\alph*)}, ref=(\alph*)]
  \item \label{enu:prop:propg_a}
    $g(z)$ and $\g(z)$ are analytic in $\compC \setminus (-\infty, b]$ and $\mathbb{H}_{\theta} \setminus [0, b]$, respectively. Furthermore, as $z\to \infty$, we have
    \begin{align}
      g(z) = {}& \log z+ \bigO(z^{-1}), & z \in {}& \compC, \\
      \g(z)= {}& \theta \log z+ \bigO(z^{-\theta}), & z\in {}& \mathbb{H}_{\theta}. \label{eq:gtilde_asy}
    \end{align}
  \item
    The $g$-functions satisfy the following boundary conditions:
    \begin{align}
      \g(e^{-\frac{\pi i}{\theta}}x) & =\g(e^{\frac{\pi i}{\theta}}x)-2 \pi i, & x > {}& 0, \label{eq:tildgpm} \\
      g_+(x) & =g_-(x)+2 \pi i, & x < {}& 0. \label{eq:gpm}
    \end{align}
  \item
    For $x\in(0,b)$, we have
    \begin{equation}\label{eq:psiandg}
      \psi(x)=-\frac{1}{2\pi i}(g_+'(x)-g_-'(x))=-\frac{1}{2\pi i}(\g_+'(x)-\g_-'(x)).
    \end{equation}
  \item \label{enu:prop:propg_d}
    With the same constant $\ell$ as in \eqref{eq:EL1}, we have
    \begin{align}
      g_{\pm}(x)+\g_{\mp}(x)-V(x)-\ell = {}& 0, & x \in {}& (0,b], \label{eq:gequal}      \\
      g(x)+\g(x)-V(x)-\ell < {}& 0, & x > {}& b. \label{eq:gequal2}
    \end{align}
  \item If $\theta>1$, we have, as $z \to 0$,
    \begin{equation} \label{eq:asy_formula_for_g}
      g(z) = g_+(0) +
        \begin{cases}
          \frac{(1+\theta)d_1\pi}{\theta} \frac{e^{\frac{2+\theta}{1+\theta}\pi i}}{\sin (\frac{\pi}{1+\theta})}z^{\frac{\theta}{1+\theta}} + \bigO(z), & \arg z \in (0,\pi),
          \\
          - 2 \pi i+\frac{(1+\theta)d_1\pi}{\theta} \frac{e^{\frac{\theta}{1+\theta}\pi i}}{\sin (\frac{\pi}{1+\theta})}z^{\frac{\theta}{1+\theta}} + \bigO(z), & \arg z \in (-\pi,0),
        \end{cases}
    \end{equation}
    where $g_+(0) = \lim_{z\to 0,\ 0 < \arg z < \pi} \g(z)$, $d_1$ is given in \eqref{def:d1}, and
    \begin{equation} \label{eq:asy_formula_for_g_tilde}
      \g(z) = \g_+(0) +
      \begin{cases}
        \frac{(1+\theta)d_1\pi}{\theta} \frac{e^{\frac{1+2\theta}{1+\theta}\pi i}}{\sin (\frac{\pi}{1+\theta} )}z^{\frac{\theta}{1+\theta}} + \bigO(z^{\frac{2\theta}{1+\theta}}), & \arg z \in (0,\frac{\pi}{\theta}), \\
        -2 \pi i+\frac{(1+\theta)d_1\pi}{\theta}  \frac{e^{\frac{3+2\theta}{1+\theta}\pi i}}{\sin (\frac{\pi}{1+\theta})}z^{\frac{\theta}{1+\theta}} + \bigO(z^{\frac{2\theta}{1+\theta}}), & \arg z \in (-\frac{\pi}{\theta},0),
      \end{cases}
    \end{equation}
    where $\g_+(0) = \lim_{z\to 0,\ 0 < \arg z < \pi/\theta} \g(z)$.
  \end{enumerate}
\end{prop}
\begin{proof}
The proofs of items \ref{enu:prop:propg_a}--\ref{enu:prop:propg_d} follow directly from the definitions \eqref{def:g}--\eqref{def:tildeg} and the Euler-Lagrange conditions for the equilibrium problem \eqref{eq:equilibriump} under our assumptions.

To show \eqref{eq:asy_formula_for_g}, it suffices to consider the behaviour of $g'(z)$ at $z=0$. By \eqref{def:g}, it is readily seen that, with $d_1$ given in \eqref{def:d1},
\begin{equation}\label{eq:gderi}
  g'(z)=\int_0^b\frac{\psi(y)}{z-y}dy = I_1(z)+ I_2(z),
\end{equation}
where
\begin{equation}
 I_1(z) = \int_0^b\frac{d_1y^{-\frac{1}{1+\theta}}}{z-y}dy, \qquad I_2(z) = \int_0^b\frac{\psi(y)-d_1y^{-\frac{1}{1+\theta}}}{z-y}dy.
\end{equation}
By \cite[Sections~30]{Muskhelishvili72}, we have, as $z\to 0$,
\begin{equation}
  I_1(z) = -d_1\pi \sin \left( \frac{\pi}{1+\theta} \right)^{-1} e^{\pm\frac{\pi i}{1+\theta}} z^{-\frac{1}{1+\theta}} + \bigO(1), \qquad  z \in \compC_{\pm}.
\end{equation}
From \eqref{eq:localb} and \cite[Section~29]{Muskhelishvili72}, it also follows that $I_2(z) = \bigO(1)$ as $z\to 0$. Combining the estimates of $I_1$ and $I_2$ as $z \to 0$, we obtain \eqref{eq:asy_formula_for_g} by integrating both sides of \eqref{eq:gderi} and \eqref{eq:gpm}.

Similarly, to show \eqref{eq:asy_formula_for_g_tilde}, we write
\begin{equation}\label{eq:tildegder}
\g'(z)=\theta z^{\theta-1}\int_0^b\frac{\psi(y)}{z^\theta-y^\theta}dy=z^{\theta-1}\int_0^{b^\theta}\frac{\psi(t^\frac{1}{\theta})t^{\frac{1}{\theta}-1}}{z^\theta-t}dt = z^{\theta-1}(\I_1(z) + \I_2(z)),
\end{equation}
where
\begin{equation}
\I_1 (z)= \int_0^{b^\theta}\frac{d_1t^{-\frac{\theta}{1+\theta}}}{z^\theta-t}dt, \qquad \I_2(z) = \int_0^{b^\theta}\frac{\psi(t^\frac{1}{\theta})t^{\frac{1}{\theta}-1}-d_1t^{-\frac{\theta}{1+\theta}}}{z^\theta-t}dt.
\end{equation}
As $z\to 0$, since $\psi(t^{1/\theta})t^{1/\theta-1}=d_1t^{-\theta/(1+\theta)}(1+\bigO(t^{1/(1+\theta)}))$ as $t \to 0$, it is readily seen that
\begin{equation}
  z^{\theta-1} \I_1(z) = d_1\pi \frac{e^{\mp\frac{\pi i}{1+\theta}}}{\sin (\frac{\theta}{1+\theta}\pi )}z^{-\frac{1}{1+\theta}}+\bigO(z^{\theta-1}), \quad z \in \halfH \cap \compC_{\pm},
\end{equation}
and $z^{\theta-1} \I_2(z) = \bigO(z^{\frac{\theta-1}{1+\theta}})$ as $z \to 0$. This, together with \eqref{eq:tildegder} and \eqref{eq:tildgpm}, implies the desired result.
\end{proof}

\subsection{RH characterization of the biorthogonal polynomials}\label{sec:RHP}

As mentioned in Section \ref{sec:proof}, the proofs of our results rely on $1 \times 2$ vector-valued RH problems characterizing the biorthogonal polynomials $p_j$  and $q_k$ in \eqref{eq:pqbioOP}, as shown in \cite[Section 3]{Claeys-Romano14}. More precisely, let
\begin{equation} \label{def:Y}
  Y (z) = (p_n(z), Cp_n(z)),
\end{equation}
where
\begin{equation} \label{eq:defn_of_Cp_n}
  Cp_n(z) = \frac{1}{2 \pi i}\int^{\infty}_0 \frac{p_n(x)}{x^\theta - z^\theta}w(x) dx=\frac{1}{2 \pi i}\int^{\infty}_0 \frac{p_n(x)}{x^\theta - z^\theta}x^{\alpha}e^{-nV(x)} dx
\end{equation}
is the modified Cauchy transform of the polynomial $p_n$. We will consider $Cp_n(z)$ as a function defined in $\mathbb{H}_{\theta} \setminus [0, +\infty)$. By \cite[Theorem 1.4]{Claeys-Romano14}, $Y$ is the unique solution of the following RH problem.

\begin{RHP} \label{RHP:original_p}
\hfill
  \begin{enumerate}[label=\emph{(\arabic*)}, ref=(\arabic*)]
  \item \label{enu:RHP:original_p:1}
    $Y = (Y_1, Y_2)$ is analytic in $(\compC, \mathbb{H}_{\theta} \setminus [0, +\infty))$.
  \item \label{enu:RHP:original_p:2}
    $Y$ has continuous boundary values $Y_{\pm}$ when approaching $(0, +\infty)$ from above $(+)$ and below $(-)$\footnote{All RH problems in this paper satisfy the continuous boundary value condition along an edge of jump curve, unless otherwise specified. Hence we are not going to state this condition in subsequent RH problems.},  and satisfies
    \begin{equation}
      Y_+(x) = Y_-(x)
      \begin{pmatrix}
        1 & \frac{1}{\theta x^{\theta-1}} w(x) \\
        0 & 1
      \end{pmatrix},
      \qquad x>0,
    \end{equation}
    where $w$ is given in \eqref{def:weight}.
  \item \label{enu:RHP:original_p:3}
    As $z \to \infty$ in $\compC$, $Y_1$ behaves as $Y_1(z) = z^n + \bigO(z^{n - 1})$.
  \item \label{enu:RHP:original_p:4}
    As $z \to \infty$ in $\mathbb{H}_{\theta}$, $Y_2$ behaves as $Y_2(z) = \bigO(z^{-(n + 1)\theta})$.
  \item \label{enu:RHP:original_p:5}
    As $z \to 0$ in $\compC$, we have
    \begin{equation}
    Y_1(z)=\bigO(1).
    \end{equation}
  \item \label{enu:RHP:original_p:6}
    As $z \to 0$ in $\mathbb{H}_{\theta}$,
    we have
    \begin{equation}
      Y_2(z) =
      \begin{cases}
        \bigO(1), & \text{$\alpha+1-\theta > 0$,} \\
         \bigO(\log z), & \text{$\alpha+1-\theta = 0$,} \\
       \bigO(z^{\alpha+1-\theta}), & \text{$\alpha+1-\theta < 0$.}
      \end{cases}
    \end{equation}
  \item \label{enu:RHP:original_p:7}
    For $x > 0$, we have the boundary condition\footnote{Here and below, for a function $f(z)$ defined on $\halfH$ and $x > 0$, $f(e^{\pm \pi i/\theta} x) = \lim_{z \to e^{\pm \pi i/\theta} x \text{ and $z\in \halfH$}} f(z)$, assuming that the limit exists.} $Y_2(e^{\frac{\pi i}{\theta} }x) = Y_2(e^{-\frac{\pi i}{\theta}}x)$.
  \end{enumerate}
\end{RHP}

The polynomials $q_j$ are also characterized by a similar RH problem. By setting
\begin{equation}\label{eq:defn_of_Cq_n}
\widetilde C q_n(z):=\frac{1}{2\pi i}\int_0^{+\infty} \frac{q_n(x^{\theta})}{x-z}w(x)dx
=\frac{1}{2\pi i}\int_0^{+\infty} \frac{q_n(x^{\theta})}{x-z}x^{\alpha}e^{-nV(x)}dx,\quad z\in \compC \setminus [0,+\infty),
\end{equation}
we have that
\begin{equation} \label{def:Ytilde}
\Y(z) = (q_n(z^{\theta}), \widetilde C q_n(z))
\end{equation}
is the unique solution of the following RH problem; see \cite[Theorem 1.4]{Claeys-Romano14}.

\begin{RHP} \label{RHP:original_q} \hfill
  \begin{enumerate}[label=\emph{(\arabic*)}, ref=(\arabic*)]
  \item
    $\Y = (\Y_1, \Y_2)$ is analytic in $(\mathbb{H}_{\theta}, \compC \setminus [0, +\infty))$.
  \item For $x>0$, we have
    \begin{equation}
      \Y_+(x) =\Y_-(x)
      \begin{pmatrix}
        1 & w(x) \\
        0 & 1
      \end{pmatrix},
    \end{equation}
    where $w$ is given in \eqref{def:weight}.
  \item
    As $z \to \infty$ in $\mathbb{H}_{\theta}$, $\Y_1$ behaves as $\Y_1(z) = z^{n\theta} + \bigO(z^{(n - 1)\theta})$.
  \item
    As $z \to \infty$ in $\compC$, $\Y_2$ behaves as $\Y_2(z) = \bigO(z^{-(n + 1)})$.
  \item
    As $z \to 0$ in $\mathbb{H}_{\theta}$, we have
    \begin{equation}
    \Y_1(z)=\bigO(1).
    \end{equation}
   \item
    As $z \to 0$ in $\compC$, we have
    \begin{equation}
    \Y_2(z) = \left\{
               \begin{array}{ll}
                 \bigO(1), & \hbox{ $\alpha > 0$,} \\
                 \bigO(\log z), & \hbox{ $\alpha = 0$,} \\
                 \bigO(z^{\alpha}), & \hbox{ $\alpha < 0$.}
               \end{array}
             \right.
        \end{equation}

  \item
    For $x > 0$, we have the boundary condition  $\Y_1(e^{\frac{\pi i}{\theta} }x) = \Y_1(e^{-\frac{\pi i}{\theta}}x)$.
  \end{enumerate}
\end{RHP}

The next two sections are then devoted to the asymptotic analysis of the RH problems for $Y$ and $\Y$, respectively.

\section{Asymptotic analysis of the RH problem for $Y$}
\label{sec:asyY}

\subsection{First transformation: $Y \to T$}

With the $g$-functions given in \eqref{def:g} and \eqref{def:tildeg}, we define a vector-valued function
\begin{equation} \label{eq:Y_to_T}
  T(z) = e^{-\frac{n\ell}{2}} Y(z)
  \begin{pmatrix}
    e^{-n g(z)} & 0 \\
    0 & e^{n \g(z)}
  \end{pmatrix}
  e^{\frac{n\ell}{2} \sigma_3} = (Y_1(z) e^{-n g(z)}, Y_2(z) e^{n (\g(z) - \ell)}),
\end{equation}
where the constant $\ell$ is the same as that in \eqref{eq:EL1}. It is then readily seen that $T$ satisfies the following RH problem.
\begin{RHP} \label{RHP:original_T} \hfill
  \begin{enumerate}[label=\emph{(\arabic*)}, ref=(\arabic*)]
  \item \label{enu:RHP:original_T:2}
    $T = (T_1, T_2)$ is analytic in $(\compC \setminus [0, b], \mathbb{H}_{\theta} \setminus [0, \infty))$.
  \item \label{enu:RHP:original_T:3}
    For $x>0$, we have
       \begin{equation}
       T_+(x) = T_-(x) J_T(x),
       \end{equation}
       where
       \begin{equation}
       J_T(x)=
       \begin{pmatrix}
        e^{n(g_-(x) - g_+(x))} & \frac{x^{\alpha+1-\theta}}{\theta} e^{n(g_-(x) + \g_+(x) - V(x) - \ell)} \\
        0 & e^{n(\g_+(x) - \g_-(x))}
      \end{pmatrix}.
    \end{equation}
  \item \label{enu:RHP:original_T:4}
    As $z \to \infty$ in $\compC$, $T_1$ behaves as $T_1(z) = 1 + \bigO(z^{-1})$.
  \item \label{enu:RHP:original_T:5}
    As $z \to \infty$ in $\mathbb{H}_{\theta}$, $T_2$ behaves as $T_2(z) = \bigO(z^{-\theta})$.
  \item \label{enu:RHP:original_T:6}
    As $z \to 0$ in $\compC$ or $\halfH$, $T(z)$ has the same behaviour as $Y(z)$.
  \item \label{enu:RHP:original_T:7}
    As $z \to b$, we have $T_1(z) = \bigO(1)$ and $T_2(z) = \bigO(1)$.
  \item \label{enu:RHP:original_T:8}
    For $x > 0$, we have the boundary condition $T_2(e^{\pi i/\theta}x) = T_2(e^{-\pi i/\theta}x)$.
  \end{enumerate}
\end{RHP}

\subsection{Second transformation: $T \to S$}

The second transformation involves ``opening of the lens''. We start with the following decomposition of $J_T$ for $x\in (0,b)$:
\begin{align}\label{eq:decomp_J_T}
    J_T(x) = {}&
    \begin{pmatrix}
      1 & 0 \\
      \frac{\theta}{x^{\alpha+1-\theta}} e^{-n \phi_-(x)} & 1
    \end{pmatrix}
    \begin{pmatrix}
      0 & \frac{x^{\alpha+1-\theta}/\theta}{e^{n(-g_-(x) - \g_+(x) + V(x) + \ell)}} \\
      -\frac{e^{n(-g_+(x) - \g_-(x) + V(x) + \ell)}}{x^{\alpha+1-\theta}/\theta}  & 0
    \end{pmatrix} \nonumber \\
    &\times
    \begin{pmatrix}
      1 & 0 \\
      \frac{\theta}{x^{\alpha+1-\theta}} e^{-n \phi_+(x)} & 1
    \end{pmatrix}
\nonumber \\
    = {}&
    \begin{pmatrix}
      1 & 0 \\
      \frac{\theta}{x^{\alpha+1-\theta}} e^{-n \phi_-(x)} & 1
    \end{pmatrix}
    \begin{pmatrix}
      0 & \frac{x^{\alpha+1-\theta}}{\theta}  \\
      -\frac{\theta}{x^{\alpha+1-\theta}}  & 0
    \end{pmatrix}
    \begin{pmatrix}
      1 & 0 \\
      \frac{\theta}{x^{\alpha+1-\theta}} e^{-n \phi_+(x)} & 1
    \end{pmatrix},
  \end{align}
where
\begin{equation}\label{def:phi}
  \phi(z): = g(z) + \g(z) - V(z) - \ell,\qquad z\in \halfH \setminus [0,+\infty),
\end{equation}
and we have made use of \eqref{eq:gequal} in the second equality of \eqref{eq:decomp_J_T}.

By opening the lens around $(0,b)$ as shown in Figure \ref{fig:Sigma}, we define two contours, $\Sigma_1$ and $\Sigma_2$, lying in the upper sector $\halfH \cap \compC_+$ and the lower sector $\halfH \cap \compC_-$, respectively, such that both of them go from $0$ to $b$. We leave the specific shapes of $\Sigma_1$ and $\Sigma_2$ to be determined later. We call the region between $\Sigma_1$ and $\Sigma_2$ the lens, and the upper/lower part of the lens is the intersection of the lens with the upper/lower sectors. The second transformation is then defined by
\begin{equation} \label{def:S}
  S(z) =
  \begin{cases}
    T(z), & \text{$z$ outside the lens}, \\
    T(z)
    \begin{pmatrix}
      1 & 0 \\
      \frac{\theta}{z^{\alpha+1-\theta}} e^{-n \phi(z)} & 1
    \end{pmatrix},
    & \text{$z$ in the lower part of the lens}, \\
    T(z)
    \begin{pmatrix}
      1 & 0 \\
      -\frac{\theta}{z^{\alpha+1-\theta}} e^{-n \phi(z)} & 1
    \end{pmatrix},
    & \text{$z$ in the upper part of the lens}.
  \end{cases}
\end{equation}
We have that $S$ satisfies the following RH problem.

\begin{RHP} \label{rhp:S} \hfill
  \begin{enumerate}[label=\emph{(\arabic*)}, ref=(\arabic*)]
  \item \label{enu:rhp:S:1}
    $S = (S_1, S_2)$ is analytic in $(\compC \setminus \Sigma, \mathbb{H}_{\theta} \setminus \Sigma)$, where
   \begin{equation}\label{def:Sigma}
   \Sigma:= [0,+\infty) \cup \Sigma_1 \cup \Sigma_2;
   \end{equation}
   see Figure \ref{fig:Sigma} for an illustration and the orientations.
  \item \label{enu:rhp:S:2}
   For $z \in \Sigma$, we have
    \begin{equation}
      S_+(z) = S_-(z) J_S(z),
      \end{equation}
      where
      \begin{equation} \label{def:Js}
      J_S(z) =
      \begin{cases}
        \begin{pmatrix}
          1 & 0 \\
         \frac{\theta}{z^{\alpha+1-\theta}} e^{-n \phi(z)} & 1
        \end{pmatrix},
        & z \in \Sigma_1 \cup \Sigma_2, \\
        \begin{pmatrix}
          0 & \frac{ z^{\alpha+1-\theta}}{\theta} \\
          -\frac{\theta}{z^{\alpha+1-\theta}} & 0
        \end{pmatrix},
        & \text{$z \in (0, b)$}, \\
        \begin{pmatrix}
          1 & \frac{ z^{\alpha+1-\theta}}{\theta} e^{n \phi(z)} \\
          0 & 1
        \end{pmatrix},
        & z \in (b, +\infty).
      \end{cases}
    \end{equation}
  \item \label{enu:rhp:S:3}
    As $z\to \infty$ in $\compC$ or $\halfH$, $S(z)$ has the same behaviour as $T(z)$.
  \item \label{enu:rhp:S:5}
    As $z \to 0$ in $\compC \setminus \Sigma$, we have
    \begin{equation} \label{eq:asy_S_1_in_lens}
      S_1(z) =
      \begin{cases}
        \bigO(z^{\theta - \alpha - 1}), & \text{$\alpha > \theta - 1$ and $z$ inside the lens,} \\
        \bigO(\log z), & \text{$\alpha = \theta - 1$ and $z$ inside the lens,} \\
        \bigO(1), & \text{$z$ outside the lens or $-1 < \alpha <\theta - 1$.}
      \end{cases}
    \end{equation}
  \item
    As $z \to 0$ in $\halfH \setminus \Sigma$, we have
    \begin{equation}
      S_2(z) =
      \begin{cases}
        \bigO(1), & \alpha > \theta - 1, \\
        \bigO(\log z), & \alpha = \theta - 1, \\
        \bigO(z^{\alpha+1-\theta}), & -1 < \alpha < \theta - 1.
      \end{cases}
    \end{equation}
  \item \label{enu:rhp:S:6}
    As $z \to b$, we have $S_1(z) = \bigO(1)$ and $S_2(z) = \bigO(1)$.
  \item \label{enu:rhp:S:7}
    For $x>0$, we have the boundary condition $S_2(e^{\pi i/\theta}x) = S_2(e^{-\pi i/\theta}x)$.
  \end{enumerate}
\end{RHP}

\begin{figure}[htb]
  \centering
  \includegraphics{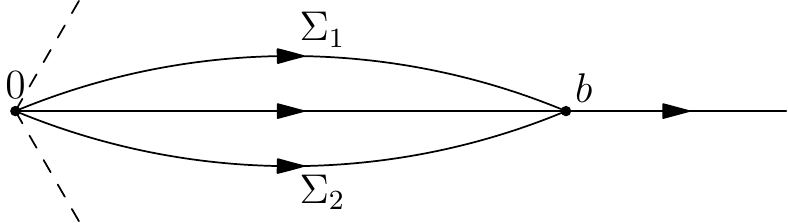}
  \caption{The contour $\Sigma$. ($\theta = 3$ in this instance.)}
  \label{fig:Sigma}
\end{figure}

Although the transformations $Y \to T \to S$ given in \eqref{eq:Y_to_T} and \eqref{def:S} are invertible, the inverse one, however, does not transform the RH problem \ref{rhp:S} for $S$ back to the RH problem \ref{RHP:original_p} for $Y$ directly. Thus, the uniqueness of the solution to RH problem \ref{rhp:S} is not a trivial consequence from that of the RH problem \ref{RHP:original_p}. For later use, we prove the following result.

\begin{prop} \label{prop:unique_S}
  The solution of RH problem \ref{rhp:S} is unique, even if item \ref{enu:rhp:S:6} therein is replaced by the weaker condition that $S_1(z) = \bigO((z-b)^{-1/4})$ and $S_2(z) = \bigO((z-b)^{-1/4})$ as $z \to b$.
\end{prop}
\begin{proof}
  Let $\Shat(z) = (\Shat_1(z), \Shat_2(z))$ be a solution of RH problem \ref{rhp:S} with a weaker version of item \ref{enu:rhp:S:6} that reads $S_1(z) = \bigO((z-b)^{-1/4})$ and $S_2(z) = \bigO((z-b)^{-1/4})$ as $z \to b$. By reversing the transformations $Y \to T \to S$, we obtain a vector-valued function $\Yhat(z) = (\Yhat_1(z), \Yhat_2(z))$ analytic in $(\compC \setminus \{ 0, b \}, \mathbb{H}_{\theta} \setminus [0, \infty))$, which satisfies items \ref{enu:RHP:original_p:3}, \ref{enu:RHP:original_p:4},  \ref{enu:RHP:original_p:6} and \ref{enu:RHP:original_p:7} of RH problem \ref{RHP:original_p}. As $z\to 0$ in $\compC$, $\Yhat_1(z)$ has the same local behaviour as $Y_1(z)$ given in item \ref{enu:RHP:original_p:5} of RH problem \ref{RHP:original_p} provided $z \to 0$ outside the lens. If $z \to 0$ inside the lens, we can only obtain that $\Yhat_1(z)$ has the same asymptotic behaviour as $S_1(z)$ given in \eqref{eq:asy_S_1_in_lens}, which is weaker than item \ref{enu:RHP:original_p:5} of RH problem \ref{RHP:original_p}. Also as $z \to b$, we do not have the continuous behaviour of $Y_1(z)$ and $Y_2(z)$ as given in item \ref{enu:RHP:original_p:1} of RH problem \ref{RHP:original_p}, but only have that $\Yhat_1(z) = \bigO((z - b)^{-1/4})$ and $\Yhat_2(z) = \bigO((z - b)^{-1/4})$. On the other hand, $\Yhat$ still has continuous boundary values when approaching $(0, b) \cup (b, +\infty)$ from the above and the below. We note that the jump condition for $Y$ in item \ref{enu:RHP:original_p:2} of the RH problem \ref{RHP:original_p} is satisfied by $\Yhat(z)$ on $(0, b) \cup (b, +\infty)$ instead of on $(0, +\infty)$.

  A key observation now is that $\Yhat_1(z)$, which agrees with $\Shat_1(z) e^{ng(z)}$ outside the lens, has a polynomial growth at $\infty$. Since it is not hard to see that both $0$ and $b$ are removable singular points for $\Yhat_1(z)$, it then follows from Liouville's theorem that $\Yhat_1(z)$ is a monic polynomial of degree $n$. Moreover, the jump condition for $\Yhat(z)$ together with the boundary conditions that $\Yhat_2(z)$ satisfied on $(0, b) \cup (b, +\infty)$, at $z=\infty$, and at $z=b$ imply that $\Yhat_1(z)$ is the monic polynomial in $z$ satisfying the biorthogonal condition \eqref{eq:pqbioOP}, and $\Yhat_2(z)$ is the Cauchy transform of $\Yhat_1(z)$. The uniqueness of the biorthogonal polynomials in \eqref{eq:pqbioOP} shows that $\Yhat(z)$ is the unique solution to RH problem \ref{RHP:original_p}. Since the transforms $Y \to T \to S$ are invertible, we conclude that the solution of RH problem \ref{rhp:S} is unique, even if item \ref{enu:rhp:S:6} therein is weakened as stated in the proposition.
\end{proof}

\subsection{Construction of the global parametrix}\label{sec:global}

By \eqref{def:phi}, \eqref{eq:gequal} and \eqref{eq:psiandg}, we have, for $x\in (0,b)$,
\begin{equation}\label{eq:phiderivative}
\phi_{\pm}'(x)=g_{\pm}'(x)+\g_{\pm}'(x)-V'(x)=g_{\pm}'(x)-g_{\mp}'(x)= \mp 2\pi i \psi(x).
\end{equation}
Note that since $\psi(x)>0$ on $(0,b)$, it then follows from the Cauchy-Riemann conditions that
\begin{equation} \label{eq:ineq_for_phi}
  \Re \phi(z)=\Re ( g(z)+\g(z)-V(z)-\ell )>0
\end{equation}
for $z$ in a small neighbourhood around $(0,b)$. Here we give the first requirement for the shapes of $\Sigma_1$ and $\Sigma_2$: they should be close to the interval $(0, b)$ enough such that the inequality \eqref{eq:ineq_for_phi} holds on them. (See Sections \ref{sec:Pb} and \ref{subsec:P0} below for more conditions of the shapes.) This, together with \eqref{eq:gequal2}, implies that all the jump matrices in \eqref{def:Js} tend to the identity matrix $I$ for large $n$, except the one on the interval $(0,b)$. This leads us to consider the following RH problem.
\begin{RHP} \label{rhp:global} \hfill
  \begin{enumerate}[label=\emph{(\arabic*)}, ref=(\arabic*)]
  \item
    $P^{(\infty)} = (P_1^{(\infty)}, P_2^{(\infty)})$ is analytic in $(\compC \setminus [0,b], \mathbb{H}_{\theta} \setminus [0,b])$.
  \item
    For $x \in (0,b)$, we have
    \begin{equation} \label{eq:P^infty_jump}
      P^{(\infty)}_+(x) = P^{(\infty)}_-(x)
      \begin{pmatrix}
          0 & \frac{x^{\alpha+1-\theta}}{\theta} \\
          -\frac{\theta}{x^{\alpha+1-\theta}} & 0
        \end{pmatrix}.
      \end{equation}
    \item
     As $z\to\infty$ in $\compC$, $P_1^{(\infty)}$ behaves as $P_1^{(\infty)}(z)=1+\bigO(z^{-1})$.
  \item
    As $ z\to \infty $ in $\mathbb{H}_\theta$, $P_2^{(\infty)}$ behaves as $P_2^{(\infty)}(z)=\bigO(z^{-\theta})$.
  \item \label{enu:pinftyboundary}
    For $x>0$, we have the boundary condition $P_2(e^{\pi i/\theta }x) = P_2(e^{-\pi i/\theta}x)$.
  \end{enumerate}
\end{RHP}

To construct a solution to the above RH problem, we follow the idea in \cite{Claeys-Wang11} to map the RH problem \ref{rhp:global} for $P^{(\infty)}$ to a scalar RH problem which can be solved explicitly. More precisely, using the function $J_c(s)$ defined in \eqref{def:Jcs}, we set
\begin{equation}
  \tP(s) :=
  \begin{cases}
    P_1^{(\infty)}(J_c(s)), & s\in \mathbb{C}\setminus \overline{D}, \\
    P_2^{(\infty)}(J_c(s)), & s\in D \setminus [-1,0],
  \end{cases}
\end{equation}
where $D$ is the region bounded by the curves $\gamma_1$ and $\gamma_2$, as shown in Figure \ref{fig:mapJ}. Since the function $P_2^{(\infty)}$ satisfies the boundary condition indicated in item \ref{enu:pinftyboundary} of RH problem \ref{rhp:global} and the function $J_c$ maps the upper/lower side of $(-1,0)$ to the boundary of  $\mathbb{H}_\theta$, the function $\tP$ is then well-defined on $(-1,0)$ by continuation. It is straightforward to check that $\tP$ satisfies the following RH problem.

\begin{RHP} \label{RHP:F} \hfill
  \begin{enumerate}[label=\emph{(\arabic*)}, ref=(\arabic*)]
  \item
    $\tP$ is analytic in $\mathbb{C}\setminus (\gamma_1 \cup \gamma_2)$.
  \item
    For $s \in \gamma_1 \cup \gamma_2$, we have
    \begin{equation}
      \tP_+(s) = \tP_-(s) J_{\tP}(s),
      \end{equation}
      where
      \begin{equation}
      J_{\tP}(s) =
      \begin{cases}
        -\theta c^{\theta-\alpha-1} (s + 1)^{\theta-\alpha-1} \left(\frac{s+1}{s}\right)^{\frac{\theta-\alpha-1}{\theta}}, & \text{$s\in\gamma_1$}, \\
        \frac{c^{\alpha+1-\theta}}{\theta} (s + 1)^{\alpha+1-\theta} \left(\frac{s+1}{s}\right)^{\frac{\alpha+1-\theta}{\theta}}, & \hbox{$s\in\gamma_2$},
      \end{cases}
    \end{equation}
    with $c$ being the constant given in \eqref{def:Jcs}.
  \item
    As $s\to\infty$, we have $\tP(s) = 1 + \bigO(s^{-1})$.
   \item
     As $s \to 0 $, we have $\tP(s) = \bigO(s)$.
  \end{enumerate}
\end{RHP}

A solution $\tP$ to the above RH problem is explicitly given by
\begin{equation} \label{eq:Fs}
  \tP(s)=
  \begin{cases}
    \frac{s}{\sqrt{(s + 1)(s - s_b)}}  \left( \frac{s+1}{s} \right)^{\frac{\theta-\alpha-1}{\theta}}, & s \in \compC \setminus \overline{D}, \\
    \frac{c^{\alpha+1-\theta}s(s+1)^{\alpha+1-\theta}}{\theta\sqrt{(s+1)(s-s_b)}}, & s\in D,
  \end{cases}
\end{equation}
where $s_b=1/\theta$, the branch cuts of $\sqrt{(s+1)(s - s_b)}$, $\left(\frac{s+1}{s}\right)^{\frac{\theta-\alpha-1}{\theta}}$ and $(s+1)^{\alpha+1-\theta}$ are taken along $\gamma_1$, $[-1,0]$ and $(-\infty,-1]$, respectively. We note that the RH problem \ref{RHP:F} for $\tP$ may have other solutions, since we do not specify the local behaviours of $\tP$ as $z \to -1$ or $z \to s_b$. Here and below, we only consider the solution \eqref{eq:Fs}.

As a consequence, one solution to the RH problem \ref{rhp:global} for $P^{(\infty)}$ is given by
\begin{align}
  P_1^{(\infty)}(z) = {}& \tP(I_1(z)), & z \in {}& \mathbb{C}\setminus [0,b], \label{eq:P1} \\
  P_2^{(\infty)}(z) = {}& \tP(I_2(z)), & z \in {}& \mathbb{H}_\theta \setminus [0,b], \label{eq:P2}
\end{align}
where $\tP$ is given by \eqref{eq:Fs}, and $I_1$ and $I_2$ are the inverses of two branches of the mapping $J_c$ satisfying
\begin{align}
  I_1(J_c(s)) = {}& s, & s \in {}& \mathbb{C}\setminus \overline{D}, \label{eq:inverse1}
  \\
  I_2(J_c(s)) = {}&s, & s \in {}& D \setminus [-1,0]. \label{eq:inverse2}
\end{align}
Similar to the RH problem \ref{RHP:F}, the RH problem \ref{rhp:global} has more than one solutions too, since we do not specify the behaviours of $P^{(\infty)}$ as $z \to 0$ or $z \to b$. For our purpose, we only consider the solution given by \eqref{eq:Fs}--\eqref{eq:P2} and its local behaviour is collected in the following Proposition.

\begin{prop}\label{prop:P_infty_asy}
  With $P^{(\infty)}(z)=(P^{(\infty)}_1(z), P^{(\infty)}_2(z))$ defined in \eqref{eq:Fs}--\eqref{eq:P2}, we have, as $z\to 0$,
  \begin{align}
    P^{(\infty)}_1(z) = {}&
                            \begin{cases}
                              \sqrt{\frac{\theta}{1+\theta}}c^{\frac{2(\alpha+1)-\theta}{2(1+\theta)}}e^{\frac{2(\alpha+1)-\theta}{2(1+\theta)}\pi i}z^{\frac{\theta-2(\alpha+1)}{2(1+\theta)}}(1 + \bigO(z^{\frac{\theta}{1+\theta}})), & \arg z \in (0, \pi),
                              \\
                              \sqrt{\frac{\theta}{1+\theta}}c^{\frac{2(\alpha+1)-\theta}{2(1+\theta)}}e^{\frac{\theta-2(\alpha+1)}{2(1+\theta)}\pi i}z^{\frac{\theta-2(\alpha+1)}{2(1+\theta)}}(1 + \bigO(z^{\frac{\theta}{1+\theta}})), & \arg z \in (-\pi, 0),
                            \end{cases} \label{eq:asy_P^infty_1} \\
    P^{(\infty)}_2(z) = {}&
                            \begin{cases}
                              \frac{c^{\frac{2(\alpha+1)-\theta}{2(1+\theta)}}}{\sqrt{\theta(1+\theta)}}e^{\frac{\theta-2(\alpha+1)}{2(1+\theta)}\pi i} z^{\frac{(\alpha+\frac12 -\theta)\theta}{1+\theta}} (1 + \bigO(z^{\frac{\theta}{1+\theta}})), & \arg z \in (0, \frac{\pi}{\theta}), \\
                              \frac{c^{\frac{2(\alpha+1)-\theta}{2(1+\theta)}}}{\sqrt{\theta(1+\theta)}}e^{\frac{2\alpha-3\theta}{2(1+\theta)}\pi i} z^{\frac{(\alpha+\frac12 -\theta)\theta}{1+\theta}} (1 + \bigO(z^{\frac{\theta}{1+\theta}})), & \arg z \in (-\frac{\pi}{\theta}, 0),
                            \end{cases} \label{eq:asy_P^infty_2}
  \end{align}
  and as $z\to b$,
  \begin{equation} \label{eq:asy_P^infty_b}
    P^{(\infty)}_i(z)=\bigO((z-b)^{-\frac14}), \qquad i=1,2.
  \end{equation}
\end{prop}
\begin{proof}
  From the definition of $I_i(z)$, $i=1,2$, it is easily seen that, as $z \to 0$,
  \begin{align}
    I_1(z) = {}& -1 +
                 \begin{cases}
                   c^{-\frac{\theta}{1+\theta}}e^{\frac{\pi }{1+\theta}i}  z^{\frac{\theta}{1+\theta}} (1 + \bigO(z^{\frac{\theta}{1+\theta}})),  & \arg z \in (0, \pi),
                   \\
                   c^{-\frac{\theta}{1+\theta}}e^{-\frac{\pi }{1+\theta}i}  z^{\frac{\theta}{1+\theta}} (1 + \bigO(z^{\frac{\theta}{1+\theta}})), & \arg z \in (-\pi, 0),
                 \end{cases} \label{eq:I1zero} \\
    I_2(z) = {}& -1 +
                 \begin{cases}
                   c^{-\frac{\theta}{1+\theta}}e^{-\frac{\pi }{1+\theta}i}  z^{\frac{\theta}{1+\theta}} (1 + \bigO(z^{\frac{\theta}{1+\theta}})), & \arg z \in (0, \frac{\pi}{\theta}), \\
                   c^{-\frac{\theta}{1+\theta}}e^{\frac{\pi }{1+\theta}i}  z^{\frac{\theta}{1+\theta}} (1 + \bigO(z^{\frac{\theta}{1+\theta}})) ,  & \arg z \in (-\frac{\pi}{\theta}, 0),
                 \end{cases} \label{eq:I2zero}
  \end{align}
  and as $z\to b$, $I_i(z)=s_b + \bigO((z-b)^{1/2})$ for $i=1,2$.
  In view of \eqref{eq:Fs}, we further have
  \begin{equation}
    \tP(s) =
    \begin{cases}
      \bigO(\lvert s - s_b \rvert^{-\frac12}), & s\to s_b, \\
      -i\frac{c^{\alpha+1-\theta}}{\sqrt{\theta(1+\theta)}}(s+1)^{\alpha+\frac12-\theta}(1+\bigO(s+1)), & \text{$s\to -1$ from $D$}, \\
      \sqrt{\frac {\theta}{1+\theta}}(s+1)^{\frac{\theta-2(\alpha+1)}{2\theta}}e^{\frac{2(\alpha+1)-\theta}{2\theta}\pi i}(1+\bigO(s+1)), & \text{$s\to -1$ from $\compC_+ \setminus \overline{D}$}, \\
      \sqrt{\frac {\theta}{1+\theta}} ((s+1)e^{\pi i})^{\frac{\theta-2(\alpha+1)}{2\theta}}(1+\bigO(s+1)), & \text{$s\to -1$ from $\compC_- \setminus \overline{D}$}.
    \end{cases}
  \end{equation}
  A combination of the above formulas with \eqref{eq:Fs}--\eqref{eq:P2} gives us \eqref{eq:asy_P^infty_1}--\eqref{eq:asy_P^infty_b}.
\end{proof}

\subsection{Third transformation: $S \to Q$}

Noting that $P_1^{(\infty)}(z)\neq 0$ for $z \in \compC \setminus [0,b]$ and $P_2^{(\infty)}(z) \neq 0 $ for $z\in \mathbb{H}_\theta \setminus [0,b]$, we define the third transformation by
\begin{equation}\label{def:thirdtransform}
  Q(z) = (Q_1(z),Q_2(z))=\left(\frac{S_1(z)}{P^{(\infty)}_1(z)},\frac{S_2(z)}{P^{(\infty)}_2(z)} \right).
\end{equation}
In view of the RH problems \ref{rhp:S} and \ref{rhp:global} for $S$ and $P^{(\infty)}$, it is then easily seen that, with the aid of Proposition \ref{prop:P_infty_asy}, $Q$ satisfies the following RH problem.
\begin{RHP} \label{RHP:Svar} \hfill
  \begin{enumerate}[label=\emph{(\arabic*)}, ref=(\arabic*)]
  \item \label{enu:RHP:Svar:1}
    $Q = (Q_1, Q_2)$ is analytic in $(\compC \setminus \Sigma, \mathbb{H}_\theta \setminus \Sigma)$,  where the contour $\Sigma$ is defined in \eqref{def:Sigma}.
  \item \label{enu:RHP:Svar:2}
    For $z \in \Sigma$, we have
    \begin{equation} \label{def:Jcals}
      Q_+(z) = Q_-(z) J_Q(z),
      \end{equation}
      where
      \begin{equation} \label{def:JQ}
      J_Q(z) =
      \begin{cases}
        \begin{pmatrix}
          1 & 0 \\
          \frac{\theta P^{(\infty)}_2(z)}{z^{\alpha+1-\theta}P^{(\infty)}_1(z)}e^{-n \phi(z)} & 1
        \end{pmatrix},
        & z \in \Sigma_1 \cup \Sigma_2, \\
        \begin{pmatrix}
          0 & 1 \\
          1 & 0
        \end{pmatrix},
        & z \in (0, b), \\
        \begin{pmatrix}
          1 & \frac{z^{\alpha+1-\theta}P^{(\infty)}_1(z)}{\theta P^{(\infty)}_2(z)}e^{n \phi(z)}  \\
          0 & 1
        \end{pmatrix},
        & z \in (b, +\infty).
      \end{cases}
    \end{equation}

  \item \label{enu:RHP:Svar:3}
    As $z \to \infty$ in $\compC$, $Q_1$ behaves as $Q_1(z)=1+\bigO(z^{-1})$.

  \item \label{enu:RHP:Svar:4}
    As $z \to \infty$ in $\mathbb{H}_\theta $, $Q_2$ behaves as $Q_2(z)=\bigO(1)$.

  \item
  \label{enu:RHP:Svar:5}
    As $z \to 0$ in $\compC \setminus \Sigma$, we have
    \begin{equation}
      Q_1(z) =
      \begin{cases}
        \bigO ( z^{\frac{\theta(\theta - \alpha - 1/2)}{1 + \theta}} ), & \text{$\alpha > \theta - 1$ and $z$ inside the lens}, \\
        \bigO ( z^{\frac{\theta}{2(1+\theta)}} \log z ), & \text{$\alpha = \theta - 1$ and $z$ inside the lens}, \\
        \bigO ( z^{\frac{\alpha + 1 - \theta/2}{1+\theta}} ), & \text{$z$ outside the lens or $-1 < \alpha < \theta - 1$}.
      \end{cases}
    \end{equation}

  \item
    As $z \to 0$ in $\halfH \setminus \Sigma$, we have
    \begin{equation}
      Q_2(z)=
      \begin{cases}
        \bigO ( z^{\frac{\theta(\theta-\alpha-1/2)}{1+\theta}} ), & \alpha > \theta - 1, \\
        \bigO ( z^{\frac{\theta}{2(1+\theta)}} \log z ), & \alpha = \theta - 1, \\
        \bigO (z^{\frac{\alpha+1-\theta/2}{1+\theta}} ), & -1 < \alpha < \theta - 1.
      \end{cases}
    \end{equation}
  \item
    As $z \to b$, we have $Q_1(z) = \bigO((z - b)^{1/4})$ and $Q_2(z) = \bigO((z - b)^{1/4})$.
  \item \label{enu:RHP:Svar:7}
    For $x>0$, we have the boundary condition $Q_2(e^{\pi i/\theta}x) = Q_2(e^{-\pi i/\theta}x)$.
  \end{enumerate}
\end{RHP}

\subsection{Local parametrix around $b$} \label{sec:Pb}
Since the convergence of the jump matrix $J_S(z)$ in \eqref{def:Js} to the identity matrix on the lens is not uniform as $n\to \infty$ near the ending points $0$ and $b$, we need to construct local parametrices near these points, which serve as local approximations to the solution of the RH problem \ref{RHP:Svar} for $Q$. Near the right ending point $b$, this parametrix can be built with the aid of the well-known Airy parametrix \cite{Deift-Kriecherbauer-McLaughlin-Venakides-Zhou99,Deift99} $\Psi^{(\Ai)}$; see Appendix \ref{app:Airy} below for the definition.

To this end, we note that the regularity assumption on the potential $V$ implies that (see \eqref{eq:phiderivative} and \eqref{eq:localb}), as $z \to b$,
\begin{equation}
\phi(z)=-\frac{4\pi}{3}d_2(z-b)^{\frac32}+\bigO(|z-b|^{\frac52}),
\end{equation}
$\phi(z)/(z - b)^{3/2}$ is analytic at $b$, and then
\begin{equation}\label{def:fb}
f_b(z):=\left(-\frac34 \phi(z)\right)^{\frac23}
\end{equation}
is a conformal mapping in a neighbourhood $D(b, \epsilon)$ around $b$ satisfying $f_b(b)=0$ and $f_b'(b)>0$, where $\epsilon$ is a small positive constant and $D(b, \epsilon)$ is defined in \eqref{eq:size_U_b}. Moreover, we also choose the shape of contour $\Sigma$ so that the image of $\Sigma \cap D(b, \epsilon)$ under the mapping $f_b$ coincides with the jump contour
\begin{equation}\label{def:AiryContour}
  \Gamma_{\Ai}:=e^{-\frac{2\pi i}{3}}[0,+\infty) \cup \mathbb{R} \cup e^{\frac{2\pi i}{3}}[0,+\infty)
\end{equation}
of the RH problem \ref{rhp:Ai} for $\Psi^{(\Ai)}$ restricted in a neighbourhood of the origin; see Figure \ref{fig:Sigma_in_U_b} for an illustration.
\begin{figure}[htb]
  \centering
  \includegraphics{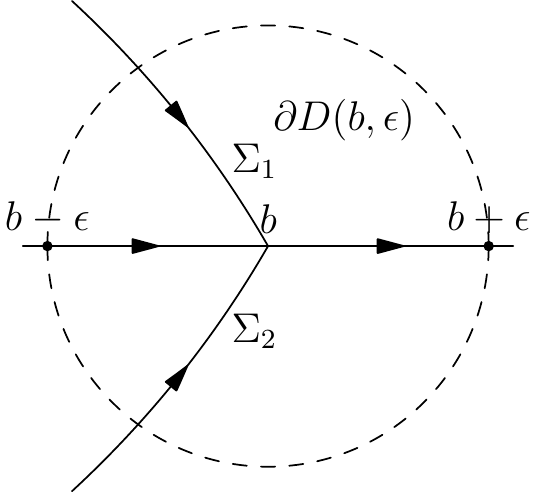}
  \caption[$D(b,\epsilon)$ and $\Sigma \cap D(b,\epsilon)$.]{Schematic figure of $D(b,\epsilon)$ and $\Sigma \cap D(b,\epsilon)$.
}
  \label{fig:Sigma_in_U_b}
\end{figure}

Let
\begin{equation}\label{def:gib}
  g^{(b)}_1(z) = \frac{z^{(\theta - \alpha - 1)/2}}{P_1^{(\infty)}(z)}, \qquad g^{(b)}_2(z) = \frac{z^{(\alpha + 1 - \theta)/2}}{\theta P_2^{(\infty)}(z)}.
\end{equation}
We then define
\begin{equation}\label{def:hatPb}
   \mathsf{P}^{(b)}(z) := \Psi^{(\Ai)}(n^{\frac23}f_b(z))
  \begin{pmatrix}
    e^{-\frac{n}{2} \phi(z)} g^{(b)}_1(z) & 0
    \\
    0 & e^{\frac{n}{2} \phi(z)} g^{(b)}_2(z)
  \end{pmatrix},  \qquad
  z \in D(b,\epsilon) \setminus \Sigma.
\end{equation}
From \eqref{def:gib}, it is easily seen that $g^{(b)}_1(z)$ and $g^{(b)}_2(z)$ are defined in $D(b,\epsilon) \setminus [b - \epsilon, b]$, and satisfy
\begin{equation} \label{eq:gbijump}
  \begin{aligned}
    g^{(b)}_{1, +}(x) = {}& -g^{(b)}_{2, -}(x),  & g^{(b)}_{2, +}(x) = {}& g^{(b)}_{1, -}(x), && \text{$x\in(b-\epsilon, b)$}, \\
    g^{(b)}_1(z) = {}& \bigO((z - b)^{\frac{1}{4}}),  & g^{(b)}_2(z) = {}& \bigO((z - b)^{\frac{1}{4}}), && \text{$z \to b$},
\end{aligned}
\end{equation}
by \eqref{eq:P^infty_jump} and \eqref{eq:asy_P^infty_b}. This, together with \eqref{eq:gequal} and the RH problem \ref{rhp:Ai} for $\Psi^{(\Ai)}$, implies the following RH problem for $\mathsf{P}^{(b)}$.
\begin{RHP} \label{RHP:Pbmodel} \hfill
  \begin{enumerate}[label=\emph{(\arabic*)}, ref=(\arabic*)]
  \item
    $ \mathsf{P}^{(b)}(z)$ is analytic for $z \in D(b,\epsilon)\setminus \Sigma$.
  \item
    For $z \in \Sigma \cap D(b,\epsilon)$, we have
    \begin{equation}\label{eq:jsfPb}
      \mathsf{P}^{(b)}_+(z) = \mathsf{P}^{(b)}_-(z) J_Q(z),
    \end{equation}
    where $J_Q(z)$ is defined in \eqref{def:JQ}.
  \item
    As $z \to b$, we have $\mathsf{P}^{(b)}(z) = \bigO((z-b)^{1/4})$ and $\mathsf{P}^{(b)}(z)^{-1} = \bigO((z-b)^{-1/4})$, which are understood in an entry-wise manner.
  \item
    For $z \in \partial D(b,\epsilon)$, we have, as $n\to \infty$,
     \begin{multline}
      \frac{1}{\sqrt{2}}
      \begin{pmatrix}
        g^{(b)}_1(z) & 0  \\
        0 & g^{(b)}_2(z)
      \end{pmatrix}^{-1}
      e^{\frac{\pi i}{4} \sigma_3}
      \begin{pmatrix}
        1 & -1 \\
        1 & 1
      \end{pmatrix}
      \begin{pmatrix}
        n^{\frac16} f_b(z)^{\frac14} & 0 \\
        0 & n^{-\frac16} f_b(z)^{-\frac14}
      \end{pmatrix}
     \mathsf{P}^{(b)}(z)\\
      = I + \bigO(n^{-1}).
    \end{multline}
  \end{enumerate}
\end{RHP}

We now define a $2 \times 2$ matrix-valued function
\begin{equation}\label{def:Pb}
  P^{(b)}(z) :=E^{(b)}(z) \mathsf{P}^{(b)}(z), \qquad  z \in D(b,\epsilon) \setminus \Sigma,
\end{equation}
where $\mathsf{P}^{(b)}$ is given in \eqref{def:hatPb} and
\begin{equation}\label{def:Eb}
  E^{(b)}(z)=\frac{1}{\sqrt{2}}
  \begin{pmatrix}
    g^{(b)}_1(z) & 0  \\
    0 & g^{(b)}_2(z)
  \end{pmatrix}^{-1}
  e^{\frac{\pi i}{4} \sigma_3}
  \begin{pmatrix}
    1 & -1 \\
    1 & 1
  \end{pmatrix}
  \begin{pmatrix}
    n^{\frac16} f_b(z)^{\frac14} & 0 \\
    0 & n^{-\frac16} f_b(z)^{-\frac14}
  \end{pmatrix}
\end{equation}
with $\sigma_3
= (\begin{smallmatrix}1 & 0
  \\
  0 & -1
\end{smallmatrix})
$ being the Pauli matrix. From \eqref{def:Eb}, it is readily seen that, for $x \in (b-\epsilon,b)$,
\begin{equation} \label{eq:jump_E^b}
    E^{(b)}_+(x)E^{(b)}_-(x)^{-1}  = \begin{pmatrix}
      g^{(b)}_{1,+}(x) & 0  \\
      0 & g^{(b)}_{2,+}(x)
    \end{pmatrix}^{-1}
    \begin{pmatrix}
      0 & -1 \\
      1 & 0
    \end{pmatrix}
    \begin{pmatrix}
      g^{(b)}_{1,-}(x) & 0  \\
      0 & g^{(b)}_{2,-}(x)
    \end{pmatrix}=\begin{pmatrix}
      0 & 1
      \\
      1 & 0
    \end{pmatrix},
\end{equation}
where we have made use of \eqref{eq:gbijump} in the last step.

A combination of the RH problem \ref{RHP:Pbmodel} for $\mathsf{P}^{(b)}$ and \eqref{eq:jump_E^b} shows that $P^{(b)}$ defined in \eqref{def:Pb} satisfies the following RH problem.
\begin{RHP} \label{prop:Pb}
\hfill
\begin{enumerate}[label=\emph{(\arabic*)}, ref=(\arabic*)]
\item
  $ P^{(b)}(z)$ is analytic in $D(b,\epsilon) \setminus \Sigma$.
\item
  For $z \in \Sigma \cap D(b,\epsilon)$, we have
  \begin{equation}
    P^{(b)}_+(z) =
    \begin{cases}
      P^{(b)}_-(z) J_Q(z), & z \in \text{$ \Sigma \cap D(b,\epsilon) \setminus [b-\epsilon, b]$}, \\
      \begin{pmatrix}
        0 & 1 \\
        1 & 0
      \end{pmatrix}
      P^{(b)}_-(z) J_Q(z), & z \in (b - \epsilon, b),
    \end{cases}
  \end{equation}
  where $J_{Q}$ is defined in \eqref{def:JQ}.
\item
  As $z \to b$, we have $P^{(b)}(z) = \bigO((z-b)^{-1/4})$ and $P^{(b)}(z)^{-1} = \bigO((z-b)^{-1/4})$, which are understood in an entry-wise manner.
\item \label{enu:Pbbnd}
  For $z$ on the boundary $\partial D(b,\epsilon)$ of $D(b,\epsilon)$, we have, as $n\to\infty$, $P^{(b)}(z) = I + \bigO(n^{-1})$.
\end{enumerate}
\end{RHP}

At last, we define a vector-valued function $V^{(b)}$ by
\begin{equation} \label{eq:defn_V^(b)}
  V^{(b)}(z) = Q(z) P^{(b)}(z)^{-1}, \quad z \in D(b,\epsilon) \setminus \Sigma,
\end{equation}
where $Q(z)$ is defined in \eqref{def:thirdtransform}. It is then easily seen that $V^{(b)}$ satisfies the following RH problem.
\begin{RHP}\label{RHP:Vb} \hfill
  \begin{enumerate}[label=\emph{(\arabic*)}, ref=(\arabic*)]
  \item
    $V^{(b)} = (V^{(b)}_1, V^{(b)}_2)$ is analytic in $D(b,\epsilon)\setminus [b-\epsilon,b]$.
  \item
    For $z \in (b-\epsilon,b)$, we have
    \begin{equation}
      V^{(b)}_+(x)=V^{(b)}_-(x)\begin{pmatrix}
0 & 1
\\
1 & 0
\end{pmatrix}.
      \end{equation}

  \item
    As $z \to b$, we have $V^{(b)}_1(z) = \bigO(1)$ and $V^{(b)}_2(z) = \bigO(1)$.

  \item
    For $z \in \partial D(b,\epsilon)$, we have, as $n\to \infty$, $V^{(b)}(z)=Q(z)(I+\bigO(n^{-1}))$.
  \end{enumerate}
\end{RHP}

\subsection{Local parametrix around $0$}\label{subsec:P0}
Near the left ending point $0$ of the support of the equilibrium measure, i.e., the hard edge in the context of random matrix theory, we still want to construct a matrix-valued function as the local parametrix. The roadmap of our construction is parallel to that for the local parametrix around $b$. We are going to have $\mathsf{P}^{(0)}$, $P^{(0)}$ and $V^{(0)}$ that are analogous to $\mathsf{P}^{(b)}$, $P^{(b)}$ and $V^{(b)}$. It comes out that the construction here is much more involved and complicated. One difficulty is that there does not exist an open disk centred at the origin lying in $\compC \cap \halfH$. Moreover, instead of using the well established Airy parametrix $\Psi^{(\Ai)}$, we have to build a Meijer G-parametrix $\Psi^{(\Mei)}$ from scratch, which will be the technical heart of this part. Throughout this subsection, we emphasize that $\theta \in \mathbb{N}$.

\subsubsection*{A local conversion of the RH problem for $Q$}
As the initial step toward the construction, we convert the RH problem \ref{RHP:Svar} for $Q$ near the origin to an equivalent one but defined in a small open disk centred at $0$. To this end, let $r > 0$ be a small enough constant, and later we will actually take $r=r_n$ to be shrinking with $n$. It is also assumed that the contours $\Sigma_1$ and $\Sigma_2$ satisfy the following requirements in the open disk $D(0,r^{1/\theta})$: $\Sigma_1\cap D(0,r^{1/\theta})$ overlaps with the ray $\{ z\in \compC \mid \arg z = \pi/(2\theta) \}$ and $\Sigma_2 \cap D(0,r^{1/\theta})$ overlaps with the ray $\{ z\in \compC \mid \arg z = -\pi/(2\theta) \}$. We then define $\theta+1$ functions $U_0(z), U_1(z), \ldots, U_{\theta}(z)$ for $z\in D(0,r)$ with some rays removed, as follows:
\begin{align}
  U_0(s^{\theta}) &=  Q_2(s), \quad \arg s \in (0, \frac{\pi}{2\theta}) \cup (\frac{\pi}{2\theta}, \frac{\pi}{\theta}] \cup (-\frac{\pi}{2\theta}, 0) \cup (-\frac{\pi}{\theta}, -\frac{\pi}{2\theta}), \label{def:U0} \\
  U_1(s^{\theta}) &=  Q_1(s), \quad \arg s \in (0, \frac{\pi}{2\theta}) \cup (\frac{\pi}{2\theta}, \frac{\pi}{\theta}] \cup (-\frac{\pi}{2\theta}, 0) \cup (-\frac{\pi}{\theta}, -\frac{\pi}{2\theta}), \\
  U_k(s^{\theta}) &=  Q_1(s), \quad \arg s \in (\frac{(2k - 3)\pi}{\theta}, \frac{(2k - 1)\pi}{\theta}), \quad k = 2, \dotsc, \theta,\label{def:Uk}
\end{align}
or equivalently,
\begin{align}
  U_0(z) &=  Q_2(z^{\frac1\theta}), && z\in D(0,r)\setminus \{(-r,r)\cup(-ir,ir)\}, \label{def:U0z}
  \\
  U_k(z) &=  Q_1(z^{\frac1\theta}e^{\frac{2(k-1)}{\theta}\pi i}), && z \in D(0,r) \setminus \{(-r,r)\cup(-ir,ir)\}, \quad k= 1, 2, \dotsc, \theta,\label{def:Uiz}
\end{align}
where $(Q_1(z), Q_2(z))$ solves the RH problem \ref{RHP:Svar} and we choose the principal branch for $z^{1/\theta}$. As a consequence, we arrive at the following $1\times (\theta+1)$ RH problem.
\begin{RHP}\label{rhp:U} \hfill
  \begin{enumerate}[label=\emph{(\arabic*)}, ref=(\arabic*)]
  \item
    $ U=(U_0,U_1,\ldots,U_\theta)$ is defined and analytic in $D(0,r) \setminus \{(-r,r) \cup (-ir,ir)\}$.
  \item
    For $z\in (-r,r)\cup(-ir,ir)\setminus\{0\}$, we have
    \begin{equation}
      U_+(z) = U_-(z) J_U(z),
    \end{equation}
    where
    \begin{equation} \label{eq:defn_J_U}
      J_U(z) =
      \begin{cases}
        \begin{pmatrix}
          1 & \theta z^{\frac{\theta - 1 - \alpha}{\theta}} e^{-n \phi(z^{1/\theta})} \frac{P^{(\infty)}_2(z^{1/\theta})}{P^{(\infty)}_1(z^{1/\theta})} \\
          0 & 1
        \end{pmatrix}
        \oplus I_{\theta - 1}, & z \in (0,ir)\cup (0,-ir), \\
        \begin{pmatrix}
          0 & 1 \\
          1 & 0
        \end{pmatrix}
        \oplus I_{\theta - 1}, & z \in (0, r), \\
        \Mcyclic, & z \in (-r, 0),
      \end{cases}
    \end{equation}
    with $\Mcyclic$ being defined in \eqref{eq:defn_Mcyclic}, and the orientations of the rays are shown in Figure \ref{fig:jumps-U}.
  \item
    As $z\to 0$ from $D(0,r) \setminus \{(-r,r) \cup (-ir,ir)\}$, we have
    \begin{enumerate}[label=\emph{(\alph*)}, ref=(\arabic*)]
    \item
      \begin{equation}
        U_0(z) =
        \begin{cases}
          \bigO (z^{1 - \frac{\alpha + 3/2}{\theta + 1}}), & \alpha > \theta-1, \\
          \bigO (z^{\frac{1}{2(1+\theta)}} \log z ), & \alpha = \theta - 1, \\
          \bigO (z^{\frac{\alpha + 1}{\theta} - \frac{\alpha + 3/2}{\theta + 1}}), & -1 < \alpha < \theta - 1,
        \end{cases}
      \end{equation}
    \item
      for $\arg z \in (0, \pi/2) \cup (-\pi/2, 0)$,
      \begin{equation}
        U_1(z) =
        \begin{cases}
          \bigO (z^{1 - \frac{\alpha + 3/2}{\theta + 1}}), & \alpha > \theta - 1 , \\
          \bigO (z^{\frac{1}{2(1+\theta)}} \log z ), & \alpha = \theta-1, \\
          \bigO (z^{\frac{\alpha + 1}{\theta} - \frac{\alpha + 3/2}{\theta + 1}}), & -1 < \alpha < \theta-1,
        \end{cases}
      \end{equation}
    \item
      for $k = 2, \dotsc, \theta$, or $k=1$ and $\arg z \in (\pi/2,\pi)\cup (-\pi,-\pi/2)$,
      \begin{equation}
        U_k(z)=\bigO (z^{\frac{\alpha + 1}{\theta} - \frac{\alpha + 3/2}{\theta + 1}}).
      \end{equation}
    \end{enumerate}
  \end{enumerate}
\end{RHP}
\begin{figure}[htb]
  \begin{minipage}[b]{0.45\linewidth}
    \centering
    \includegraphics{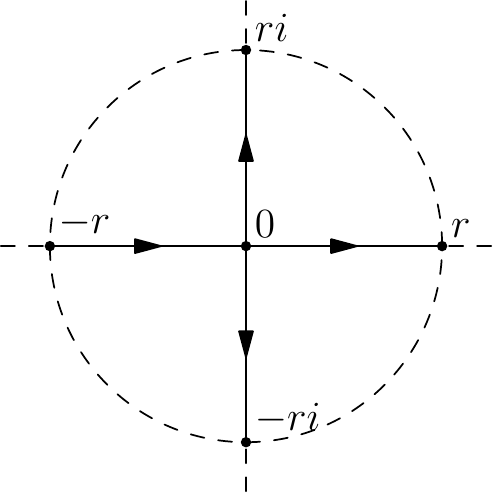}
    \caption{The jump contour for the RH problem \ref{rhp:U} for $U$ and for the RH problem \ref{rhp:tildeU} for $\U$.}
    \label{fig:jumps-U}
  \end{minipage}
  \hspace{\stretch{1}}
  \begin{minipage}[b]{0.45\linewidth}
    \centering
    \includegraphics{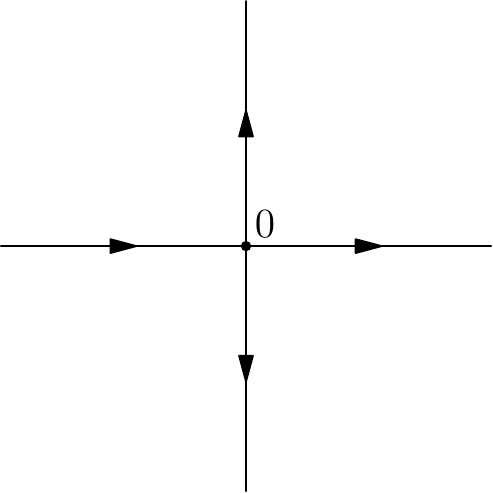}
    \caption{The jump contour for the RH problem \ref{RHP:MeiG} for $\Psi^{(\Mei)}$ and for the RH problem \ref{RHP:tildeMeiG} for $\G$.}
    \label{fig:jumps-MeiG}
  \end{minipage}
\end{figure}

We emphasize that the RH problem for $U$ is defined in $D(0,r)$, and an approximation of $U$ will in turn provide a local approximation of $Q$ near the origin through the relations \eqref{def:U0}--\eqref{def:Uk}. It turns out that the construction will be made with the aid of a model RH problem solved in terms of the Meijer G-functions, which is described next.

\subsubsection*{The Meijer G-parametrix $\Psi^{(\Mei)}$}
Here we define the model parametrix we need in the construction of a local parametrix near the origin, analogous to the Airy parametrix. We start with the exact but uninspiring definition of it, and then show that it satisfies a specified RH problem. More precisely, we define $(\theta + 1)^2$ functions $\Psi_{kj}^{(\Mei)}(\zeta)$ ($j, k = 0, 1, \dotsc, \theta$) explicitly by the Meijer G-functions in three steps for \ref{enu:Meijer_G_kernel_step_1} $j = 2, \dotsc, \theta$, \ref{enu:Meijer_G_kernel_step_2} $j = 1$, \ref{enu:Meijer_G_kernel_step_3} $j = 0$, and later show that they constitute a desired $(\theta + 1) \times (\theta + 1)$ matrix-valued function $\Psi^{(\Mei)}$. For $k=0,1,\ldots, \theta$, we set
\begin{equation}\label{def:psik}
    \psi_k(\zeta):=\zeta^{-(\alpha + 1 - \theta)} \MeijerG{\theta, 0}{0, \theta + 1}{-}{\frac{\alpha - \theta + 1}{\theta}, \frac{\alpha - \theta + 2}{\theta}, \dotsc, \frac{\alpha - 1}{\theta}, \frac{\alpha}{\theta}, k}{\zeta^\theta},
  \end{equation}
which is an entire function in $\zeta$, since the singularities in the integrand of the above Meijer G-function are located at the points $-\theta^{-1}(\alpha-\theta+m) - n$, where $m=1,\ldots,\theta$, and $n\in \{0\} \cup \mathbb{N}$. In particular, we will use the function $\psi_k(\zeta^{1/\theta})$ for $\arg \zeta \in (-\pi, (2\theta-1)\pi)$ such that it is real when $\arg \zeta =0$. The functions $\Psi_{kj}^{(\Mei)}(\zeta)$ are defined as follows.


\begin{enumerate}[label=(\arabic*), ref=(\arabic*)]
 \item $j = 2 \dotsc, \theta$: \label{enu:Meijer_G_kernel_step_1}
\begin{equation} \label{eq:defn_tilde_G_kj}
  \Psi_{kj}^{(\Mei)}(\zeta) = (-1)^k \theta \psi_k(e^{\frac{2(j-1)\pi i}{\theta}}\zeta^{\frac{1}{\theta}}), \quad \arg \zeta \in (-\pi, \pi).
\end{equation}
\item $j = 1$: \label{enu:Meijer_G_kernel_step_2}
  \begin{multline} \label{def:psiMeik1}
    \Psi_{k1}^{(\Mei)}(\zeta) = \theta \zeta^{-\frac{\alpha + 1 - \theta}{\theta}} \\
    \times
    \begin{dcases}
      (-1)^k \zeta^{\frac{\alpha + 1 - \theta}{\theta}}\psi_k(\zeta^{\frac{1}{\theta}}), & \arg \zeta \in (\frac{\pi}{2}, \pi) \cup (-\pi, -\frac{\pi}{2}), \\
      -\frac{1}{2\pi i} \MeijerG{\theta + 1, 0}{0, \theta + 1}{-}{\frac{\alpha - \theta + 1}{\theta}, \frac{\alpha - \theta + 2}{\theta}, \dotsc, \frac{\alpha - 1}{\theta}, \frac{\alpha}{\theta}, k}{\zeta e^{\pi i}}, & \arg \zeta \in (0, \frac{\pi}{2}), \\
      \frac{1}{2\pi i} \MeijerG{\theta + 1, 0}{0, \theta + 1}{-}{\frac{\alpha - \theta + 1}{\theta}, \frac{\alpha - \theta + 2}{\theta}, \dotsc, \frac{\alpha - 1}{\theta}, \frac{\alpha}{\theta}, k}{\zeta e^{-\pi i}}, & \arg \zeta \in (-\frac{\pi}{2}, 0).
    \end{dcases}
  \end{multline}
\item $j = 0$: \label{enu:Meijer_G_kernel_step_3}
  \begin{equation} \label{def:PsiMeik0}
    \Psi^{(\Mei)}_{k0}(\zeta) = \frac{1}{2\pi i} \times
    \begin{dcases}
      \MeijerG{\theta + 1, 0}{0, \theta + 1}{-}{\frac{\alpha - \theta + 1}{\theta}, \frac{\alpha - \theta + 2}{\theta}, \dotsc, \frac{\alpha - 1}{\theta}, \frac{\alpha}{\theta}, k}{\zeta e^{-\pi i}}, & \arg \zeta \in (0, \pi), \\
      \MeijerG{\theta + 1, 0}{0, \theta + 1}{-}{\frac{\alpha - \theta + 1}{\theta}, \frac{\alpha - \theta + 2}{\theta}, \dotsc, \frac{\alpha - 1}{\theta}, \frac{\alpha}{\theta}, k}{\zeta e^{\pi i}}, & \arg \zeta \in (-\pi, 0).
    \end{dcases}
  \end{equation}
\end{enumerate}

Analogous to the RH problem \ref{rhp:Ai} for the Airy parametrix $\Psi^{(\Ai)}$, the following model RH problem will play an important role in our construction of a local parametrix near the origin.
\begin{RHP}\label{RHP:MeiG} \hfill
  \begin{enumerate}[label=\emph{(\arabic*)}, ref=(\arabic*)]
  \item \label{enu:thm:MeiG_-1}
    $\Psi^{(\Mei)}(\zeta)$ is a $(\theta+1)\times (\theta+1)$ matrix-valued function defined and analytic in $\compC \setminus \{\mathbb{R}\cup i\mathbb{R} \}$.
  \item \label{enu:thm:MeiG_0}
    For $\zeta \in \mathbb{R}\cup i\mathbb{R} \setminus\{0\}$, we have
    \begin{equation}\label{eq:PsiMeijump}
      \Psi^{(\Mei)}_+(\zeta) = \Psi^{(\Mei)}_-(\zeta) J_{\Psi^{(\Mei)}}(\zeta),
    \end{equation}
    where
    \begin{equation} \label{eq:defn_J_Mei}
      J_{\Psi^{(\Mei)}}(\zeta) =
      \begin{cases}
        \begin{pmatrix}
          1 & \frac{\theta}{\zeta^{(\alpha + 1 - \theta)/\theta}} \\
          0 & 1
        \end{pmatrix}
        \oplus I_{\theta - 1}, & \zeta \in i\realR \setminus \{0\}, \\
        \begin{pmatrix}
          0 & -\frac{\theta}{\zeta^{(\alpha + 1 - \theta)/\theta}} \\
          \frac{\zeta^{(\alpha + 1 - \theta)/\theta}}{\theta} & 0
        \end{pmatrix}
        \oplus I_{\theta - 1}, & \zeta >0, \\
        \Mcyclic, & \zeta <0,
      \end{cases}
    \end{equation}
    and orientations of the real and imaginary axes are shown in Figure \ref{fig:jumps-MeiG}.
  \item \label{enu:thm:MeiG_1}
    As $\zeta \to \infty$, we have, for $\zeta \in \compC_{\pm}$,
    \begin{multline} \label{eq:G_asy_infty}
      \Psi^{(\Mei)}(\zeta) = \frac{1}{2\pi i} \frac{\theta (2\pi)^{\frac{\theta}{2}}}{\sqrt{\theta + 1}} e^{\left(\frac{2(\alpha + 1)}{\theta} - \frac{\alpha + 3/2}{\theta + 1} \right) \pi i} \zeta^{\frac{\alpha + 3/2}{\theta + 1} - \frac{\alpha + 1}{\theta}} \Upsilon(\zeta)
      \\
      \times \left( \Omega_{\pm} + \bigO(\zeta^{-\frac{1}{\theta + 1}}) \right) e^{-\Lambda(\zeta)}\Xi(\zeta),
    \end{multline}
    where
    \begin{equation}
      \Upsilon(\zeta) =  \diag\left(e^{-\frac{k}{\theta + 1} \pi i} \zeta^{\frac{k}{\theta + 1}}\right)_{k=0}^\theta, \label{def:Upsilon}
    \end{equation}
    \begin{equation}
      \Lambda(\zeta) = (\theta+1)\zeta^{\frac{1}{\theta+1}}
                           \begin{cases}
                             \begin{pmatrix}
                               e^{-\frac{\pi i}{\theta+1}} & 0 \\
                               0 & e^{\frac{\pi i}{\theta+1}}
                             \end{pmatrix}
                             \oplus \diag\left(e^{\frac{2j-1  }{\theta+1}\pi i}\right)_{j=2}^{\theta}, & \zeta \in \compC_+, \\
                             \begin{pmatrix}
                               e^{\frac{\pi i}{\theta+1}} & 0 \\
                               0 & e^{-\frac{\pi i}{\theta+1}}
                             \end{pmatrix}
                             \oplus \diag\left(e^{\frac{2j-1 }{\theta+1}\pi i} \right)_{j=2}^{\theta}, & \zeta \in \compC_-,
                           \end{cases} \label{def:Lambda}
    \end{equation}
    \begin{multline} \label{def:Xi}
      \Xi(\zeta) =
      \left\{
        \begin{aligned}
          \begin{pmatrix}
            -\frac{1}{\theta} e^{-\frac{2(\alpha + 1)}{\theta} \pi i} \zeta^{\frac{\alpha + 1 - \theta}{\theta}} & 0 \\
            0 & e^{2\left( \frac{\alpha + 3/2}{\theta + 1} - \frac{\alpha + 1}{\theta} \right) \pi i}
          \end{pmatrix},
          && \zeta \in \compC_+, \\
          \begin{pmatrix}
            -\frac{1}{\theta} e^{2\left( \frac{\alpha + 3/2}{\theta + 1} - \frac{\alpha + 1}{\theta} \right) \pi i} \zeta^{\frac{\alpha + 1 - \theta}{\theta}} & 0 \\
            0 & -e^{-\frac{2(\alpha+1)}{\theta} \pi i}
          \end{pmatrix},
          && \zeta \in \compC_-,
        \end{aligned}
      \right\} \\
      \oplus \diag\left(e^{ 2 j \left( \frac{\alpha + 3/2}{\theta + 1} - \frac{\alpha + 1}{\theta} \right) \pi i}\right)_{j=2}^{\theta},
    \end{multline}
    and
    \begin{equation} \label{def:Omegapm}
      \Omega_+=\left(e^{\frac{2 k j}{\theta+1}\pi i}\right)_{k,j=0}^{\theta},\quad \Omega_-=\Omega_+ \left(
        \begin{pmatrix}
          0 & 1 \\
          1 & 0
        \end{pmatrix}
        \oplus I_{\theta - 1} \right).
    \end{equation}
  \item \label{enu:thm:MeiG_2}
    As $\zeta \to 0$ from $\compC \setminus (\mathbb{R}\cup i\mathbb{R})$, we have for $k = 0, 1, \dotsc, \theta$,
    \begin{enumerate}[label=\emph{(\alph*)}, ref=(\arabic*)]
    \item
      \begin{equation} \label{eq:psiMeik0zero}
        \Psi^{(\Mei)}_{k0}(\zeta) =
        \begin{cases}
          \bigO(1), & \alpha > \theta - 1, \\
          \bigO(\log \zeta), & \alpha = \theta - 1, \\
          \bigO(\zeta^{\frac{\alpha + 1 - \theta}{\theta}}), & -1 < \alpha < \theta - 1,
        \end{cases}
      \end{equation}
    \item
      for $\arg \zeta \in (0, \pi/2) \cup (-\pi/2, 0)$,
      \begin{equation} \label{eq:psiMeik1zero}
        \Psi^{(\Mei)}_{k1}(\zeta) = \bigO(\zeta^{-\frac{\alpha + 1 - \theta}{\theta}} \Psi^{(\Mei)}_{k0}(\zeta)),
      \end{equation}
    \item
      for $j=2,3,\ldots,\theta$, or $j=1$ and $\arg \zeta \in (\pi/2,\pi)\cup (-\pi,-\pi/2)$,
      \begin{equation} \label{eq:psiMeikjzero}
        \Psi^{(\Mei)}_{kj}(\zeta) = \bigO(1),
      \end{equation}
    \end{enumerate}
where $\Psi^{(\Mei)}_{kj}(\zeta)$ stands for the $(k,j)$-th entry of $\Psi^{(\Mei)}(\zeta)$.

  \item \label{enu:prop:detPsiMei_2}
    As $\zeta \to 0$, each entry of $\Psi^{(\Mei)}(\zeta)^{-1}$ blows up at most as a power function near the origin. More precisely, we have, for $j=0,1,\ldots,\theta$,
    \begin{enumerate}[label=\emph{(\alph*)}, ref=(\arabic*)]
    \item for $\arg \zeta \in (\pi/2,\pi)\cup (-\pi,-\pi/2)$,
      \begin{equation} \label{eq:PsiMeiInv0jzero}
        (\Psi^{(\Mei)}(\zeta)^{-1})_{0j} = \bigO(1),
      \end{equation}
    \item for  $\arg \zeta \in (0, \pi/2) \cup (-\pi/2, 0)$,
      \begin{equation} \label{eq:PsiMeiInv0jzero2}
        (\Psi^{(\Mei)}(\zeta)^{-1})_{0j} = \begin{cases}
          \bigO(\zeta^{-\frac{\alpha}{\theta}}), & \alpha > 0, \\
          \bigO(\log \zeta), & \alpha = 0, \\
          \bigO(1), & -1 < \alpha < 0,
        \end{cases}
      \end{equation}
    \item
      for $k=1,\dotsc, \theta$,
     \begin{equation} \label{eq:PsiMeiInvkjzero}
        (\Psi^{(\Mei)}(\zeta)^{-1})_{kj} =
        \begin{cases}
          \bigO(\zeta^{\frac{1 - \theta}{\theta}}), & \alpha > 0, \\
          \bigO(\zeta^{\frac{1 - \theta}{\theta}} \log \zeta), & \alpha = 0, \\
          \bigO(\zeta^{\frac{\alpha + 1 - \theta}{\theta}}), & -1 < \alpha < 0,
        \end{cases}
      \end{equation}
    \end{enumerate}
\end{enumerate}
where $(\Psi^{(\Mei)}(\zeta)^{-1})_{kj}$ stands for the $(k,j)$-th entry of $\Psi^{(\Mei)}(\zeta)^{-1}$.
\end{RHP}

In what follows, we show the above RH problem can be solved explicitly.
\begin{prop}\label{thm:MeiG}
Let $\Psi^{(\Mei)}_{kj}(\zeta)$ be the functions given in \eqref{eq:defn_tilde_G_kj}--\eqref{def:PsiMeik0}. The matrix-valued function $\Psi^{(\Mei)}(\zeta)$ defined by
  \begin{equation}\label{eq:PsiMei}
    \Psi^{(\Mei)}(\zeta)=\left(\Psi^{(\Mei)}_{kj}(\zeta)\right)_{k,j=0}^{\theta},
  \end{equation}
  solves RH problem \ref{RHP:MeiG} for $\Psi^{(\Mei)}$. Moreover, we have
    \begin{equation}\label{eq:detPsiMei}
      \det(\Psi^{(\Mei)}(\zeta)) = \theta^{\theta} (2\pi)^{\frac{(\theta + 1)(\theta - 2)}{2}} e^{\frac{\theta(3 - \theta)}{4} \pi i} \zeta^{\frac{\theta-1}{2}}.
    \end{equation}
\end{prop}
\begin{proof}
It is clear that $\Psi^{(\Mei)}(\zeta)$ defined in \eqref{eq:PsiMei} is analytic in $\compC \setminus \{\mathbb{R}\cup i\mathbb{R} \}$. We will then check the other items of RH problem \ref{RHP:MeiG} one by one, and leave the proof of \eqref{eq:detPsiMei} to the end.

\paragraph{Item \ref{enu:thm:MeiG_0}}
It is straightforward to check that the jump condition for $\Psi^{(\Mei)}(\zeta)$ on $\mathbb{R} \setminus \{0\}$ is satisfied if it is defined by \eqref{eq:PsiMei}. For the jump on $i\mathbb{R}\setminus \{0\}$, we observe from the integral representation of Meijer G-function \eqref{def:Meijer} that for $k=0,1,\ldots,\theta$, and $\arg \zeta \in (-\pi, \pi)$,
\begin{align}
      & \MeijerG{\theta + 1, 0}{0, \theta + 1}{-}{\frac{\alpha - \theta + 1}{\theta}, \frac{\alpha - \theta + 2}{\theta}, \dotsc, \frac{\alpha - 1}{\theta}, \frac{\alpha}{\theta}, k}{\zeta e^{-\pi i}}- \MeijerG{\theta + 1, 0}{0, \theta + 1}{-}{\frac{\alpha - \theta + 1}{\theta}, \frac{\alpha - \theta + 2}{\theta}, \dotsc, \frac{\alpha - 1}{\theta}, \frac{\alpha}{\theta}, k}{\zeta e^{\pi i}}
      \nonumber
      \\
      & =  \frac{1}{\pi}\int_{L}\prod_{l=1}^{\theta}\Gamma\left(u+\frac{\alpha-\theta+l}{\theta}\right)\Gamma(u+k)\sin( \pi u)\zeta^{-u}du
      \nonumber \\
      & =  (-1)^k\int_{L}\frac{\prod_{l=1}^{\theta}\Gamma\left(u+\frac{\alpha-\theta+l}{\theta}\right)}{\Gamma(1-k-u)}\zeta^{-u}du
      \nonumber
      \\
      & = (-1)^k 2\pi i \MeijerG{\theta, 0}{0, \theta + 1}{-}{\frac{\alpha - \theta + 1}{\theta}, \frac{\alpha - \theta + 2}{\theta}, \dotsc, \frac{\alpha - 1}{\theta}, \frac{\alpha}{\theta}, k}{\zeta}, \label{eq:Grelation}
\end{align}
where we have made use of the reflection formula
  \begin{equation}\label{eq:reflformula}
    \Gamma(z)\Gamma(1-z)=\frac{\pi}{\sin(\pi z)}, \qquad z \neq 0, \pm 1, \pm 2, \ldots,
  \end{equation}
in the second equality. An appeal to this relation gives us the jump of $\Psi^{(\Mei)}$ on $i\realR \setminus \{ 0 \}$ as indicated in \eqref{eq:PsiMeijump} and \eqref{eq:defn_J_Mei}.

  \paragraph{Item \ref{enu:thm:MeiG_1}}
The large $\zeta$ behaviour of $\Psi^{(\Mei)}(\zeta)$ essentially follows from asymptotics of the Meijer G-functions. Indeed, by \cite[Theorem 5 in Section 5.7 and Theorem 2 in Section 5.10]{Luke69} and the definitions of $\Psi^{(\Mei)}_{kj}$, it follows that, for $k=0,1,\ldots,\theta$,
  \begin{multline} \label{eq:conv_G_k0}
    \Psi^{(\Mei)}_{k0}(\zeta) = -\frac{1}{2\pi i} \frac{(2\pi)^{\frac{\theta}{2}}}{\sqrt{\theta + 1}} \zeta^{\frac{\alpha + k + 3/2}{\theta + 1} - 1} \\
    \times
    \begin{cases}
      e^{-\frac{\alpha + k + 3/2}{\theta + 1}\pi i} \exp \left( -(\theta + 1) e^{-\frac{\pi i}{\theta + 1}} \zeta^{\frac{1}{\theta + 1}} \right) (1 + \bigO(\zeta^{-\frac{1}{\theta + 1}})), & \arg \zeta \in (0, \pi),
      \\
      e^{\frac{\alpha + k + 3/2}{\theta + 1}\pi i} \exp \left( -(\theta + 1) e^{\frac{\pi i}{\theta + 1}} \zeta^{\frac{1}{\theta + 1}} \right) (1 + \bigO(\zeta^{-\frac{1}{\theta + 1}})), & \arg  \zeta \in (-\pi, 0),
    \end{cases}
  \end{multline}
  \begin{multline} \label{eq:conv_G_k1}
    \Psi^{(\Mei)}_{k1}(\zeta) = \frac{1}{2\pi i} \frac{\theta(2\pi)^{\frac{\theta}{2}}}{\sqrt{\theta + 1}} \zeta^{\frac{\alpha + k + 3/2}{\theta + 1} - \frac{\alpha + 1}{\theta}} \\
    \times
    \begin{cases}
      e^{\frac{\alpha + k + 3/2}{\theta + 1} \pi i} \exp \left( -(\theta + 1) e^{\frac{\pi i}{\theta + 1}} \zeta^{\frac{1}{\theta + 1}} \right) (1 + \bigO(\zeta^{-\frac{1}{\theta + 1}})), & \arg \zeta \in (0, \frac{\pi}{2}) \cup (\frac{\pi}{2}, \pi), \\
      -e^{-\frac{\alpha + k + 3/2}{\theta + 1} \pi i} \exp \left( -(\theta + 1) e^{-\frac{\pi i}{\theta + 1}} \zeta^{\frac{1}{\theta + 1}} \right) (1 + \bigO(\zeta^{-\frac{1}{\theta + 1}})), & \arg \zeta \in (-\pi, -\frac{\pi}{2}) \cup (-\frac{\pi}{2}, 0),
    \end{cases}
  \end{multline}
  and if $j = 2, \dotsc, \theta$,
  \begin{multline} \label{eq:asy_G_jk_infty}
    \Psi^{(\Mei)}_{kj}(\zeta) = \frac{1}{2\pi i} \frac{\theta (2\pi)^{\frac{\theta}{2}}}{\sqrt{\theta + 1}} \zeta^{\frac{\alpha + k + 3/2}{\theta + 1} - \frac{\alpha + 1}{\theta}} e^{j \left( \frac{k + 3/2}{\theta + 1} - \frac{\alpha}{\theta(\theta + 1)} - \frac{1}{\theta} \right)2 \pi i} e^{\left( \frac{\alpha - k - 3/2}{\theta + 1} + \frac{2\alpha}{\theta(\theta + 1)} + \frac{2}{\theta} \right) \pi i} \\
    \times   \exp \left( -(\theta + 1) e^{\frac{(2j - 1)\pi i}{\theta + 1}} \zeta^{\frac{1}{\theta + 1}} \right) (1 + \bigO(\zeta^{-\frac{1}{\theta + 1}})), \quad \arg \zeta \in (-\pi, \pi).
  \end{multline}
  Inserting the above individual asymptotics into \eqref{eq:PsiMei}, we obtain \eqref{eq:G_asy_infty} after direct calculations. We note that the jump of $\Psi^{(\Mei)}$ on the imaginary axis does not affect the sectoral asymptotics, since
  \begin{equation}
    e^{-\Lambda(\zeta)} \Xi(\zeta) \left(
      \begin{pmatrix}
        1 & \frac{\theta}{\zeta^{(\alpha + 1 - \theta)/\theta}} \\
        0 & 1
      \end{pmatrix}
      \oplus I_{\theta - 1} \right) \Xi(\zeta)^{-1}e^{\Lambda(\zeta)}=I+\bigO(\zeta^{-\infty}),\qquad \zeta \in i\mathbb{R}\setminus \{0\}.
  \end{equation}
  Moreover, it is straightforward to check that, for $\zeta \in \mathbb{R}\setminus \{0\}$,
  \begin{equation}\label{eq:UpsilonOmegajump}
    \Upsilon_+(\zeta) \Omega_+ =  \Upsilon_-(\zeta) \Omega_- \times
    \begin{cases}
      \begin{pmatrix}
        0 & 1 \\
        1 & 0
      \end{pmatrix}
      \oplus I_{\theta - 1}, & \zeta >0, \\
      \Mcyclic, & \zeta <0,
    \end{cases}
  \end{equation}
  which is consistent with the jump of $\Psi^{(\Mei)}$ on the real axis.


  \paragraph{Item \ref{enu:thm:MeiG_2}}
  We first consider the entries $\Psi^{(\Mei)}_{kj}(\zeta)$ with either $j= 2, \dotsc, \theta$, or $j = 1$ and $\arg \zeta \in(\pi/2, \pi) \cup (-\pi, -\pi/2)$. Recall that the function $\psi_k(\zeta)$ in \eqref{def:psik} is entire in $\zeta$, it then follows from \eqref{eq:defn_tilde_G_kj} and \eqref{def:psiMeik1} that for all $k = 0, 1, \dots, \theta$, $\Psi^{(\Mei)}_{kj}(\zeta)$ admit the following Puiseux series representations:
  \begin{equation} \label{eq:Psikjzero}
    \Psi^{(\Mei)}_{kj}(\zeta) = \sum^{\infty}_{l = 0} a_{k, l} e^{\frac{2(j - 1)l\pi i}{\theta}} \zeta^{\frac{l}{\theta}}.
  \end{equation}
  The coefficients $a_{k,l}$ depending on the parameters $\alpha$ and $\theta$ in the above formula can be calculated explicitly by evaluating the residue of the integrand of $\MeijerG{\theta, 0}{0, \theta + 1}{-}{(\alpha - \theta + 1)/\theta, (\alpha - \theta + 2)/\theta, \dotsc, (\alpha - 1)/\theta, \alpha/\theta, k}{\zeta}$ at each pole. In particular, we have $a_{k,i}\neq 0$ for $i=0,1,\ldots,\theta-1$. Equivalently, by setting the analytic functions
  \begin{equation} \label{eq:defn_F_kj}
    f_{k,j}(\zeta) := \sum^{\infty}_{l = 0} a_{k, \theta l + j - 1} \zeta^l,\qquad j=1,\ldots,\theta,
  \end{equation}
  it follows that for $j =2, \dotsc, \theta$, or $j=1$ and $\arg \zeta \in (\pi/2,\pi)\cup (-\pi,-\pi/2)$,
  \begin{equation}\label{eq:expr_tilde_G_kj}
    \Psi^{(\Mei)}_{kj}(\zeta) = \sum^{\theta}_{l = 1} c_{l, j} \zeta^{\frac{l - 1}{\theta}} f_{k, l}(\zeta), \qquad c_{l, j} = e^{\frac{2(j - 1)(l - 1)\pi i}{\theta}},\qquad l,j=1,\ldots,\theta,
  \end{equation}
which gives us \eqref{eq:psiMeikjzero}.

We next show \eqref{eq:psiMeik0zero} by splitting the discussions into two cases, namely, \ref{case:alphanint} $\alpha \notin \intZ$, and \ref{case:alphaint} $\alpha\in \{0\}\cup \mathbb{N}$.
  \begin{enumerate}[label=(\roman*)]
  \item \label{case:alphanint} $\alpha \notin \intZ$.
    From the definition of $\Psi^{(\Mei)}_{k0}$ given in \eqref{def:PsiMeik0}, we see from \eqref{def:Meijer} and \eqref{eq:reflformula} that
      \begin{align} \label{eq:PsiMeiInt}
        \frac{2i}{ (-1)^k  } \Psi^{(\Mei)}_{k0}(\zeta) = {}& \frac{1}{(-1)^k \pi } \MeijerG{\theta + 1, 0}{0, \theta + 1}{-}{\frac{\alpha - \theta + 1}{\theta}, \frac{\alpha - \theta + 2}{\theta}, \dotsc, \frac{\alpha - 1}{\theta}, \frac{\alpha}{\theta}, k}{\zeta e^{\mp \pi i}} \nonumber \\
        = {}& \frac{1}{(-1)^k \pi }\frac{1}{2\pi i}\int_{L}\prod_{l=1}^{\theta}\Gamma\left(u+\frac{\alpha-\theta+l}{\theta}\right)\Gamma(u+k)(\zeta e^{\mp \pi i})^{-u}du \nonumber \\
        = {}& \frac{1}{2\pi i} \int_{L}\frac{\prod_{l=1}^{\theta} \Gamma\left(u+\frac{\alpha-\theta+l}{\theta}\right)}{\Gamma(1-k-u)\sin(\pi u)}(\zeta e^{\mp \pi i})^{-u}du, \qquad \zeta \in \compC_{\pm}.
      \end{align}

    On account of \eqref{eq:Psikjzero} and \eqref{def:psiMeik1}, we also have
    \begin{align}
       {}& \MeijerG{\theta, 0}{0, \theta + 1}{-}{\frac{\alpha - \theta + 1}{\theta}, \frac{\alpha - \theta + 2}{\theta}, \dotsc, \frac{\alpha - 1}{\theta}, \frac{\alpha}{\theta}, k}{\zeta}
        \nonumber \\ {}& =  \frac{1}{2\pi i} \int_{L}\frac{\prod_{l=1}^{\theta}\Gamma\left(u+\frac{\alpha-\theta+l}{\theta}\right)}{\Gamma(1-k-u)}\zeta^{-u}du
        = \frac{(-1)^k}{\theta}\zeta^{\frac{\alpha+1-\theta}{\theta}}\sum^{\infty}_{l = 0} a_{k, l} \zeta^{\frac{l}{\theta}}.
    \end{align}
    Thus, a combination of the above two formulas and the residue theorem implies that, for $\zeta \in \compC_{\pm}$,
    \begin{align}\label{eq:expr_tilde_G_k0}
        \Psi^{(\Mei)}_{k0}(\zeta) = {}& f_{k,0}(\zeta) + \frac{\zeta^{\frac{\alpha + 1 - \theta}{\theta}}}{\theta} \sum^{\infty}_{l = 0} a_{k, l} \frac{i e^{\mp \frac{\alpha + 1 - \theta + l}{\theta} \pi i}}{2\sin(\frac{\alpha + 1 - \theta + l}{\theta} \pi)}\zeta^{\frac{l}{\theta}} \nonumber \\
        = {}& f_{k,0}(\zeta) + \zeta^{\frac{\alpha - \theta}{\theta}} \sum^{\theta}_{l = 1} c_{l, 0} \zeta^{\frac{l}{\theta}} f_{k, l}(\zeta), \qquad c_{l, 0} = \frac{i e^{\mp \frac{\alpha - \theta + l}{\theta} \pi i}}{2\theta \sin(\frac{\alpha - \theta + l}{\theta} \pi)},
      \end{align}
   where $f_{k,0}(\zeta)$ is an analytic function with $f_{k,0}(\zeta) = \bigO(\zeta^k)$ as $\zeta \to 0$. Since $\alpha \notin \intZ$, we note that $\sin(\frac{\alpha - \theta + l}{\theta} \pi)\neq 0$ for $l=1,\ldots,\theta$, which leads to \eqref{eq:psiMeik0zero} in this case.

  \item \label{case:alphaint} $\alpha\in \{0\}\cup \mathbb{N}$.
  If $\alpha$ is a nonnegative integer, there exists a unique $j_{\alpha} \in \{ 1, \dotsc, \theta \}$ such that $(\alpha - \theta + j_{\alpha})/\theta \in \intZ$. More precisely, we have
    \begin{equation}
      m_{\alpha} := \frac{\alpha - \theta + j_{\alpha}}{\theta} \geq 0.
    \end{equation}
    The analytic functions $f_{k,j}(\zeta)$ for $k=0,1,\ldots,\theta$, and $j = 1, \dotsc, \theta$, by \eqref{eq:defn_F_kj}, satisfy $f_{k,j}(0) \neq 0$ if $j \neq j_{\alpha}$, while as $\zeta \to 0$,
    \begin{equation} \label{eq:f_kj_alpha_int}
      f_{k,j_{\alpha}}(\zeta) =
      \begin{cases}
        \bigO(1), & k \leq m_{\alpha}, \\
        \bigO(\zeta^{k - m_{\alpha}}), & k > m_{\alpha}.
      \end{cases}
    \end{equation}
    The representation of $\Psi^{(\Mei)}_{k0}$ analogue to \eqref{eq:expr_tilde_G_k0} in this case is
    \begin{multline} \label{eq:G_tilde_k0_degen}
      \Psi^{(\Mei)}_{k0}(\zeta) = \zeta^{\frac{\alpha - \theta}{\theta}} \sum_{l = 1, \dotsc, \theta, \, l \neq j_{\alpha}} c_{l,0}\zeta^{\frac{l}{\theta}} f_{k, l}(\zeta) \\
      + f_{k, 0}(\zeta) - \frac{\zeta^{m_{\alpha}}}{2 \theta \pi i}  f_{k, j_{\alpha}}(\zeta) \times
      \begin{cases}
        \log(\zeta e^{-\pi i}), & \arg \zeta \in (0, \pi), \\
        \log(\zeta e^{\pi i}), & \arg \zeta \in (-\pi, 0),
      \end{cases}
    \end{multline}
    where $c_{l,0}$ is given in \eqref{eq:expr_tilde_G_k0} and $f_{k, 0}(\zeta)$ is an analytic function with $f_{k, 0}(\zeta) = \bigO(\zeta^{\min \{k, m_{\alpha} \}})$. Hence, we again obtain \eqref{eq:psiMeik0zero} in this case.
  \end{enumerate}

Finally, the local behaviour of $\Psi^{(\Mei)}_{k1}(\zeta)$ near the origin for $\arg \zeta \in(0, \pi/2) \cup (-\pi/2, 0)$ given in \eqref{eq:psiMeik1zero} follows directly from \eqref{def:psiMeik1}, \eqref{def:PsiMeik0} and \eqref{eq:psiMeik0zero}.

\paragraph{Item \ref{enu:prop:detPsiMei_2}}

In the proof of item \ref{enu:prop:detPsiMei_2}, we assume \eqref{eq:detPsiMei}, which implies that $\Psi^{(\Mei)}(\zeta)^{-1}$ is well-defined if $\zeta \neq 0$.

If $\alpha \notin \intZ$ and $\arg \zeta \in (\pi/2,\pi)\cup (-\pi,-\pi/2)$, it is readily seen from \eqref{eq:expr_tilde_G_kj} and \eqref{eq:expr_tilde_G_k0} that
  \begin{equation} \label{eq:PsiMeiExp}
    \Psi^{(\Mei)}(\zeta) = \Theta(\zeta) \diag(\zeta^{\frac{\theta - \alpha - 1}{\theta}}, 1, \zeta^{\frac1\theta}, \dotsc, \zeta^{\frac{\theta - 1}{\theta}}) \mathsf{C}_{\alpha} \diag(\zeta^{\frac{\alpha + 1 - \theta}{\theta}}, 1, \dotsc, 1),
  \end{equation}
  where
  \begin{equation}
    \Theta(\zeta)=
    \begin{pmatrix}
      f_{k, j}(\zeta)
    \end{pmatrix}^{\theta}_{k, j = 0}
  \end{equation}
  with $f_{k,j}(\zeta)$, $k,j=0,1,\ldots,\theta$, given in \eqref{eq:defn_F_kj} and \eqref{eq:expr_tilde_G_k0}, and
  \begin{equation}\label{eq:psimeizero}
    \mathsf{C}_{\alpha}=
    \begin{pmatrix}
      1 & 0_{1 \times \theta} \\
      \vec{c}_{\alpha} & C_{\alpha}
    \end{pmatrix},
    \quad \text{where} \quad C_{\alpha} =
    \begin{pmatrix}
      c_{1,1} & \cdots & c_{1,j} & \cdots & c_{1,\theta} \\
      c_{2,1} & \cdots & c_{2,j} & \cdots & c_{2,\theta} \\
      \vdots &\cdots &\vdots &\cdots & \vdots \\
      c_{\theta,1} & \cdots & c_{\theta,j} & \cdots & c_{\theta,\theta}
    \end{pmatrix},
    \quad
    \vec{c}_{\alpha} =
    \begin{pmatrix}
      c_{1, 0} \\
      c_{2, 0} \\
      \vdots \\
      c_{\theta, 0}
    \end{pmatrix},
  \end{equation}
  with $c_{k,j}$, $k=1,\ldots,\theta$, $j=0,1,\ldots,\theta$, given in \eqref{eq:expr_tilde_G_kj} and \eqref{eq:expr_tilde_G_k0}. In view of \eqref{eq:det_Schur} below, it is clear that $\det \mathsf{C}_{\alpha} = \det C_{\alpha} \neq 0$, and then
  \begin{equation}
    \mathsf{C}^{-1}_{\alpha} =
    \begin{pmatrix}
      1 & 0_{1 \times \theta} \\
      -C^{-1}_{\alpha} \vec{c}_{\alpha} & C^{-1}_{\alpha}
    \end{pmatrix}.
  \end{equation}
  Also, by \eqref{eq:detPsiMei} and \eqref{eq:PsiMeiExp}, it follows that $\det(\Theta(\zeta))$ is a non-zero constant and
  \begin{equation}\label{eq:ThetaInvZero}
    \Theta(\zeta)^{-1}=\bigO(1),
    \qquad \zeta \to 0,
  \end{equation}
  which is understood in an entry-wise manner.
  Since
  \begin{align} \label{eq:intermediate_Psi_inverse}
      \Psi^{(\Mei)}(\zeta)^{-1} = {}& \diag(\zeta^{\frac{\theta - \alpha - 1}{\theta}}, 1, \dotsc, 1) \mathsf{C}_{\alpha}^{-1} \diag(\zeta^{\frac{\alpha + 1 - \theta}{\theta}}, 1, \zeta^{-\frac1\theta}, \dotsc, \zeta^{-\frac{\theta - 1}{\theta}}) \Theta(\zeta)^{-1} \nonumber \\
      = {}&
      \begin{pmatrix}
        1 & 0 & 0 & \cdots & 0 \\
        \bigO(\zeta^{\frac{\alpha + 1 - \theta}{\theta}}) & \bigO(1) & \bigO(\zeta^{-\frac{1}{\theta}}) & \cdots & \bigO(\zeta^{-\frac{\theta - 1}{\theta}}) \\
        \vdots & \vdots & \vdots & \cdots & \vdots \\
        \bigO(\zeta^{\frac{\alpha + 1 - \theta}{\theta}}) & \bigO(1) & \bigO(\zeta^{-\frac{1}{\theta}}) & \cdots & \bigO(\zeta^{-\frac{\theta - 1}{\theta}})
      \end{pmatrix}
      \Theta(\zeta)^{-1},
  \end{align}
  we then obtain \eqref{eq:PsiMeiInv0jzero} and \eqref{eq:PsiMeiInvkjzero} for $\alpha \notin \intZ$ and $\arg \zeta \in (\pi/2, \pi) \cup (-\pi, -\pi/2)$.

If $\alpha \in \{ 0 \} \cup \mathbb{N}$ and $\arg \zeta \in (\pi/2,\pi)\cup (-\pi,-\pi/2)$, then the decomposition \eqref{eq:PsiMeiExp} up to \eqref{eq:ThetaInvZero} remain valid (with $f_{k,0}$ given in \eqref{eq:G_tilde_k0_degen}). The component $c_{j_{\alpha}, 0}$ in $\vec{c}_{\alpha}$, however, should be replaced by $i(2\theta \pi)^{-1} \log(\zeta e^{\mp \pi i})$ (for $\zeta \in \compC_{\pm}$), and as a consequence, the entries $\bigO(\zeta^{\frac{\alpha + 1 - \theta}{\theta}})$ in the first column of the matrix in the second line of \eqref{eq:intermediate_Psi_inverse} should be replaced by $\bigO(\zeta^{\frac{\alpha + 1 - \theta}{\theta}} \log \zeta)$, which again gives us \eqref{eq:PsiMeiInv0jzero} and \eqref{eq:PsiMeiInvkjzero}.

We note that all the estimates above are still valid if we analytically continue the entries of $\Psi^{(\Mei)}(\zeta)^{-1}$ to $(0, \pi) \cup (-\pi, 0)$, cross $\pm i\realR$. Actually, from \eqref{eq:PsiMeijump} and \eqref{eq:defn_J_Mei}, we find the jump
\begin{equation}
  \Psi^{(\Mei)}_-(\zeta)^{-1}= \left(
    \begin{pmatrix}
      1 & \frac{\theta}{\zeta^{(\alpha + 1 - \theta)/\theta}} \\
      0 & 1
    \end{pmatrix}
    \oplus I_{\theta - 1} \right)
  \Psi^{(\Mei)}_+(\zeta)^{-1},\quad \zeta\in i\mathbb{R}\setminus \{0\}.
\end{equation}
Thus, the estimates above for $\arg \zeta \in (\pi/2, \pi) \cup (-\pi, -\pi/2)$ also implies item \ref{enu:prop:detPsiMei_2} for $\arg \zeta \in (0,\pi/2) \cup (-\pi/2, 0)$, and we omit the details here.

\paragraph{Evaluation of $\det(\Psi^{(\Mei)}(\zeta)) $}
From \eqref{def:Upsilon}--\eqref{def:Omegapm}, it is readily seen that
  \begin{equation} \label{eq:detUpsilon}
    \begin{gathered}
      \det (\Upsilon(\zeta)) = e^{-\frac{\theta}{2}\pi i}\zeta^{\frac {\theta}{2}}, \quad \det (e^{-\Lambda(\zeta)})=1, \quad \det(\Omega_{\pm})=\pm (\theta+1)^{\frac{\theta+1}{2}}e^{\frac{\theta(5-\theta)}{4}\pi i}, \\
      \det (\Xi(\zeta)) = \mp\frac1\theta e^{(\frac{\theta}{2}-\alpha-1-\frac{2(1+\alpha)}{\theta})\pi i}\zeta^{\frac{\alpha+1-\theta}{\theta}},\qquad \zeta \in \compC_{\pm}.
    \end{gathered}
  \end{equation}
Here, to derive $\det(\Omega_{\pm})$, we have made use of the determinant formula of the Schur matrix, that is, for any $d \in \mathbb{N}$,
  \begin{equation} \label{eq:det_Schur}
    \det \left( (\xi^{kj})^{d - 1}_{k, j = 0} \right) = d^{\frac d2} e^{\frac{2+7d-d^2}{4}\pi i}, \qquad \xi = e^{\frac{2\pi i}{d}},
  \end{equation}
  cf.~\cite[Part 4, Chapter VI, Appdendix 2, Pages 211--213]{Landau58}. A combination of \eqref{eq:detUpsilon} and \eqref{eq:G_asy_infty} implies that, as $\zeta\to \infty$,
  \begin{equation} \label{eq:det_G_at_infty}
    \det(\Psi^{(\Mei)}(\zeta)) = \theta^{\theta} (2\pi)^{\frac{(\theta + 1)(\theta - 2)}{2}} e^{\frac{\theta(3 - \theta)}{4} \pi i} \zeta^{\frac{\theta-1}{2}} (1 + \bigO(\zeta^{-\frac{1}{\theta + 1}})).
  \end{equation}

We further observe from \eqref{eq:expr_tilde_G_kj} and \eqref{eq:expr_tilde_G_k0} (if $\alpha \notin \intZ$) and from \eqref{eq:expr_tilde_G_kj}, \eqref{eq:f_kj_alpha_int} and \eqref{eq:G_tilde_k0_degen} (if $\alpha\in \{0\}\cup \mathbb{N}$) that
  \begin{equation}\label{eq:detGzero}
    \det (\Psi^{(\Mei)}(\zeta)) = \bigO(\zeta^{\frac{\theta - 1}{2}}),\qquad \zeta \to 0.
  \end{equation}
Also it is easy to see from \eqref{eq:PsiMeijump} and \eqref{eq:defn_J_Mei} that $\zeta^{(1 - \theta)/2} \det (\Psi^{(\Mei)}(\zeta))$ is analytic on $\compC \setminus \{ 0 \}$. Thus, Liouville's theorem implies \eqref{eq:detPsiMei}.
\end{proof}

We omit the discussion on uniqueness of the solution to the RH problem \ref{RHP:MeiG}, and always understand $\Psi^{(\Mei)}(\zeta)$ in the sense of \eqref{eq:PsiMei}.

\subsubsection*{Construction of the local parametrix around $0$}
We are now ready to construct the local parametrix near the origin by building a $(\theta+1)\times (\theta+1)$ matrix-valued function $P^{(0)}(z)$ as an approximation of the RH problem \ref{rhp:U} for $U$. For this purpose, we define
\begin{equation}\label{def:M}
  M(z) = \diag \left( m_0(z), m_1(z), \dotsc, m_{\theta}(z) \right), \quad z\in\compC \setminus \realR,
\end{equation}
where
\begin{align}
  m_0(z) = {}&
               \begin{cases}
                 \g(z^{\frac1 \theta}) - \g_+(0), & z \in \compC_+, \\
                 2\pi i + \g(z^{\frac 1 \theta}) - \g_+(0), & z \in \compC_-,
               \end{cases}
                                                              \label{def:m0}\\
  m_1(z) = {}&
               \begin{cases}
                 -g(z^{\frac1 \theta}) + V(z^{\frac1 \theta}) + \ell - \g_-(0), & z \in \compC_+, \\
                 -2\pi i - g(z^{\frac1 \theta}) + V(z^{\frac1 \theta}) + \ell - \g_-(0), & z \in \compC_-,
               \end{cases} \\
  \intertext{and for $j = 2, \dotsc, \theta$, $\arg z \in (-\pi, 0) \cup (0, \pi)$,}
  m_j(z) = {}&
               \begin{cases}
                 -g(z^{\frac1 \theta} e^{\frac{2(j - 1)}{\theta}\pi i}) + V(z^{\frac1 \theta} e^{\frac{2(j - 1)}{\theta}\pi i}) + \ell - \g_-(0), & z^{\frac1 \theta} e^{\frac{2(j - 1)}{\theta}\pi i} \in \compC_+, \\
                 -2\pi i - g(z^{\frac1 \theta} e^{\frac{2(j - 1)}{\theta}\pi i}) + V(z^{\frac1 \theta} e^{\frac{2(j - 1)}{\theta}\pi i}) + \ell - \g_-(0), & z^{\frac1 \theta} e^{\frac{2(j - 1)}{\theta}\pi i} \in \compC_-,
               \end{cases}
                                                                                                                                                             \label{def:mj}
\end{align}
and also define
\begin{equation} \label{def:N}
  N(z) = \diag (n_0(z), n_1(z), \dotsc, n_{\theta}(z)), \quad z\in\compC \setminus (-\infty,b^{\theta}],
\end{equation}
where
\begin{equation}\label{def:ni}
  n_0(z) = P^{(\infty)}_2(z^{\frac1 \theta}), \quad n_j(z) = P^{(\infty)}_1(e^{\frac{2(j - 1)}{\theta} \pi i} z^{\frac1 \theta}), \quad j = 1, \dotsc, \theta.
\end{equation}
In the above definitions, we recall that $g$, $\g$ and $V$ are given in \eqref{def:g}, \eqref{def:tildeg} and \eqref{def:weight}, respectively, $\ell$ is the constant in the Euler-Lagrange condition \eqref{eq:gequal}, and $P^{(\infty)} =(P^{(\infty)}_1, P^{(\infty)}_2)$ solves the RH problem \ref{rhp:global}.

In view of \eqref{eq:gequal}, \eqref{eq:asy_formula_for_g}, \eqref{eq:asy_formula_for_g_tilde} and \eqref{eq:asy_P^infty_1}, \eqref{eq:asy_P^infty_2}, the following local behaviours of $M$ and $N$ near the origin are immediate.
\begin{prop}\label{prop:MNzero}
Let $m_i(z)$ and $n_i(z)$, $i=0,1,\ldots,\theta$, be the functions defined in \eqref{def:m0}--\eqref{def:mj} and \eqref{def:ni}. As $z\to 0$, we have, with $\rho$ defined in \eqref{eq:defn_rho}, $c$ given in \eqref{def:Jcs}, and $j = 2, \dotsc, \theta$,
\begin{align}
  m_0(z) = {}&
               \begin{cases}
                 \rho  (1 + \theta) e^{-\frac{1}{1 + \theta} \pi i} z^{\frac{1}{1 + \theta}} + \bigO(z^{\frac{2}{1 + \theta}}), & z \in \compC_+, \\
                 \rho (1 + \theta)  e^{\frac{1}{1 + \theta} \pi i} z^{\frac{1}{1 + \theta}} + \bigO(z^{\frac{2}{1 + \theta}}), & z \in \compC_-,
               \end{cases} \\
  m_1(z) = {}&
               \begin{cases}
                 \rho (1 + \theta)  e^{\frac{1}{1 + \theta} \pi i} z^{\frac{1}{1 + \theta}} + \bigO(z^{\frac{1}{\theta}}), & z \in \compC_+, \\
                 \rho (1 + \theta)  e^{-\frac{1}{1 + \theta} \pi i} z^{\frac{1}{1 + \theta}} + \bigO(z^{\frac{1}{\theta}}), & z \in \compC_-,
               \end{cases} \\
  m_j(z) = {}&  \rho (1 + \theta) e^{\frac{2j - 1}{1 + \theta} \pi i} z^{\frac{1}{1 + \theta}} + \bigO(z^{\frac{1}{\theta}}), \\
  n_0(z) = {}&
               \begin{cases}
                 \frac{c^{\frac{2(\alpha+1)-\theta}{2(1+\theta)}}}{\sqrt{\theta(1+\theta)}}e^{\frac{\theta-2(\alpha+1)}{2(1+\theta)}\pi i} z^{\frac{\alpha + 3/2}{1+\theta} - 1} (1 + \bigO(z^{\frac{1}{1 + \theta}})), & z \in \compC_+, \\
                 -\frac{c^{\frac{2(\alpha+1)-\theta}{2(1+\theta)}}}{\sqrt{\theta(1+\theta)}}e^{\frac{2(\alpha+1) - \theta}{2(1+\theta)}\pi i} z^{\frac{\alpha + 3/2}{1+\theta} - 1} (1 + \bigO(z^{\frac{1}{1 + \theta}})), & z \in \compC_-,
               \end{cases} \\
  n_1(z) = {}&
               \begin{cases}
                 \frac{\theta c^{\frac{2(\alpha+1)-\theta}{2(1+\theta)}}}{\sqrt{\theta(1+\theta)}} e^{\frac{2(\alpha+1)-\theta}{2(1+\theta)}\pi i}z^{\frac{\alpha + 3/2}{1+\theta} - \frac{\alpha + 1}{\theta}} (1 + \bigO(z^{\frac{1}{1+\theta}})), & z \in \compC_+, \\
                 \frac{\theta c^{\frac{2(\alpha+1)-\theta}{2(1+\theta)}}}{\sqrt{\theta(1+\theta)}} e^{\frac{\theta - 2(\alpha+1)}{2(1+\theta)}\pi i}z^{\frac{\alpha + 3/2}{1+\theta} - \frac{\alpha + 1}{\theta}} (1 + \bigO(z^{\frac{1}{1+\theta}})), & z \in \compC_-,
               \end{cases} \\
  n_j(z) = {}& \frac{\theta c^{\frac{2(\alpha+1)-\theta}{2(1+\theta)}}}{\sqrt{\theta(1+\theta)}} e^{\frac{2(\alpha+1)-\theta}{2(1+\theta)}\pi i} e^{\frac{(j - 1)(\theta - 2(\alpha + 1))}{\theta(1 + \theta)} \pi i} z^{\frac{\alpha + 3/2}{1+\theta} - \frac{\alpha + 1}{\theta}} (1 + \bigO(z^{\frac{1}{1+\theta}})).
\end{align}
\end{prop}

From now on, we assume, as mentioned before, that $r=r_n$ shrinks with $n$, namely,
\begin{equation}\label{def:rn}
r=r_n=n^{-\frac{2(\theta+1)}{2\theta+1}}.
\end{equation}
One could actually take the rate of shrinking to be $n^{-\kappa}$ with $1<\kappa <1 + 1/\theta$. Here, we choose $\kappa = 2(\theta + 1)/(2\theta + 1)$ so that more explicit error estimates can be obtained. For $z \in D(0,r_n^{\theta})\setminus (\mathbb{R} \cup i \mathbb{R})$, we first introduce a $(\theta + 1) \times (\theta + 1)$ matrix-valued function
\begin{multline}\label{def:sfP0}
  \mathsf{P}^{(0)}(z) = (2\pi)^{1 - \frac{\theta}{2}}\theta^{-\frac{1}{2}} c^{\frac{2(\alpha+1)-\theta}{2(1+\theta)}} (\rho n)^{\frac{\alpha + 1}{\theta} - \frac{1}{2}} \diag(1,(\rho n)^{-1},\ldots,(\rho n)^{-\theta}) \\
  \times \Psi^{(\Mei)}((\rho n)^{\theta + 1} z) ((\rho n)^{-\frac{\theta + 1}{\theta}(\alpha + 1 - \theta)} \oplus I_{\theta}) N(z)^{-1} e^{nM(z)},
\end{multline}
where $M(z)$ and $N(z)$ are defined in \eqref{def:M} and \eqref{def:N}, respectively, and $\Psi^{(\Mei)}$ is the Meijer G-parametrix solving the RH problem \ref{RHP:MeiG}. We then have the following proposition regarding the RH problem for $\mathsf{P}^{(0)}$.
\begin{prop}\label{RHP:sfP0}
The function $\mathsf{P}^{(0)}(z)$ defined in \eqref{def:sfP0} has the following properties.
  \begin{enumerate}[label=\emph{(\arabic*)}, ref=(\arabic*)]
  \item
    $ \mathsf{P}^{(0)}(z)$ is analytic in $D(0,r_n^{\theta}) \setminus (\mathbb{R} \cup i \mathbb{R})$.
  \item
    For $z\in D(0,r_n^{\theta}) \cap (\mathbb{R} \cup i \mathbb{R})$, we have
    \begin{equation}
      \mathsf{P}^{(0)}_{+}(z) = \mathsf{P}^{(0)}_{-}(z) J_U(z),
    \end{equation}
    where $J_U(z)$ is defined in \eqref{eq:defn_J_U}.
  \item
    For $z\in \partial D(0,r_n^{\theta})$, we have, as $n\to \infty$,
    \begin{equation}\label{eq:sfP0asy}
      \mathsf{P}^{(0)}(z) = \Upsilon(z)\Omega_{\pm} (I+\bigO(n^{-\frac{1}{2\theta+1}})), \qquad z \in \compC_{\pm},
    \end{equation}
    where $\Upsilon(z)$ and $\Omega_{\pm}$ are defined in \eqref{def:Upsilon} and \eqref{def:Omegapm}, respectively.
  \end{enumerate}
\end{prop}
\begin{proof}
  While it is straightforward to check the jump of $\mathsf{P}^{(0)}$ on $(0,ir_n^\theta)\cup(0,-ir_n^\theta)$, we need a little bit effort to see its jump on $(-r_n^\theta, 0)$ and $(0, r_n^\theta)$. If $z\in(-r_n^\theta,0)$, we obtain from the definitions of $m_i(z)$ and $n_{i}(z)$, $i=0,1,\ldots,\theta$, \eqref{eq:tildgpm}, \eqref{eq:P^infty_jump} and the boundary condition indicated in item \ref{enu:pinftyboundary} of the RH problem \ref{rhp:global} for $P^{(\infty)}$ that
  \begin{equation}\label{eq:relmi1}
    \begin{gathered}
      e^{nm_{0,+}(z)}= e^{nm_{0,-}(z)}, \quad e^{nm_{\theta,+}(z)} = e^{nm_{0,-}(z)},
      \\
      \quad n_{0,+}(z)=n_{0,-}(z), \quad n_{\theta,+}(z)=n_{1,-}(z), \\
      e^{nm_{j-1,+}(z)}= e^{nm_{j,-}(z)}, \quad n_{j-1,+}(z)=n_{j,-}(z), \qquad j=2, \dotsc, \theta.
    \end{gathered}
  \end{equation}
  It is then easily seen from \eqref{def:sfP0}, \eqref{eq:PsiMeijump} and the above relations that
  \begin{equation}
    \mathsf{P}^{(0)}_+(z) = \mathsf{P}^{(0)}_-(z) \Mcyclic, \quad z\in (-r_n^\theta,0).
  \end{equation}
  On the other hand, if $z\in (0,r_n^\theta)$, it is readily seen from \eqref{eq:P^infty_jump} that
  \begin{equation}\label{eq:relmi12}
    \begin{gathered}
    e^{nM_{+}(z)}=e^{nM_{-}(z)}, \quad n_{0,+}(z) = n_{1,-}(z) \frac{z^{(\alpha+1-\theta)/\theta}}{\theta}, \quad n_{1,+}(z) = -n_{0,-}(z)\frac{\theta}{z^{(\alpha + 1 - \theta)/\theta}}, \\
     \quad n_{j,+}(z)=n_{j,-}(z), \qquad j=2,\ldots,\theta.
    \end{gathered}
  \end{equation}
  This, together with \eqref{eq:PsiMeijump}, implies that for $z\in (0,r_n^\theta)$,
  \begin{align}
      & \Psi^{(\Mei)}_+((\rho n)^{\theta + 1} z) \left( (\rho n)^{-\frac{\theta + 1}{\theta}(\alpha + 1 - \theta)} \oplus I_{\theta} \right) N_+(z)^{-1} e^{nM_+(z)} \notag \\
      = {}& \Psi^{(\Mei)}_-((\rho n)^{\theta + 1} z) \left(
        \begin{pmatrix}
          0 & -\frac{\theta}{z^{(\alpha+1-\theta)/\theta}} (\rho n)^{-\frac{\theta+1}{\theta}(\alpha+1-\theta)}
          \\
          \frac{z^{(\alpha+1-\theta)/\theta}}{\theta} (\rho n)^{\frac{\theta+1}{\theta}(\alpha+1-\theta)} & 0
        \end{pmatrix}
        \oplus I_{\theta-1} \right) \notag \\
      & \phantom{\Psi^{(\Mei)}_-}
      \times ((\rho n)^{-\frac{\theta + 1}{\theta}(\alpha + 1 - \theta)} \oplus I_{\theta}) N_+(z)^{-1} e^{nM_+(z)} \notag \\
      = {}& \Psi^{(\Mei)}_-((\rho n)^{\theta + 1} z) \left( (\rho n)^{-\frac{\theta + 1}{\theta}(\alpha + 1 - \theta)} \oplus I_{\theta} \right) \left(
        \begin{pmatrix}
          0 & -\frac{\theta}{z^{(\alpha+1-\theta)/\theta}}
          \\
          \frac{z^{(\alpha+1-\theta)/\theta}}{\theta} & 0
        \end{pmatrix}
        \oplus I_{\theta-1} \right) \notag \\
      & \phantom{\Psi^{(\Mei)}_-}
      \times \diag\left( \frac{z^{(\alpha+1-\theta)/\theta}}{\theta} n_{1,-}(z),-\frac{\theta}{z^{(\alpha+1-\theta)/\theta}} n_{0,-}(z), n_{2,-}(z), \dotsc, n_{\theta,-}(z) \right)^{-1} e^{nM_+(z)} \notag \\
      = {}& \Psi^{(\Mei)}_-((\rho n)^{\theta + 1} z) \left( (\rho n)^{-\frac{\theta + 1}{\theta}(\alpha + 1 - \theta)} \oplus I_{\theta} \right) N_-(z)^{-1} \left(
        \begin{pmatrix}
          0 & 1
          \\
          1 & 0
        \end{pmatrix}
        \oplus I_{\theta-1} \right)e^{nM_{+}(z)} \notag \\
      = {}& \Psi^{(\Mei)}_-((\rho n)^{\theta + 1} z) \left( (\rho n)^{-\frac{\theta + 1}{\theta}(\alpha + 1 - \theta)} \oplus I_{\theta} \right) N_-(z)^{-1} e^{nM_{-}(z)} \notag \\
      & \phantom{\Psi^{(\Mei)}_-}
      \times \left(
        \begin{pmatrix}
          0 & e^{n(m_{1,+}(z)-m_{0,-}(z))}
          \\
          e^{n(m_{0,+}(z)-m_{1,-}(z))} & 0
        \end{pmatrix}
        \oplus I_{\theta-1} \right) \notag \\
      = {}& \Psi^{(\Mei)}_-((\rho n)^{\theta + 1} z) \left( (\rho n)^{-\frac{\theta + 1}{\theta}(\alpha + 1 - \theta)} \oplus I_{\theta} \right) N_-(z)^{-1} e^{nM_{-}(z)} \left(
        \begin{pmatrix}
          0 & 1
          \\
          1 & 0
        \end{pmatrix}
        \oplus I_{\theta-1} \right),
  \end{align}
  where we have made use of \eqref{eq:gequal} and the fact that $\g_{+}(0)=\g_{-}(0)+2\pi i$ (see \eqref{eq:gpm}) in the last step. A combination of the above formula and \eqref{def:sfP0} shows that
  \begin{equation}
    \mathsf{P}^{(0)}_+(z) = \mathsf{P}^{(0)}_-(z)
    \left(\begin{pmatrix}
      0 & 1
      \\
      1 & 0
    \end{pmatrix}
  \oplus I_{\theta-1}\right), \quad z\in (0,r_n^\theta),
  \end{equation}
as required.

To show the large $n$ asymptotics of $\mathsf{P}^{(0)}$ on the boundary of the disc, we observe from Proposition \ref{prop:MNzero} that for $z \in \partial D(0,r_n^{\theta})$,
  \begin{equation}\label{eq:enMeLambda}
    e^{nM(z)}e^{-\Lambda((\rho n)^{\theta+1}z)}=I+\bigO(n^{-\frac{1}{2\theta+1}}),
  \end{equation}
  and
  \begin{multline}\label{eq:XiNinver}
    \left( (\rho n)^{-\frac{\theta + 1}{\theta}(\alpha + 1 - \theta)} \oplus I_{\theta} \right) \Xi((\rho n)^{\theta+1}z)N(z)^{-1} = \\
    \sqrt{(1+\theta)/\theta} c^{\frac{\theta-2(\alpha+1)}{2(1+\theta)}}e^{(\frac{\alpha+3/2}{1+\theta}
      -\frac{2(\alpha+1)}{\theta}+\frac12)\pi i} z^{\frac{\alpha+1}{\theta}-\frac{\alpha+3/2}{1+\theta}}(I+\bigO(n^{-\frac{2\theta}{2\theta+1}})),
  \end{multline}
  where the functions $\Lambda$ and $\Xi$ are defined in \eqref{def:Lambda} and \eqref{def:Xi}, respectively. Since $\lvert n^{\theta+1}z \rvert = n^{\frac{\theta+1}{2\theta+1}} \to +\infty $ if $ z \in \partial D(0,r_n^{\theta})$ and $n \to +\infty$, we then obtain \eqref{eq:sfP0asy} by combining \eqref{eq:G_asy_infty}, \eqref{eq:enMeLambda} and \eqref{eq:XiNinver}.
\end{proof}

With $\mathsf{P}^{(0)}$ given in \eqref{def:sfP0}, we next define
\begin{equation}\label{def:P0}
  P^{(0)}(z) = \Omega_{\pm}^{-1}\Upsilon(z)^{-1}\mathsf{P}^{(0)}(z), \quad z \in \compC_{\pm}.
\end{equation}
In view of \eqref{eq:UpsilonOmegajump}, \eqref{def:P0} and Proposition \ref{RHP:sfP0}, the following RH problem for $P^{(0)}$ is then immediate.
\begin{RHP}\label{rhp:P0}
\hfill
  \begin{enumerate}[label=\emph{(\arabic*)}, ref=(\arabic*)]
  \item
    $P^{(0)}(z)$ is analytic in $D(0,r_n^{\theta}) \setminus (\mathbb{R} \cup i \mathbb{R})$.
  \item
    For $z\in D(0,r_n^{\theta}) \cap (\mathbb{R} \cup i \mathbb{R})$, we have
    \begin{equation}
      P^{(0)}_+(z) =
      \begin{cases}
        P^{(0)}_-(z) J_U(z), & z \in (0, ir_n^{\theta}) \cup (0, -ir_n^{\theta}), \\
        \left(\begin{pmatrix}
            0 & 1
            \\
            1 & 0
          \end{pmatrix} \oplus I_{\theta - 1} \right)
        P^{(0)}_-(z) J_U(z), & z \in (0, r_n^{\theta}), \\
        \Mcyclic^{-1} P^{(0)}_-(z) J_U(z), & z \in (-r_n^{\theta}, 0),
      \end{cases}
    \end{equation}
    where $J_U(z)$ is defined in \eqref{eq:defn_J_U}.
  \item
    For $z\in \partial D(0,r_n^{\theta})$, we have, as $n\to \infty$,
    \begin{equation}\label{eq:P0asy}
      P^{(0)}(z) = I+\bigO(n^{-\frac{1}{2\theta+1}}).
    \end{equation}
  \end{enumerate}
\end{RHP}

Finally, as in the definition of $V^{(b)}$ in \eqref{eq:defn_V^(b)}, we set
\begin{equation} \label{def:V0}
  V^{(0)}(z) = U(z) P^{(0)}(z)^{-1}, \quad z \in D(0,r_n^{\theta}) \setminus (\mathbb{R} \cup i \mathbb{R}),
\end{equation}
where $U$ is the solution of the RH problem \ref{rhp:U}, and the following proposition holds.
\begin{prop} \label{rhp:V0}
The function $V^{(0)}(z)$ defined in \eqref{def:V0} has the following properties.
  \begin{enumerate}[label=\emph{(\arabic*)}, ref=(\arabic*)]
  \item \label{enu:rhp:V0:1}
    $V^{(0)}=(V^{(0)}_0, V^{(0)}_1, \dotsc, V^{(0)}_{\theta})$ is analytic in $D(0,r_n^{\theta}) \setminus \mathbb{R} $, or equivalently, its definition can be analytically extended onto $(0, ir^{\theta}_n) \cup (-ir^{\theta}_n, 0)$.
  \item \label{enu:rhp:V0:2}
    For $z\in (-r_n^\theta,r_n^{\theta})\setminus \{0\}$, we have
    \begin{equation} \label{eq:enu:rhp:V0:2}
      V^{(0)}_+(z) = V^{(0)}_-(z)
      \begin{cases}
        \begin{pmatrix}
          0 & 1 \\
          1 & 0
        \end{pmatrix}
        \oplus I_{\theta - 1}, & z \in (0, r_n^{\theta}), \\
        \Mcyclic^{-1}, & z \in (-r_n^{\theta}, 0).
      \end{cases}
    \end{equation}
  \item \label{enu:rhp:V0:3}
    For $z\in \partial D(0,r_n^{\theta})$, we have, as $n \to \infty$,
    \begin{equation}\label{eq:V0asy}
      V^{(0)}(z) = U(z)(I+\bigO(n^{-\frac{1}{2\theta+1}})).
    \end{equation}
  \item
    As $z \to 0$, we have
    \begin{equation}\label{eq:V0kzero}
      V^{(0)}_k(z)=\bigO(1), \qquad k=0,1,\ldots,\theta.
    \end{equation}
  \end{enumerate}
\end{prop}

\begin{proof}
Since it is straightforward to see items \ref{enu:rhp:V0:1}--\ref{enu:rhp:V0:3} with the aid of the RH problems for $U$ and $P^{(0)}$, it remains to check the local behaviour of $V^{(0)}$ near the origin. By \eqref{def:V0} and \eqref{def:P0}, we have
  \begin{equation}\label{eq:V0exp}
    V^{(0)}(z)=U(z)P^{(0)}(z)^{-1}=U(z)\mathsf{P}^{(0)}(z)^{-1}\Upsilon(z)\Omega_{\pm}, \quad z \in \compC_{\pm}.
  \end{equation}
It is readily seen from \eqref{def:sfP0}, item \ref{enu:prop:detPsiMei_2} of the RH problem \ref{thm:MeiG} for $\Psi^{(\Mei)}$ and Proposition \ref{prop:MNzero} that each entry of $\mathsf{P}^{(0)}(z)^{-1}$ blows up at most as a power function near the origin. Furthermore, let $\mathsf{P}^{(0)}(z)^{-1}=((\mathsf{P}^{(0)}(z)^{-1})_{kj})_{k,j=0}^\theta$, if $z \to 0$ from the sector $\arg z \in (\frac{\pi}{2}, \pi) \cup (-\pi, -\frac{\pi}{2})$, one has
\begin{equation} \label{eq:asy_P0_high}
  (\mathsf{P}^{(0)}(z)^{-1})_{0j} =  \bigO ( z^{\frac{\alpha + 3/2}{\theta + 1} - 1}), \quad (\mathsf{P}^{(0)}(z)^{-1})_{kj} =
  \begin{cases}
    \bigO (z^{\frac{\alpha + 3/2}{\theta + 1} - \frac{\alpha}{\theta} - 1}), & \alpha > 0, \\
    \bigO (z^{\frac{3/2}{\theta + 1} - 1} \log z), & \alpha = 0, \\
    \bigO (z^{\frac{\alpha + 3/2}{\theta + 1} - 1}), & \alpha \in (-1, 0),
  \end{cases}
\end{equation}
for $k=1,\ldots,\theta$, and $j=0,1,\ldots,\theta$. This, together with the local behaviour of $U$ indicated in item (3) of the RH problem \ref{rhp:U}, implies that each entry of $U(z)\mathsf{P}^{(0)}(z)^{-1}$ again blows up at most like a power function as $z\to 0$ and for all entries,
\begin{equation}\label{eq:UsfP0inv}
  U(z)\mathsf{P}^{(0)}(z)^{-1} =
  \begin{cases}
    \bigO(z^{\frac{1}{\theta} - 1}), & \alpha > 0, \\
    \bigO(z^{\frac{1}{\theta} - 1} \log z), & \alpha = 0, \\
    \bigO(z^{\frac{\alpha + 1}{\theta} - 1}), & \alpha \in (-1, 0),
  \end{cases}
\end{equation}
if $z \to 0$ with $\arg z \in (\frac{\pi}{2}, \pi) \cup (-\pi, -\frac{\pi}{2})$. Note that since both $U$ and $\mathsf{P}^{(0)}$ satisfy the same jump condition on $D(0,r_n^{\theta})\cap (\mathbb{R} \cup i \mathbb{R})$, it follows that $U(z)\mathsf{P}^{(0)}(z)^{-1}$ is analytic in $D(0,r_n^{\theta})\setminus \{0\}$. A further appeal to \eqref{eq:UsfP0inv} shows that $0$ is a removable singular point, and $U(z)\mathsf{P}^{(0)}(z)^{-1}$ is actually analytic in $D(0,r_n^{\theta})$. We then obtain \eqref{eq:V0kzero} by combining this fact with \eqref{eq:V0exp}, \eqref{def:Upsilon} and \eqref{def:Omegapm}.
\end{proof}

\subsection{Final transformation} \label{subsec:final_trans_Y}

The usual final transformation in the RH analysis consists of constructing a matrix-valued function tending to the identity matrix as $n \to \infty$. Following the same spirit, we will build a vector-valued function $R(z) = (R_1(z), R_2(z))$ with the aid of the functions $Q(z)$, $V^{(b)}(z)$ and $V^{(0)}(z)$, where $R_1$ and $R_2$ are defined in $\compC$ and $\halfH$, respectively. An essential difference here is that, instead of working on $R(z)$ directly, we will eventually show that a single complex-valued function $\tR(s)$ defined via $R(z)$ and the mapping $z=J_c(s)$ tends to 1 as $n \to \infty$.

To state the definition of $R(z)$, we first introduce some contours and domains. Recall the two open disks $D(0,r_n)$ with $r_n$ given in \eqref{def:rn} and $D(b, \epsilon)$, we set
\begin{equation}\label{def:sigmaR}
 \Sigma^{R}:=[0,b]\cup [b+\epsilon, \infty) \cup \partial D(0,r_n) \cup \partial D(b,\epsilon) \cup \Sigma_1^R \cup \Sigma_2^R,
\end{equation}
where
\begin{equation}\label{def:SigmaiR}
\Sigma_i^R:=\Sigma_i \setminus \{D(0,r_n)\cup D(b,\epsilon)\}, \qquad i=1,2;
\end{equation}
see Figure \ref{fig:Sigma_R} for an illustration. We also divide the open disk $D(0,r_n)$ into $\theta$ parts by setting
\begin{equation}\label{def:Wk}
  W_k:=\{z\in D(0,r_n)\setminus\{0\} \mid  \arg z \in (\frac{(2k - 3)\pi}{\theta}, \frac{(2k - 1)\pi}{\theta}) \},\qquad k=1,\ldots, \theta;
\end{equation}
see Figure \ref{fig:W_0_W_1_W_3} for an illustration (with $\theta=3$).
\begin{figure}[htb]
  \begin{minipage}[t]{0.55\linewidth}
    \centering
    \includegraphics{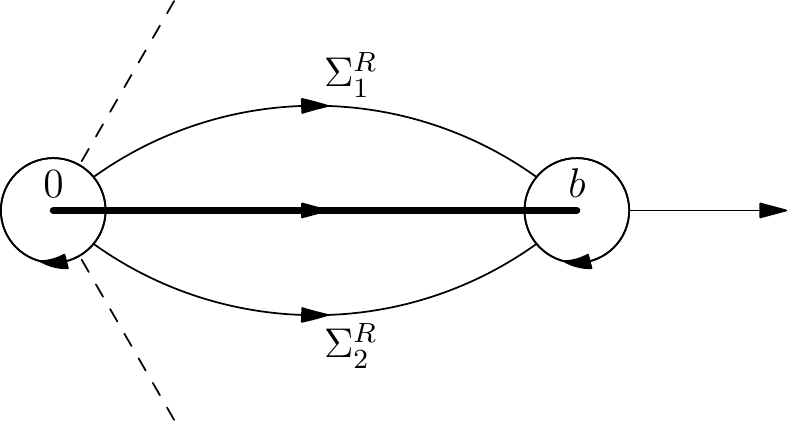}
    \caption[Contour $\Sigma^{R}$]
    {Contour $\Sigma^{R}$ }
    \label{fig:Sigma_R}
  \end{minipage}
  \hspace{\stretch{1}}
  \begin{minipage}[t]{0.4\linewidth}
  \centering
  \includegraphics{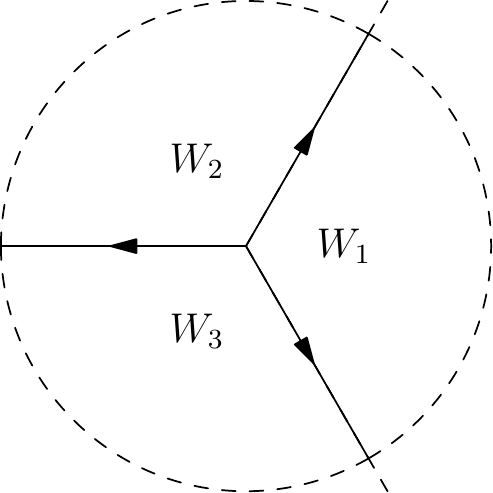}
  \caption[$W_1, \dotsc, W_{\theta}$.]{Schematic figures of $W_1, \dotsc, W_{\theta}$. ($\theta = 3$ for instance.)}
  \label{fig:W_0_W_1_W_3}
  \end{minipage}
\end{figure}
It is clear that $W_1$ can also be regarded as a subset of $\halfH$, and each $W_k$ has a outer boundary
\begin{equation}\label{def:Gammak}
  \Gamma_k := \{ z\in \partial D(0,r_n) \mid \arg z \in (\frac{(2k - 3)\pi}{\theta}, \frac{(2k - 1)\pi}{\theta}) \},\quad k=1,\ldots, \theta.
\end{equation}

We then define a $1\times 2$ vector-valued function $R(z) = (R_1(z), R_2(z))$ such that $R_1(z)$ is analytic in $\compC \setminus \Sigma^R$ and $R_2(z)$ is analytic in $\halfH \setminus \Sigma^R$, as follows:
\begin{align}
  R_1(z) = {}&
  \begin{cases}
    Q_1(z), & \hbox{$z\in \mathbb{C} \setminus \{D(b,\epsilon) \cup D(0,r_n) \cup \Sigma^R \}$,} \\
    V_1^{(b)}(z), & \hbox{$z\in D(b,\epsilon) \setminus [b-\epsilon, b]$,} \\
    V_1^{(0)}(z^\theta), & \hbox{$z\in W_1 \setminus [0,r_n]$,} \\
    V_k^{(0)}(z^\theta), & \hbox{$z\in W_k$, $k=2,\ldots,\theta$,}
  \end{cases} \label{def:R1} \\
  R_2(z) = {}&
  \begin{cases}
    Q_2(z), & \hbox{$z\in \halfH \setminus \{D(b,\epsilon) \cup D(0,r_n) \cup \Sigma^R \}$,} \\
    V_2^{(b)}(z), & \hbox{$z\in D(b,\epsilon) \setminus [b-\epsilon, b]$,} \\
    V_0^{(0)}(z^\theta), & \hbox{$z \in W_1 \setminus [0,r_n]$,} \\
  \end{cases} \label{def:R2}
\end{align}
where $\Sigma^R$ and $W_k$, $k=1,\ldots,\theta$, are defined in \eqref{def:sigmaR} and \eqref{def:Wk}, respectively. In view of the RH problems \ref{RHP:Svar}, \ref{RHP:Vb} for $Q$, $V^{(b)}$, and the RH problem for $V^{(0)}$ given in Proposition \ref{rhp:V0}, it is straightforward to check that $R$ satisfies the following RH problem.
\begin{RHP} \label{rhp:R} \hfill
  \begin{enumerate}[label=\emph{(\arabic*)}, ref=(\arabic*)]
  \item
    $R=(R_1,R_2)$ is analytic in $(\compC \setminus \Sigma^{R}, \halfH \setminus \Sigma^{R})$.

  \item \label{enu:rhp:R:2}
    $R$ satisfies the following jump conditions:
    \begin{equation} \label{eq:rhp:R:2}
      R_+(z)=R_-(z)
      \begin{cases}
        J_Q(z), & \hbox{$z\in \Sigma_1^R \cup \Sigma_2^R \cup (b+\epsilon, +\infty)$,} \\
        P^{(b)}(z), & \hbox{$z\in \partial D(b,\epsilon)$,} \\
        \begin{pmatrix}
          0 & 1 \\
          1 & 0
        \end{pmatrix}, & \hbox{$z\in (0,b) \setminus \{r_n, b-\epsilon\}$,}
      \end{cases}
    \end{equation}
    where $J_Q(z)$ and $P^{(b)}(z)$ are defined in \eqref{def:JQ} and \eqref{def:Pb}, respectively; with $\Gamma_1$ defined in \eqref{def:Gammak}, we have for $z\in\Gamma_1$,
    \begin{equation}
      R_{2,+}(z) = R_{2,-}(z)P_{00}^{(0)}(z^{\theta}) + \sum^{\theta}_{j = 1} R_{1,-}(e^{\frac{2(j - 1)\pi i}{\theta}}z) P^{(0)}_{j0}(z^{\theta}),\quad \arg z \neq 0, \pm \frac{\pi}{2\theta},\label{eq:jumpRgamma1}
    \end{equation}
    and for $k=1,\ldots,\theta$,
    \begin{equation}
    R_{1,+}(e^{\frac{2(k - 1)\pi i}{\theta}}z) = R_{2,-}(z) P^{(0)}_{0k}(z^{\theta}) + \sum^{\theta}_{j = 1} R_{1,-}(e^{\frac{2(j - 1)\pi i}{\theta}}z) P^{(0)}_{jk}(z^{\theta}), \label{eq:jumpRgammai}
    \end{equation}
    where $P^{(0)}(z)=(P^{(0)}_{jk}(z))_{j,k=0}^\theta$ defined in \eqref{def:P0} is the solution of the RH problem \ref{rhp:P0}. The orientations of the curves in $\Sigma^{R}$ are shown in Figure \ref{fig:Sigma_R}, and the orientations of the circles $\partial D(0, r_n)$ and $\partial D(b, \epsilon)$ are particularly taken in a clockwise manner.
  \item
    As $z \to \infty$ in $\compC$, $R_1$ behaves as $R_1(z) = 1 + \bigO(z^{-1})$.
  \item \label{enu:rhp:R:4}
    As $z \to \infty$ in $\halfH$, $R_2$ behaves as $R_2(z) = \bigO(1)$.
  \item \label{enu:rhp:R:5}
    As $z \to b$ or $z \to 0$, we have $R_1(z) = \bigO(1)$ and $R_2(z) = \bigO(1)$.
  \item \label{enu:rhp:R:6}
    For $x > 0$, we have the boundary condition $R_2(e^{\pi i/\theta}x) = R_2(e^{-\pi i/\theta}x)$.
   \end{enumerate}
\end{RHP}

It is readily seen from \eqref{eq:jumpRgamma1} and \eqref{eq:jumpRgammai} that the boundary value of $R$ on $\partial D(0, r_n)$, or equivalently on each of the arcs $\Gamma_i$, $i=1,\ldots,\theta$, defined in \eqref{def:Gammak} also depends on its values on the other arcs. Hence, we call these jump conditions `shifted' jump conditions, and the RH problem for $R$ a shifted RH problem, following the convention in \cite{Gakhov90}.

Similar to the idea used in the construction of global parametrix, to estimate $R$ for large $n$, we now transform the RH problem for $R$ to a scalar one on the complex $s$-plane by defining
\begin{equation} \label{eq:scalar_R_defn}
  \tR(s) =
  \begin{cases}
    R_1(J_c(s)), & \text{$s \in \compC \setminus \overline{D}$ and $s \notin I_1(\Sigma^R )$,} \\
    R_2(J_c(s)), & \text{$s \in D \setminus[-1, 0]$ and $s \notin I_2(\Sigma^R)$,}
  \end{cases}
\end{equation}
where recall that $D$ is the region bounded by the curve $\gamma=\gamma_1\cup\gamma_2$, $I_1: \compC \setminus [0, b] \to \compC \setminus \overline{D}$ and $I_2: \halfH \setminus [0, b] \to D$ are defined in \eqref{eq:inverse1} and \eqref{eq:inverse2}, respectively.

We are now at the stage of describing the RH problem for $\tR$. For that purpose, we define
\begin{equation}\label{def:SigmaRi}
  \begin{aligned}
    \SigmaR^{(1)} := {}& I_1(\Sigma^R_1 \cup \Sigma^R_2) \subseteq \compC \setminus \overline{D}, & \SigmaR^{(1')} := {}& I_2(\Sigma^R_1 \cup \Sigma^R_2) \subseteq D, \\
    \SigmaR^{(2)} := {}& I_1((b + \epsilon, +\infty)) \subseteq \compC \setminus \overline{D}, & \SigmaR^{(2')} := {}& I_2((b + \epsilon, +\infty)) \subseteq D, \\
    \SigmaR^{(3)} := {}& I_1(\partial D(b,\epsilon)) \subseteq \compC \setminus \overline{D}, & \SigmaR^{(3')} := {}& I_2(\partial D(b,\epsilon)) \subseteq D, \\
    \SigmaR^{(4)} := {}& I_2(\Gamma_1) \subseteq D, & \SigmaR^{(5)}_k := {}& I_1(\Gamma_k) \subseteq \compC \setminus \overline{D}, \qquad k = 1, \dotsc, \theta,
  \end{aligned}
\end{equation}
and set
\begin{equation}\label{def:SigmaR}
  \SigmaR = \SigmaR^{(1)} \cup \SigmaR^{(2')} \cup \SigmaR^{(3)} \cup \SigmaR^{(3')}  \cup \SigmaR^{(4)} \cup \bigcup_{k=1}^{\theta} \SigmaR^{(5)}_k ,
\end{equation}
which is the union of the solid and the dotted curves in Figure \ref{fig:jump_R_scalar}. We also define the following functions on each curve constituting $\SigmaR$:
\begin{align}
  J_{\SigmaR^{(1)}}(s) = {}& \frac{\theta P^{(\infty)}_2(z)}{z^{\alpha + 1 - \theta} P^{(\infty)}_1(z)} e^{-n\phi(z)}, & s \in {}& \SigmaR^{(1)}, \label{def:Jsigma1} \\
  \intertext{where $z = J_c(s) \in \Sigma^R_1 \cup \Sigma^R_2$,}
  J_{\SigmaR^{(2')}}(s) = {}& \frac{z^{\alpha + 1 - \theta} P^{(\infty)}_1(z)}{\theta P^{(\infty)}_2(z)} e^{n\phi(z)}, & s \in {}& \SigmaR^{(2')}, \label{def:Jsigma2'} \\
  \intertext{where $z = J_c(s) \in (b + \epsilon, +\infty)$,}
  J^{1}_{\SigmaR^{(3)}}(s) = {}& P^{(b)}_{11}(z) - 1, \quad J^{2}_{\SigmaR^{(3)}}(s) = P^{(b)}_{21}(z), & s \in {}& \SigmaR^{(3)}, \label{def:Jsigma3} \\
  \intertext{where $z = J_c(s) \in \partial D(b,\epsilon)$,}
  J^{1}_{\SigmaR^{(3')}}(s) = {}& P^{(b)}_{22}(z) - 1, \quad J^{2}_{\SigmaR^{(3')}}(s) = P^{(b)}_{12}(z), & s \in {}& \SigmaR^{(3')},
  \label{def:Jsigma32}\\
  \intertext{where $z = J_c(s) \in \partial D(b,\epsilon)$,}
  J^{0}_{\SigmaR^{(4)}}(s) = {}& P^{(0)}_{00}(z^{\theta}) - 1, \quad J^{j}_{\SigmaR^{(4)}}(s) = P^{(0)}_{j0}(z^{\theta}), & s \in {}& \SigmaR^{(4)}, \\
  \intertext{where $z = J_c(s) \in \Gamma_1$ and $j=1,\ldots,\theta$,}
  J^k_{\SigmaR^{(5)}_k}(s) = {}& P^{(0)}_{kk}(z^{\theta}) - 1, \quad J^j_{\SigmaR^{(5)}_k}(s) = P^{(0)}_{jk}(z^{\theta}), \quad j\neq k, & s \in {}& \SigmaR^{(5)}_k, \label{def:Jsigma5}
\end{align}
where $z = J_c(s) \in \Gamma_k$, $k=1,\ldots,\theta$, and $j=0,1,\ldots,\theta$. In \eqref{def:Jsigma3} and \eqref{def:Jsigma32}, $P^{(b)}(z)=(P^{(b)}_{jk}(z))_{j,k=1}^2$ is defined in \eqref{def:Pb}. With the aid of these functions, we further define an operator
$\DeltaR$ that acts on functions defined on $\SigmaR$ by
\begin{equation}\label{def:DeltaR}
  \DeltaR f(s) =
  \begin{cases}
    J_{\SigmaR^{(1)}}(s) f(\s), & \text{$s \in \SigmaR^{(1)}$ and $\s = I_2(J_c(s))$}, \\
    J_{\SigmaR^{(2')}}(s) f(\s), & \text{$s \in \SigmaR^{(2')}$ and $\s = I_1(J_c(s))$}, \\
    J^{1}_{\SigmaR^{(3)}}(s) f(s) + J^{2}_{\SigmaR^{(3)}}(s) f(\s), & \text{$s \in \SigmaR^{(3)}$ and $\s = I_2(J_c(s))$}, \\
    J^{1}_{\SigmaR^{(3')}}(s) f(s) + J^{2}_{\SigmaR^{(3')}}(s) f(\s), & \text{$s \in \SigmaR^{(3')}$ and $\s = I_1(J_c(s))$}, \\
    J^{0}_{\SigmaR^{(4)}}(s) f(s) + \sum\limits ^{\theta}_{j = 1} J^{j}_{\SigmaR^{(4)}}(s) f(\s_j), & \text{$s \in \SigmaR^{(4)}$ and $\s_j = I_1(J_c(s) e^{\frac{2(j - 1)\pi i}{\theta}}) \in I_1(\Gamma_j)$}, \\
    \sum\limits^{\theta}_{j = 0} J^j_{\SigmaR^{(5)}_k}(s) f(\s_j), & \text{$s \in \SigmaR^{(5)}_k$ and  $\s_0 = I_2(J_c(s) e^{\frac{2(1 - k)\pi i}{\theta}}) \in D$,} \\
    & \text{$\s_j = I_1(J_c(s) e^{\frac{2(j - k)\pi i}{\theta}}) \in I_1(\Gamma_j) \subseteq \compC \setminus \overline{D}$} \\
    & \text{for $j = 1, \dotsc, \theta$,  such that $\s_k = s$},
  \end{cases}
\end{equation}
where $f$ is a complex-valued function defined on $\SigmaR$. Hence, we can define a scalar shifted RH problem as follows.

\begin{figure}[htb]
  \centering
  \includegraphics{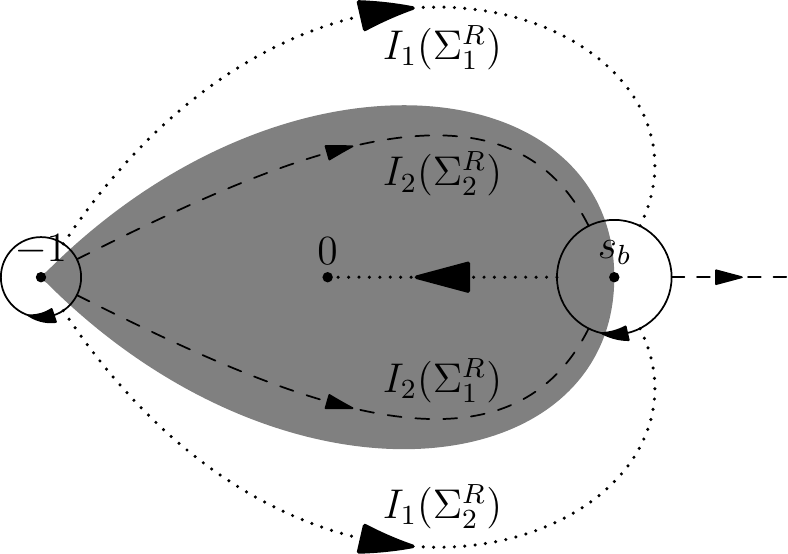}
  \caption{The solid and the dotted curves are contour $\SigmaR$ for the RH problem for $\tR$. (The solid and dashed curves are $\omega(\widetilde{\SigmaR})$ where $\widetilde{\SigmaR}$ is for the RH problem for $\widetilde{\tR}$.)}
  \label{fig:jump_R_scalar}
\end{figure}

\begin{RHP} \label{rhp:tR} \hfill
  \begin{enumerate}[label=\emph{(\arabic*)}, ref=(\arabic*)]
  \item
    $\tR(s)$ is analytic in $\compC \setminus \SigmaR$, where the contour $\SigmaR$ is defined in \eqref{def:SigmaR}.
  \item \label{enu:rhp:tR:2}
    For $s\in \SigmaR$, we have
    \begin{equation}
      \tR_+(s) - \tR_-(s) = \DeltaR \tR_-(s),
    \end{equation}
    where $\DeltaR$ is the operator defined in \eqref{def:DeltaR}.
  \item
    As $s \to \infty$, we have
    \begin{equation}
      \tR(s)=1+\bigO(s^{-1}).
    \end{equation}
  \item \label{enu:rhp:tR:4}
    As $s \to 0$, we have $\tR(s)=\bigO(1)$.
  \end{enumerate}
\end{RHP}

We then have the following proposition about the solution of the above RH problem.
\begin{prop} \label{prop:uniquess}
The function $\tR(s)$ defined in \eqref{eq:scalar_R_defn} is the unique solution of the RH problem \ref{rhp:tR} after trivial analytical extension.
\end{prop}
\begin{proof}
 First, we assume that $\tR$ is defined by \eqref{eq:scalar_R_defn}. It is clear that $\tR$ is analytic in $\compC \setminus \{[-1,0]\cup I_1(\Sigma^R) \cup I_2(\Sigma^R)\}$. It comes out that the jumps of $\tR$ on some of these curves are trivial. If $s\in(-1,0)$, the image of $J_{c,\pm}(s)$ is the ray $re^{\mp\frac{\pi i}{\theta}}$ with $r>0$; cf. Figure \ref{fig:mapJ}. Hence, by item \ref{enu:RHP:Svar:7} of the RH problem \ref{RHP:Svar} for $Q$ and the fact that $V^{(0)}_{0,+}(x)=V^{(0)}_{0,-}(x)$ for $x\in(-r_n^\theta,0)$ (see \eqref{eq:enu:rhp:V0:2}), we obtain from \eqref{def:R2} and \eqref{eq:scalar_R_defn} that $\tR_+(s)=\tR_-(s)$ for $s\in(-1,0)$. Similarly, by \eqref{eq:rhp:R:2}, it is straightforward to check that $\tR_+(s)=\tR_-(s)$ for $s\in \gamma \cup \SigmaR^{(1')} \cup \SigmaR^{(2)} \setminus\{\-1,s_b\}$, where the curves $\SigmaR^{(1')}$ and $\SigmaR^{(2)}$ are defined in \eqref{def:SigmaRi}. Hence, $\tR$ is analytic in $\compC \setminus (\SigmaR \cup \{ -1, s_b\})$. Since $J_{c}(-1)=0$ and $J_c(s_b)=b$, we further conclude from item (5) of the RH problem \ref{rhp:R} for $R$ that $-1$ and $s_b$ are removable singular points, and $\tR$ defined by \eqref{eq:scalar_R_defn} is actually analytic in $\compC \setminus \SigmaR$ after trivial analytical extension. The remaining items \ref{enu:rhp:tR:2}--\ref{enu:rhp:tR:4} of the RH problem \ref{rhp:tR} follow directly by combining \eqref{eq:scalar_R_defn} and items \ref{enu:rhp:R:2}--\ref{enu:rhp:R:4} of the RH problem \ref{rhp:R} for $R$. Thus, we have proved that the function $\tR(s)$ defined in \eqref{eq:scalar_R_defn} solves the RH problem \ref{rhp:tR}.

To show the uniqueness of the RH problem \ref{rhp:tR}, let $\tRhat(s)$ be a solution in place of $\tR$. By reversing the transformations $S \to Q \to R \to \tR$, we obtain, after some trivial analytic continuation, a vector-valued function $\widehat{S}(z) = (\widehat{S}_1(z), \widehat{S}_2(z))$ defined in $(\compC \setminus \Sigma, \mathbb{H}_{\theta} \setminus \Sigma)$ from $\widehat\tR(s)$. It can be checked that $\widehat{S}(z)$ satisfies the RH problem \ref{rhp:S} for $S$ but with item \ref{enu:rhp:S:6} therein replaced by $S_1(z) = \bigO((z - b)^{-1/4})$ and $S_2(z) = \bigO((z - b)^{-1/4})$. In view of Proposition \ref{prop:unique_S}, $\Shat$ still solves the RH problem \ref{rhp:S} for $S$ uniquely. Since all the transforms $S \to Q \to R \to \tR$ are reversible, we have that the RH problem \ref{rhp:tR} has a unique solution, which is given in \eqref{eq:scalar_R_defn}.
\end{proof}

Finally, we show that $\tR(s) \to 1$ as $n\to\infty$, using the strategy proposed in \cite{Claeys-Wang11}. We start with the claim that $\tR$ satisfies the integral equation
\begin{equation} \label{eq:constr_R}
  \tR(s) = 1 + \mathcal{C}(\DeltaR \tR_-)(s),
\end{equation}
where
\begin{equation}
  \mathcal{C}g(s) = \frac{1}{2\pi i} \int_{\SigmaR} \frac{g(\xi)}{\xi - s} d\xi,  \qquad s\in\compC \setminus \SigmaR,
\end{equation}
is the Cauchy transform of a function $g$. Indeed, on account of Proposition \ref{prop:uniquess}, it suffices to show that the right-hand side of \eqref{eq:constr_R} satisfies the RH problem for $\tR$ and the verification is straightforward. As a consequence of \eqref{eq:constr_R}, it is readily seen that
\begin{equation}\label{eq:tRs-1split}
  \tR(s) - 1 = \frac{1}{2\pi i} \int_{\SigmaR} \frac{\DeltaR(\tR_- - 1)(\xi)}{\xi - s} d\xi + \frac{1}{2\pi i} \int_{\SigmaR} \frac{\DeltaR(1)(\xi)}{\xi - s} d\xi,\qquad s\in\compC \setminus \SigmaR.
\end{equation}
To estimate the two terms on the right-hand side of the above formula, we need the following estimate of the operator $\DeltaR$.
\begin{prop}\label{prop:estdletaR}
Let $\DeltaR: L^2(\SigmaR) \to L^2(\SigmaR)$ be the operator defined in \eqref{def:DeltaR}. There exits a constant $M_{\SigmaR}>0$ such that
\begin{equation}\label{eq:estOperator}
\lVert \DeltaR \rVert_{L^2(\SigmaR)} \leq M_{\SigmaR}n^{-\frac{1}{2\theta + 1}},
\end{equation}
if $n$ is large enough.
\end{prop}
\begin{proof}
  From \eqref{def:DeltaR}, it follows that
    \begin{align}\label{eq:estSigmaR}
      &\lVert \DeltaR f \rVert_{L^2(\SigmaR)} \nonumber \\
      &\leq \lVert J_{\SigmaR^{(1)}}(s) f(I_2(J_c(s))) \rVert_{L^2(\SigmaR^{(1)})}+ \lVert J_{\SigmaR^{(2')}}(s) f(I_1(J_c(s))) \rVert_{L^2(\SigmaR^{(2')})} \nonumber \\
      &~~~+\lVert J^{1}_{\SigmaR^{(3)}}(s) f(s) + J^{2}_{\SigmaR^{(3)}}(s) f(I_2(J_c(s))) \rVert_{L^2(\SigmaR^{(3)})} \nonumber \\
      &~~~+\lVert J^{1}_{\SigmaR^{(3')}}(s) f(s) + J^{2}_{\SigmaR^{(3')}}(s) f(I_1(J_c(s))) \rVert_{L^2(\SigmaR^{(3')})} \nonumber \\
      &~~~ +\lVert J^{0}_{\SigmaR^{(4)}}(s) f(s) + \sum\limits ^{\theta}_{j = 1} J^{j}_{\SigmaR^{(4)}}(s) f(I_1(J_c(s)e^{\frac{2(j-1)\pi i}{\theta}})) \rVert_{L^2(\SigmaR^{(4)})} \nonumber\\
      &~~~ +\sum\limits^{\theta}_{k = 1} \lVert J^{0}_{\SigmaR^{(5)}_k}(s) f(I_2(J_c(s)e^{\frac{2(1-k)\pi i}{\theta}}))+\sum\limits^{\theta}_{j = 1} J^j_{\SigmaR^{(5)}_k}(s) f(I_1(J_c(s)e^{\frac{2(j-k)\pi i}{\theta}})) \rVert_{L^2(\SigmaR^{(5)}_k)} \nonumber\\
      &\leq \max_{s \in \SigmaR^{(1)}} \lvert J_{\SigmaR^{(1)}}(s) \rvert \cdot \lVert f(I_2(J_c(s))) \rVert_{L^2(\SigmaR^{(1)})} + \max_{s \in \SigmaR^{(2')}} \lvert J_{\SigmaR^{(2')}}(s) \rvert \cdot \lVert f(I_1(J_c(s))) \rVert_{L^2(\SigmaR^{(2')})} \nonumber\\
      &~~~ + \max_{s \in \SigmaR^{(3)}} \lvert J^{1}_{\SigmaR^{(3)}}(s) \rvert \cdot \lVert f(s) \rVert_{L^2(\SigmaR^{(3)})} + \max_{s \in \SigmaR^{(3)}} \lvert J^{2}_{\SigmaR^{(3)}}(s) \rvert \cdot \lVert f(I_2(J_c(s))) \rVert_{L^2(\SigmaR^{(3)})} \nonumber\\
      &~~~ + \max_{s \in \SigmaR^{(3')}} \lvert J^{1}_{\SigmaR^{(3')}}(s) \rvert \cdot \lVert f(s) \rVert_{L^2(\SigmaR^{(3')})} + \max_{s \in \SigmaR^{(3')}} \lvert J^{2}_{\SigmaR^{(3')}}(s) \rvert \cdot \lVert f(I_1(J_c(s))) \rVert_{L^2(\SigmaR^{(3')})} \nonumber\\
      &~~~ + \max_{s \in \SigmaR^{(4)}} \lvert J^{0}_{\SigmaR^{(4)}}(s) \rvert \cdot \lVert f(s) \rVert_{L^2(\SigmaR^{(4)})} + \sum^{\theta}_{j = 1} \max_{s \in \SigmaR^{(4)}} \lvert J^{j}_{\SigmaR^{(4)}}(s) \rvert \cdot \lVert f(I_1(J_c(s) e^{\frac{2(j - 1)\pi i}{\theta}})) \rVert_{L^2(\SigmaR^{(4)})} \nonumber\\
      &~~~ + \sum^{\theta}_{k = 1} \left( \max_{s \in \SigmaR^{(5)}_k} \lvert J^{0}_{\SigmaR^{(5)}_k}(s) \rvert \cdot \lVert f(I_2(J_c(s) e^{\frac{2(1 - k)\pi i}{\theta}})) \rVert_{L^2(\SigmaR^{(5)}_k)}
        \vphantom{\sum^{\theta}_{j = 1} \max_{s \in \SigmaR^{(6)}(k)} \lvert J^{j}_{\SigmaR^{(6)}(k)}(s) \rvert \cdot \lVert f(I_1(z e^{2(j - k)\pi i/\theta})) \rVert_{L^2(\SigmaR^{(6)}(k))}}
      \right. \nonumber \\
      &~~~   \phantom{\smash{+ \sum^{\theta}_{k = 1}}\max_{s \in \SigmaR^{(5)}_k} \lvert J^{0}_{\SigmaR^{(5)}_k}(s) \rvert}
      +\left. \sum^{\theta}_{j = 1} \max_{s \in \SigmaR^{(5)}_k} \lvert J^{j}_{\SigmaR^{(5)}_k}(s) \rvert \cdot \lVert f(I_1(J_c(s) e^{\frac{2(j - k)\pi i}{\theta}})) \rVert_{L^2(\SigmaR^{(5)}_k)} \right).
    \end{align}

We next estimate the factors $\max |\cdot|$ in \eqref{eq:estSigmaR} for large $n$. From \eqref{def:Jsigma3}--\eqref{def:Jsigma5}, \eqref{eq:P0asy} and item \ref{enu:Pbbnd} of the RH problem \ref{prop:Pb} for $P^{(b)}$, it is readily seen that
  \begin{align}
    \max_{s \in \SigmaR^{(5)}_k} \lvert J^{j}_{\SigmaR^{(5)}_k}(s) \rvert = {}& \bigO(n^{-\frac{1}{2\theta + 1}}), & j= {}& 0,1,\ldots,\theta, \quad k = 1, \dotsc, \theta, \label{eq:estsigmaR5} \\
    \max_{s \in \SigmaR^{(4)}} \lvert J^{j}_{\SigmaR^{(4)}}(s) \rvert = {}& \bigO(n^{-\frac{1}{2\theta + 1}}), & j= {}& 0,1,\ldots,\theta, \\
    \max_{s \in \SigmaR^{(i)}} \lvert J^{j}_{\SigmaR^{(i)}}(s) \rvert  = {}& \bigO(n^{-1}), & j= {}& 1,2, \quad i=3,3', \label{def:estsigmaR3}
  \end{align}
if $n$ is large enough. To estimate $\max_{s \in \SigmaR^{(1)}} \lvert J_{\SigmaR^{(1)}}(s) \rvert$ and $\max_{s \in \SigmaR^{(2')}} \lvert J_{\SigmaR^{(2')}}(s) \rvert$, we observe from
\eqref{eq:localb} and \eqref{eq:phiderivative} that
\begin{equation} \label{eq:constr_Sigma}
  \lvert e^{-n\phi(z)} \rvert \leq e^{-c'n \lvert z \rvert^{\frac{\theta}{1+\theta}}}, \qquad z \in \Sigma_1^R \cup  \Sigma_2^R,
\end{equation}
and from \eqref{def:phi} and \eqref{eq:gequal2} that
  \begin{equation}
    \phi (z) < -c'', \qquad  z\in (b+\epsilon,+\infty),
  \end{equation}
where $\Sigma_i^R$, $i=1,2$, is defined in \eqref{def:SigmaiR} and $c',c''$ are two positive constants. These estimates, together with \eqref{def:Jsigma1} and \eqref{def:Jsigma2'}, shows that
  \begin{equation}\label{eq:estsigmaR1}
    \max_{s \in \SigmaR^{(1)}} \lvert J_{\SigmaR^{(1)}}(s) \rvert= \bigO(e^{-c'n^{\frac{1}{2\theta+1}}}),\qquad  \max_{s \in \SigmaR^{(2')}} \lvert J_{\SigmaR^{(2')}}(s) \rvert= \bigO (e^{-c''n}),
  \end{equation}
for large $n$.

Finally, a combination of \eqref{eq:estsigmaR1}, \eqref{eq:estsigmaR5}--\eqref{def:estsigmaR3} and \eqref{eq:estSigmaR} gives us
  \begin{equation}
    \lVert \DeltaR f \rVert_{L^2(\SigmaR)} \leq M_{\SigmaR}n^{-\frac{1}{2\theta + 1}} \lVert f \rVert_{L^2(\SigmaR)},
  \end{equation}
for some positive constant $M_{\SigmaR}$ and large $n$, which implies \eqref{eq:estOperator}.
\end{proof}

By taking the limit where $s$ approaches the minus side of $\SigmaR$, we obtain from \eqref{eq:tRs-1split} that
\begin{equation}\label{eq:tR-}
  \tR_-(s) - 1 = \mathcal{C}_{\DeltaR}(\tR_- - 1)(s) + \mathcal{C}_-(\DeltaR(1))(s),
\end{equation}
where
\begin{equation}
  \mathcal{C}_{\DeltaR}f(s) = \mathcal{C}_-(\DeltaR(f))(s), \qquad   \mathcal{C}_-g(s) = \frac{1}{2\pi i} \lim_{s' \to s_-} \int_{\SigmaR} \frac{g(\xi)}{\xi - s'} d\xi,
\end{equation}
and the limit $s' \to s_-$ is taken when approaching the contour from the minus side. Since the Cauchy operator $\mathcal{C}_-$ is bounded, we see from Proposition \ref{prop:estdletaR} that the operator norm of $\mathcal{C}_{\DeltaR}$ is also uniformly $\bigO(n^{-\frac{1}{2\theta + 1}})$ as $n\to \infty$. Hence, if $n$ is large enough, the operator $1-\mathcal{C}_{\DeltaR}$ is invertible, and we could rewrite \eqref{eq:tR-} as
\begin{equation}
\tR_-(s) - 1 = (1 - \mathcal{C}_{\DeltaR})^{-1}( \mathcal{C}_-(\DeltaR(1)))(s).
\end{equation}
As one can check directly that
\begin{equation}\label{eq:DeltaR1norm}
  \lVert \DeltaR(1) \rVert_{L^2(\SigmaR)} = \bigO(n^{-\frac{\theta + 1}{2\theta + 1}}),
\end{equation}
combining the above two formulas gives us
\begin{equation}\label{eq:tR-1norm}
  \lVert \tR_- - 1 \rVert_{L^2(\SigmaR)} = \bigO(n^{-\frac{\theta + 1}{2\theta + 1}}).
\end{equation}

By \eqref{def:SigmaRi}, \eqref{eq:I1zero} and \eqref{eq:I2zero}, it is readily seen that the radius of the circle $\SigmaR^{(4)} \cup \bigcup_{k=1}^{\theta} \SigmaR^{(5)}_k$ around $-1$ is of the order $r_n^{\theta/(1+\theta)}$. Suppose $\lvert s+1 \rvert < \epsilon r_n^{\theta/(1+\theta)}$, where $\epsilon>0$ is small enough so that $s$ is within the circle around $-1$, we then obtain from \eqref{eq:tRs-1split}, \eqref{eq:DeltaR1norm}, \eqref{eq:tR-1norm} and the Cauchy-Schwarz inequality that
  \begin{equation}
    \begin{split}
      \lvert \tR(s) - 1 \rvert & \leq  \frac{1}{2\pi} \left( \lVert \DeltaR(\tR_- - 1) \rVert_{L^2(\SigmaR)} + \lVert \DeltaR(1) \rVert_{L^2(\SigmaR)} \right) \cdot \lVert \frac{1}{\xi - s} \rVert_{L^2(\SigmaR)} \\
      &\leq  \frac{1}{2\pi} \left( \bigO(n^{-\frac{\theta + 2}{2\theta + 1}}) + \bigO(n^{-\frac{\theta + 1}{2\theta + 1}}) \right) \cdot \bigO(n^{\frac{\theta}{2\theta + 1}})= \bigO(n^{-\frac{1}{2\theta + 1}}).
    \end{split}
  \end{equation}
It is easily seen that the above estimate is also valid for $\lvert s \rvert < \epsilon$. As a consequence, we conclude the following lemma.
\begin{lemma}\label{lem:tRest}
As $n\to \infty$, we have
\begin{equation}\label{eq:esttR}
\tR(s)=1+\bigO(n^{-\frac{1}{2\theta + 1}}),
\end{equation}
uniformly for $\lvert s+1 \rvert < \epsilon r_n^{\theta/(1+\theta)}$ or $\lvert s \rvert < \epsilon$, where $\epsilon$ is a small positive constant and $r_n$ is given in \eqref{def:rn}.
\end{lemma}

We note that the estimate \eqref{eq:esttR} actually holds uniformly for $s\in \compC \setminus \SigmaR$, but do not pursue this generalization here.


\section{Asymptotic analysis of the RH problem for $\Y$}
\label{sec:asytildeY}

We now perform the asymptotic analysis of the RH problem \ref{RHP:original_q} for $\Y$, which characterizes the polynomials $q_n$. Again, the analysis involves a series of explicit, invertible transformations, and end up with a small-norm shifted RH problem. To emphasize the analogies with the previous section, we will use the notations $\T,\St,\ldots$, for the counter parts of the functions $T,S,\ldots$, used before.

\subsection{First transformation: $\Y \to \T$}
Recall the $g$-functions given in \eqref{def:g} and \eqref{def:tildeg}, we define
\begin{equation}\label{def:tildeT}
  \T(z) = e^{-\frac{n\ell}{2}} \Y(z)
  \begin{pmatrix}
    e^{-n\widetilde g(z)} & 0
    \\
    0 & e^{n g(z)}
  \end{pmatrix}
  e^{\frac{n\ell}{2} \sigma_3} = (\Y_1(z) e^{-n \g(z)}, \Y_2(z) e^{n (g(z) - \ell)}),
\end{equation}
where the constant $\ell$ is the same as that in \eqref{eq:EL1}. It is then readily seen that $\T$ satisfies the following RH problem.
\begin{RHP}\label{rhp:tildeT}
\hfill
  \begin{enumerate}[label=\emph{(\arabic*)}, ref=(\arabic*)]
  \item
    $ \T = ( \T_1,  \T_2)$ is analytic in $(\mathbb{H}_\theta \setminus [0, b], \mathbb{C} \setminus [0, \infty))$.
  \item For $x>0$, we have
       \begin{equation}\label{eq:tildeTjump}
       \T_+(x) = \T_-(x) J_{\T}(x),
       \end{equation}
       where
       \begin{equation}
       J_{\T}(x)=
       \begin{pmatrix}
        e^{n(\g_-(x) - \g_+(x))} & x^{\alpha} e^{n( g_+(x)+ \g_-(x) - V(x) - \ell)} \\
        0 & e^{n(g_+(x) - g_-(x))}
      \end{pmatrix}.
    \end{equation}
  \item
    As $z \to \infty$ in $\mathbb{H}_{\theta}$, $\T_1$ behaves as $\T_1(z) = 1 + \bigO(z^{-\theta})$.
  \item
    As $z \to \infty$ in $\compC$, $\T_2$ behaves as $\T_2(z) = \bigO(z^{-1})$.
  \item
    As $z \to 0$ in $\compC$ or $\mathbb{H}_{\theta}$, $\T$ has the same behaviour as $\Y$.
\item
    As $z \to b$, we have $\T_1(z)=\bigO(1)$ and $\T_2(z)=\bigO(1)$.
  \item
    For $x > 0$, we have the boundary condition $\T_1(e^{\pi i/\theta}x) = \T_1(e^{-\pi i/\theta}x)$.
  \end{enumerate}
\end{RHP}

\subsection{Second transformation: $\T \to \St$}
From \eqref{eq:gequal}, we have, like \eqref{eq:decomp_J_T}, the following factorization of $J_{\T}$:
\begin{equation}
 J_{\T}(x)=
 \begin{pmatrix}
 1 & 0
 \\
 x^{-\alpha}e^{-n\phi_-(x)} & 1
 \end{pmatrix}
 \begin{pmatrix}
 0 & x^{\alpha}
 \\
 -x^{-\alpha} & 0
 \end{pmatrix}
\begin{pmatrix}
 1 & 0
 \\
 x^{-\alpha}e^{-n\phi_+(x)} & 1
 \end{pmatrix}, \qquad x\in(0,b),
\end{equation}
where $\phi$ is defined in \eqref{def:phi}. By opening lens around $(0,b)$ as shown in Figure \ref{fig:Sigma}, we define
\begin{equation}\label{def:tildeS}
  \St(z) =
  \begin{cases}
    \T(z), & \text{$z$ outside the lens}, \\
    \T(z)
    \begin{pmatrix}
      1 & 0
      \\
   z^{-\alpha} e^{-n \phi(z)} & 1
    \end{pmatrix},
    & \text{$z$ in the lower part of the lens}, \\
    \T(z)
    \begin{pmatrix}
      1 & 0 \\
      -z^{-\alpha} e^{-n \phi(z)} & 1
    \end{pmatrix},
    & \text{$z$ in the upper part of the lens}.
  \end{cases}
\end{equation}
We then have that $\St$ solves the following RH problem.

\begin{RHP} \label{rhp:tildeS}
\hfill
  \begin{enumerate}[label=\emph{(\arabic*)}, ref=(\arabic*)]
  \item
    $\St = (\St_1, \St_2)$ is analytic in $(\mathbb{H}_\theta \setminus \Sigma, \mathbb{C} \setminus \Sigma)$, where the contour $\Sigma$ is defined in \eqref{def:Sigma}.
  \item
   For $z \in \Sigma$, we have
    \begin{equation}
     \St_+(z) = \St_-(z) J_{\St}(z),
      \end{equation}
      where
      \begin{equation} \label{def:Jtildes}
      J_{\St}(z) =
      \begin{cases}
        \begin{pmatrix}
          1 & 0 \\
         z^{-\alpha} e^{-n \phi(z)} & 1
        \end{pmatrix},
        & \text{$z \in \Sigma_1 \cup \Sigma_2$}, \\
        \begin{pmatrix}
          0 & z^{\alpha} \\
          -z^{-\alpha} & 0
        \end{pmatrix},
        & \text{$z \in (0, b)$}, \\
        \begin{pmatrix}
          1 & z^{\alpha} e^{n \phi(z)} \\
          0 & 1
        \end{pmatrix},
        & \text{$z \in (b, +\infty)$}.
      \end{cases}
    \end{equation}

    \item As $z \to \infty$ in $\compC$ or $\halfH$, $\St$ has the same behaviour as $\T$.

    \item
    As $z \to 0$ in $\mathbb{H}_\theta \setminus \Sigma$, we have
    \begin{equation} \label{eq:tildeasy_S_1_in_lens}
      \St_1(z) =
      \begin{cases}
        \bigO(z^{ - \alpha }), & \text{$\alpha  > 0$ and $z$ inside the lens,} \\
        \bigO(\log z), & \text{$\alpha  = 0$ and $z$ inside the lens,} \\
        \bigO(1), & \text{$z$ outside the lens or $\alpha < 0$.}
      \end{cases}
    \end{equation}

   \item As $z \to 0$ in $\compC \setminus \Sigma$, we have
    \begin{equation}
      \St_2(z) =
      \begin{cases}
        \bigO(1), & \alpha > 0, \\
        \bigO(\log z), & \alpha=0, \\
        \bigO(z^{\alpha}), & \alpha < 0.
      \end{cases}
    \end{equation}

   \item
   As $z \to b$,  we have $\St_1(z) = \bigO(1)$ and $\St_2(z) = \bigO(1)$.

    \item
    For $x>0$, we have the boundary condition $\St_1(e^{\pi i/\theta}x) = \St_1(e^{-\pi i/\theta}x)$.
  \end{enumerate}
\end{RHP}

Analogous to Proposition \ref{prop:unique_S}, we have the following proposition. Since the proof is similar, we omit the details here.
\begin{prop} \label{prop:unique_tildeS}
The solution of RH problem \ref{rhp:tildeS} is unique, even if item \ref{enu:rhp:S:6} therein is replaced by the weaker condition that $\St_1(z) = \bigO((z-b)^{-1/4})$ and $\St_2(z) = \bigO((z-b)^{-1/4})$ as $z \to b$.
\end{prop}

\subsection{Construction of the global parametrix}
Like the RH problem \ref{rhp:S} for $S$, it is easily seen that the jump matrix $J_{\St}$ of $\St$ tends to the identity matrix uniformly as $n\to \infty$, except for $z$ in a small neighbourhood around $[0,b]$. This leads us to consider the following RH problem.
\begin{RHP}\label{rhp:global2}
\hfill
  \begin{enumerate}[label=\emph{(\arabic*)}, ref=(\arabic*)]
  \item
    $\Pinfty = (\Pinfty_1, \Pinfty_2)$ is analytic in $(\mathbb{H}_{\theta} \setminus [0,b], \compC \setminus [0,b])$.
  \item
    For $x \in (0,b)$, we have
    \begin{equation}
      \Pinfty_+(x) = \Pinfty_-(x)
      \begin{pmatrix}
          0 & x^{\alpha}
          \\
          -x^{-\alpha} & 0
      \end{pmatrix}.
      \end{equation}
  \item
    As $z \to \infty$ in $\mathbb{H}_{\theta}$, $\Pinfty_1$ behaves as $\Pinfty_1(z) = 1 + \bigO(z^{-\theta})$.
  \item
    As $z \to \infty$ in $\compC$, $\Pinfty_2$ behaves as  $\Pinfty_2(z) = \bigO(z^{-1})$.

   \item \label{enu:tildepinftyboundary}
   For $x>0$, we have the boundary condition $
\widetilde P_1^{(\infty)}(e^{\pi i/\theta}x) = \widetilde P_1^{(\infty)}(e^{-\pi i/\theta}x)$.
  \end{enumerate}
\end{RHP}

As in the construction of $P^{(\infty)}$ in Section \ref{sec:global}, we will solve the above RH problem with the aid of the function
\begin{equation}\label{def:tildeF}
\widetilde {\tP} (s):=\left\{
        \begin{array}{ll}
          \widetilde P_2^{(\infty)}(J_c(s)), & \hbox{$s\in \mathbb{C}\setminus \overline{D}$,} \\
          \widetilde P_1^{(\infty)}(J_c(s)), & \hbox{$s\in D \setminus [-1,0]$,}
        \end{array}
      \right.
\end{equation}
where recall that $J_c$ is defined in \eqref{def:Jcs} and $D$ is the region bounded by the curves $\gamma_1$ and $\gamma_2$. Since the function $\widetilde P_1^{(\infty)}$ satisfies the boundary condition indicated in item \ref{enu:tildepinftyboundary} of the RH problem \ref{rhp:global2}, it is readily seen from the mapping properties of $J_c$ that $\widetilde \tP$ satisfies the following RH problem.

\begin{RHP}
\hfill
  \begin{enumerate}[label=\emph{(\arabic*)}, ref=(\arabic*)]
  \item
    $\widetilde \tP$ is analytic in $\mathbb{C}\setminus \gamma$.
  \item
    For $s \in \gamma$, we have
    \begin{equation}
     \widetilde \tP_+(s)= \widetilde \tP_-(s) J_{\widetilde \tP}(s),
\end{equation}
where
\begin{equation}
 J_{\widetilde \tP}(s)=\left\{
        \begin{array}{ll}
          c^\alpha (s+1)^{\alpha} \left(\frac{s+1}{s}\right)^{\frac{\alpha}{\theta}}, & \hbox{$s\in\gamma_1$,} \\
          -c^{-\alpha} (s+1)^{-\alpha} \left(\frac{s+1}{s}\right)^{-\frac{\alpha}{\theta}}, & \hbox{$s\in\gamma_2$,}
        \end{array}
      \right.
    \end{equation}
with $c$ being the constant given in \eqref{def:Jcs}
  \item  As $s\to\infty$, we have $\widetilde \tP(s)=\mathcal{O}(s^{-1})$.
  \item  As $s \to 0 $, we have $\widetilde \tP(s)=1+ \mathcal{O}(s)$.
  \end{enumerate}
\end{RHP}

A solution $\widetilde \tP$ to the above RH problem is explicitly given by
\begin{equation}\label{eq:tildeFs}
  \widetilde \tP(s)=
  \begin{cases}
    \frac{c^{\alpha}\sqrt{s_b}i}{\sqrt{(s + 1)(s - s_b)}} \left( \frac{s+1}{s} \right)^{\frac{\alpha}{\theta}}, & \text{$s \in \compC \setminus \overline{D}$,} \\
     \frac{\sqrt{s_b}i }{\sqrt{(s + 1)(s - s_b)}} (s+1)^{-\alpha}, & \text{$s\in D$,}
  \end{cases}
\end{equation}
where $s_b=1/\theta$, the branch cuts of $\sqrt{(s+1)(s - s_b)}$, $\left(\frac{s+1}{s}\right)^{\frac{\theta-\alpha-1}{\theta}}$ and $(s+1)^{\alpha+1-\theta}$ are taken along $\gamma_2$, $[-1,0]$ and $(-\infty,-1]$, respectively. Hence, in view of \eqref{def:tildeF}, we obtain the following solution to the RH problem \ref{rhp:global2} for $\Pinfty$:
\begin{align}\label{eq:tildeP1infty}
\widetilde P_1^{(\infty)}(z)&=\widetilde \tP(I_2(z)), \qquad z\in \mathbb{H}_\theta \setminus [0,b],
\\
\widetilde P_2^{(\infty)}(z)&=\widetilde \tP(I_1(z)), \qquad z\in \mathbb{C} \setminus [0,b], \label{eq:tildeP2infty}
\end{align}
where $\widetilde \tP$ is given in \eqref{eq:tildeFs}, $I_1$ and $I_2$ are the two inverses of two branches of the mapping $J_c$ defined in \eqref{eq:inverse1} and \eqref{eq:inverse2}, respectively. Again, we note that the RH problem \ref{rhp:global2} admits more than one solution and the one relevant to the present work is given by \eqref{eq:tildeP1infty} and \eqref{eq:tildeP2infty}. We conclude this section with its local behaviours at $z=0$ and $z=b$ for later use.

\begin{prop}\label{prop:tildeP_infty}
With $\Pinfty(z)=(\widetilde P_1^{(\infty)}(z),\widetilde P_2^{(\infty)}(z))$ defined in \eqref{eq:tildeP1infty} and \eqref{eq:tildeP2infty}, we have, as $z \to 0$,
\begin{align}
\label{eq:asy0tildeP1infty}
 \Pinfty_1(z) &=  \left\{
                          \begin{array}{ll}
                            \frac{c^{\frac{(\alpha+1/2)\theta}{1+\theta}}}{\sqrt{1+\theta}}e^{\frac{\alpha+1/2}{1+\theta}\pi i}z^{-\frac{(\alpha+1/2)\theta}{1+\theta}}(1 + \bigO(z^{\frac{\theta}{1+\theta}})), & \arg z \in (0, \frac{\pi}{\theta}),
                            \\
                           \frac{c^{\frac{(\alpha+1/2)\theta}{1+\theta}}}{\sqrt{1+\theta}}e^{-\frac{\alpha+1/2}{1+\theta}\pi i}z^{-\frac{(\alpha+1/2)\theta}{1+\theta}}(1 + \bigO(z^{\frac{\theta}{1+\theta}})), & \arg z \in (-\frac{\pi}{\theta}, 0),
                          \end{array}
                        \right.
\\
\Pinfty_2(z) &=
                          \begin{cases}
                             \frac{c^{\frac{(\alpha+1/2)\theta}{1+\theta}}}{\sqrt{1+\theta}}e^{-\frac{\alpha+1/2}{1+\theta}\pi i} z^{\frac{2\alpha -\theta}{2(1+\theta)}} (1 + \bigO(z^{\frac{\theta}{1+\theta}})), & \arg z \in (0, \pi), \\
                            -\frac{c^{\frac{(\alpha+1/2)\theta}{1+\theta}}}{\sqrt{1+\theta}}e^{\frac{\alpha+1/2}{1+\theta}\pi i} z^{\frac{2\alpha -\theta}{2(1+\theta)}} (1 + \bigO(z^{\frac{\theta}{1+\theta}})), & \arg z \in (-\pi, 0),
                          \end{cases}
\label{eq:asy0tildeP1infty2}
\end{align}
and as $z\to b$,
\begin{equation} \label{eq:asytildeP^infty_b}
\Pinfty_i(z)=\bigO((z-b)^{-\frac14}), \qquad i=1,2.
\end{equation}
\end{prop}
\begin{proof}
From \eqref{eq:tildeFs}, it is easily seen that
\begin{equation*}\label{eq:tildeFlocal}
 \widetilde \tP(s)=
  \begin{cases}
    \bigO(\lvert s - s_b \rvert^{-\frac12}), & \text{$s\to s_b$}, \\
    \sqrt{\frac{1}{1+\theta}}(s+1)^{-\alpha-\frac12}(1+\bigO(s+1)), & \text{$s\to -1$ from $D$}, \\
    c^{\alpha}\sqrt{\frac {1}{1+\theta}}(s+1)^{\frac{\alpha}{\theta}-\frac 12}e^{-\frac{\alpha \pi i}{\theta}}(1+\bigO(s+1)), & \text{$s\to -1$ from $\compC_+ \setminus \overline{D}$}, \\
    -c^{\alpha}\sqrt{\frac {1}{1+\theta}}(s+1)^{\frac{\alpha}{\theta}-\frac 12}e^{\frac{\alpha \pi i}{\theta}}(1+\bigO(s+1)), & \text{$s\to -1$ from $\compC_- \setminus \overline{D}$},
  \end{cases}
\end{equation*}
Combining the above formula with \eqref{eq:I1zero}, \eqref{eq:I2zero}, \eqref{eq:tildeP1infty} and \eqref{eq:tildeP2infty} gives us  \eqref{eq:asy0tildeP1infty}--\eqref{eq:asytildeP^infty_b}.
\end{proof}

\subsection{Third transformation: $\St \to \Q$}
Like the transformation defined in \eqref{def:thirdtransform}, the third transformation is defined by
\begin{equation} \label{eq:thirdtransform_tilde}
\Q(z)=(\Q_1(z),\widetilde{ Q}_2(z))=\left(\frac{\St_1(z)}{\Pinfty_1(z)},\frac{\St_2(z)}{\Pinfty_2(z)} \right).
\end{equation}
In view of the RH problems \ref{rhp:tildeS} and \ref{rhp:global2} for $\St$ and $\Pinfty$, it is then straightforward to check that, with the aid of Proposition \ref{prop:tildeP_infty}, $\Q$ satisfies the following RH problem.
\begin{RHP} \label{RHP:tildeSvar}
\hfill
  \begin{enumerate}[label=\emph{(\arabic*)}, ref=(\arabic*)]
  \item
    $\Q = (\Q_1, \Q_2)$ is analytic in $(\mathbb{H}_\theta  \setminus \Sigma, \compC \setminus \Sigma)$,  where the contour $\Sigma$ is given in \eqref{def:Sigma}.
  \item
   For $z \in \Sigma$, we have
    \begin{equation}
      \Q_+(z) = \Q_-(z) J_{\Q}(z),
      \end{equation}
      where
      \begin{equation} \label{def:Jtildecals}
      J_{\Q}(z) =
      \begin{cases}
        \begin{pmatrix}
          1 & 0 \\
          \frac{\Pinfty_2(z)}{\Pinfty_1(z)}z^{-\alpha}e^{-n \phi(z)} & 1
        \end{pmatrix},
        & \text{$z \in \Sigma_1 \cup \Sigma_2$}, \\
        \begin{pmatrix}
          0 & 1 \\
          1 & 0
        \end{pmatrix},
        & \text{$z \in (0, b)$}, \\
        \begin{pmatrix}
          1 & \frac{\Pinfty_1(z)}{\Pinfty_2(z)}z^{\alpha}e^{n \phi(z)}  \\
          0 & 1
        \end{pmatrix},
        & \text{ $z \in (b, +\infty)$}.
      \end{cases}
    \end{equation}

    \item As $z \to \infty$ in $\mathbb{H}_\theta $, $\widetilde {Q}_1$ behaves as $\Q_1(z)=1+\bigO(z^{-\theta})$.

    \item  As $z \to \infty$ in $\compC$, $\Q_2$ behaves as $\Q_2(z)=\bigO(1)$.

    \item
    As $z \to 0$ in $\mathbb{H}_\theta \setminus \Sigma$, we have
    \begin{equation} \label{eq:tildeasy_Q_1_in_lens}
      \Q_1(z) =
      \begin{cases}
        \bigO(z^{ \frac{\theta-2\alpha}{2(1+\theta)} }), & \text{$\alpha  > 0$ and $z$ inside the lens,} \\
        \bigO(z^{\frac{\theta/2}{1+\theta}}\log z), & \text{$\alpha  = 0$ and $z$ inside the lens,} \\
        \bigO(z^{\frac{(\alpha+1/2)\theta}{1+\theta}}), & \text{$z$ outside the lens or $\alpha < 0$.}
      \end{cases}
    \end{equation}

    \item As $z \to 0$ in $\compC \setminus \Sigma$,  we have
    \begin{equation}
     \Q_2(z)=
      \begin{cases}
        \bigO(z^{\frac{\theta-2\alpha}{2(1+\theta)}}), & \alpha > 0,
       \\
        \bigO (  z^{\frac{\theta}{2(1+\theta)}} \log z ), & \alpha =0,
        \\
        \bigO (z^{\frac{\theta(1+2\alpha)}{2(1+\theta)}} ), & \alpha < 0.
      \end{cases}
    \end{equation}

    \item
    As $z \to b$, we have $\widetilde Q_1(z) = \bigO((z - b)^{1/4})$ and $\widetilde Q_2(z) = \bigO((z - b)^{1/4})$.

    \item For $x>0$, we have the boundary condition
    $
   \Q_1(e^{\pi i/\theta}x) = \Q_1(e^{-\pi i/\theta}x).
    $
  \end{enumerate}
\end{RHP}


\subsection{Local parametrix around $b$}
Again, due to the non-uniform convergence of $J_{\St}$ near the ending points $0$ and $b$ for large $n$, we need to build local parametrices around these points. In a similar way as for the construction of $P^{(b)}$ performed in Section \ref{sec:Pb}, we set
\begin{equation}\label{def:tildePb}
 \Pb (z) :=\widetilde E^{(b)}(z)
  \Psi^{(\Ai)}(n^{\frac23}f_b(z))
  \begin{pmatrix}
    e^{-\frac{n}{2} \phi(z)} \widetilde g^{(b)}_1(z) & 0
    \\
    0 & e^{\frac{n}{2} \phi(z)} \widetilde g^{(b)}_2(z)
  \end{pmatrix},  \quad
  z \in D(b,\epsilon) \setminus \Sigma,
\end{equation}
where $\Psi^{(\Ai)}$ is the Airy parametrix, $f_b(z)$ is defined in \eqref{def:fb},
\begin{equation}\label{def:tildegib}
  \widetilde g^{(b)}_1(z) = \frac{z^{-\frac{\alpha }{2}}}{\widetilde P_1^{(\infty)}(z)}, \qquad \widetilde g^{(b)}_2(z) = \frac{z^{\frac{\alpha }{2}}}{\widetilde P_2^{(\infty)}(z)},
\end{equation}
and
\begin{equation}\label{def:tildeEb}
\widetilde E^{(b)}(z)=\frac{1}{\sqrt{2}}
  \begin{pmatrix}
    \widetilde g^{(b)}_1(z) & 0  \\
    0 & \widetilde g^{(b)}_2(z)
  \end{pmatrix}^{-1}
  e^{\frac{\pi i}{4} \sigma_3}
  \begin{pmatrix}
    1 & -1 \\
    1 & 1
  \end{pmatrix}
  \begin{pmatrix}
    n^{\frac16} f_b(z)^{\frac14} & 0
    \\
    0 & n^{-\frac16} f_b(z)^{-\frac14}
  \end{pmatrix}.
\end{equation}
The contour $\Sigma$ satisfies the same condition as required in the construction of $P^{(b)}$. Thus, analogue to the RH problem \ref{prop:Pb} for $P^{(b)}$, $\Pb$ solves the following RH problem.
\begin{RHP}\label{RHP:Pbtilde}
\hfill
\begin{enumerate}[label=\emph{(\arabic*)}, ref=(\arabic*)]
   \item  $\Pb(z)$ is analytic in $D(b,\epsilon) \setminus \Sigma$.
   \item For $z \in \Sigma \cap D(b,\epsilon)$, we have
    \begin{equation}
      \Pb_+(z) =
      \begin{cases}
        \Pb_-(z) J_{\Q}(z), & z \in \text{$\Sigma \cap D(b,\epsilon) \setminus [b-\epsilon, b]$}, \\
        \begin{pmatrix}
          0 & 1 \\
          1 & 0
        \end{pmatrix}
        \Pb_-(z) J_{\Q}(z), & z \in (b - \epsilon, b),
      \end{cases}
    \end{equation}
    where $J_{\Q}$ is defined in \eqref{def:Jtildecals}.

   \item  As $z \to b$, we have
    $\Pb(z) = \bigO((z-b)^{-1/4})$ and $\Pb(z)^{-1} = \bigO((z-b)^{-1/4})$, which are understood in an entry-wise manner.
   \item For $z\in \partial D(b,\epsilon)$, we have, as $n\to \infty$,
    $ \Pb(z) = I + \bigO(n^{-1}). $
  \end{enumerate}
\end{RHP}

Finally, we define
\begin{equation} \label{eq:defn_V^(b)_tilde}
    \Vb(z) = \widetilde Q(z) \Pb (z)^{-1}, \qquad z \in D(b,\epsilon) \setminus \Sigma.
\end{equation}
It is then easily seen that $\Vb$ satisfies the following RH problem.
\begin{RHP}\label{RHP:tildeVb}
\hfill
  \begin{enumerate}[label=\emph{(\arabic*)}, ref=(\arabic*)]
  \item
    $\Vb = (\Vb_1, \Vb_2)$ is analytic in $D(b,\epsilon)\setminus [b-\epsilon,b]$.
  \item
    For $z \in (b-\epsilon,b)$, we have
    \begin{equation}
      \Vb_+(x)=\Vb_-(x)\begin{pmatrix}
0 & 1
\\
1 & 0
\end{pmatrix}.
      \end{equation}

  \item [\rm (3)]
    As $z \to b$, $\Vb_1(z) = \bigO(1)$ and $\Vb_2(z) = \bigO(1)$.

  \item [\rm (4)]
   For $z \in \partial D(b,\epsilon)$, we have, as $n\to \infty$,
$
\Vb(z)=\Q(z)(I+\bigO(n^{-1})).
$
  \end{enumerate}
\end{RHP}

\subsection{Local parametrix around $0$}
To construct the local parametrix near the origin, we follow a similar strategy used in Section \ref{subsec:P0} and start with a local conversion of the RH problem for $\Q$. Again, we emphasize that $\theta \in \mathbb{N}$ throughout this subsection.

\subsubsection*{A local conversion of the RH problem for $\Q$}
By assuming that $\Sigma_1$ overlaps with the ray $\{ z\in \compC \mid \arg z = \pi/(2\theta) \}$ and $\Sigma_2$ overlaps with the ray $\{ z\in \compC \mid \arg z = -\pi / (2\theta) \}$ for $z\in D(0,r^{1/\theta})$ with $r$ being a small positive constant, we define $\theta+1$ functions $\U_0(z), \U_1(z), \dotsc, \U_{\theta}(z)$ for $z\in D(0,r)$ with some rays removed as follows:

\begin{align}
  \U_0(s^{\theta}) &=  \Q_1(s), \quad \arg s \in (0, \frac{\pi}{2\theta}) \cup (\frac{\pi}{2\theta}, \frac{\pi}{\theta}] \cup (-\frac{\pi}{2\theta}, 0) \cup (-\frac{\pi}{\theta}, -\frac{\pi}{2\theta}), \label{def:tildeU0} \\
  \U_1(s^{\theta}) &=  \Q_2(s), \quad \arg s \in (0, \frac{\pi}{2\theta}) \cup (\frac{\pi}{2\theta}, \frac{\pi}{\theta}] \cup (-\frac{\pi}{2\theta}, 0) \cup (-\frac{\pi}{\theta}, -\frac{\pi}{2\theta}), \label{def:tildeU1}\\
  \U_k(s^{\theta}) &=  \Q_2(s), \quad \arg s \in (\frac{(2k - 3)\pi}{\theta}, \frac{(2k - 1)\pi}{\theta}), \quad k = 2, \dotsc, \theta,\label{def:tildeUk}
\end{align}
or equivalently,
\begin{align}
  \U_0(z) &=  \Q_1(z^{\frac1\theta}), && z\in D(0,r)\setminus \{(-r,r)\cup(-ir,ir)\},
  \\
  \U_k(z) &=  \Q_2(z^{\frac1\theta}e^{\frac{2(k-1)}{\theta}\pi i}), && z \in D(0,r) \setminus \{(-r,r)\cup(-ir,ir)\}, \quad k= 1, 2, \dotsc, \theta,
\end{align}
where $(\Q_1,\Q_2)$ solves the RH problem \ref{RHP:tildeSvar}. Thus, we arrive at the  following $1\times (\theta+1)$ RH problem.

\begin{RHP}\label{rhp:tildeU} \hfill
  \begin{enumerate}[label=\emph{(\arabic*)}, ref=(\arabic*)]
  \item
    $ \U=(\U_0,\U_1,\ldots,\U_\theta)$ is defined and analytic in $D(0,r) \setminus \{(-r,r) \cup (-ir,ir)\}$.
  \item
    For $z\in (-r,r)\cup(-ir,ir)\setminus\{0\}$, we have
    \begin{equation}
      \U_+(z) = \U_-(z) J_{\U}(z),
    \end{equation}
    where
    \begin{equation} \label{eq:defn_J_tildeU}
      J_{\U}(z) =
      \begin{cases}
        \begin{pmatrix}
          1 & 0 \\
          z^{-\frac{\alpha}{\theta}}e^{-n \phi(z^{1/\theta})}\frac{\Pinfty_2(z^{1/\theta})}{\Pinfty_1(z^{1/\theta})} & 1
        \end{pmatrix}
        \oplus I_{\theta - 1}, & z \in (0,ir)\cup (0,-ir), \\
        \begin{pmatrix}
          0 & 1 \\
          1 & 0
        \end{pmatrix}
        \oplus I_{\theta - 1}, & z \in (0, r), \\
        \Mcyclic, & z \in (-r, 0),
      \end{cases}
    \end{equation}
    with $\Mcyclic$ being defined in \eqref{eq:defn_Mcyclic}, and the orientations of the rays are shown in Figure \ref{fig:jumps-U}.
  \item
    As $z\to 0$ from $D(0,r) \setminus \{(-r,r) \cup (-ir,ir)\}$, we have
    \begin{enumerate}[label=\emph{(\alph*)}, ref=(\arabic*)]
    \item for $k=1,\ldots,\theta$, or $k=0$ and $\arg z\in (0, \pi/2) \cup (-\pi/2, 0)$,
      \begin{equation}
      \U_k(z) =
      \begin{cases}
        \bigO (z^{\frac{\theta-2\alpha}{2\theta(1+\theta)}}), & \alpha  > 0, \\
        \bigO ( z^{\frac{1}{2(1+\theta)}} \log z ), &  \alpha = 0, \\
        \bigO ( z^{\frac{\alpha+ 1/2}{1+\theta}} ), & -1< \alpha  < 0,
      \end{cases}
    \end{equation}

    \item
      for $\arg z \in (\pi/2,\pi)\cup (-\pi,-\pi/2)$,
      \begin{equation}
        \U_0(z)=\bigO (z^{\frac{\alpha+ 1/2}{1+\theta}}).
      \end{equation}
    \end{enumerate}
\end{enumerate}
\end{RHP}

Instead of using the Meijer G-parametrix $\Psi^{(\Mei)}$ introduced in Section \ref{subsec:P0}, we need a different one in the construction of the approximation of $\U$, which we describe next.

\subsubsection*{The Meijer G-parametrix $\G$}
We will define $(\theta + 1)^2$ functions $\G_{kj}(\zeta)$ ($j, k = 0, 1, \dotsc, \theta$) with the aid of the Meijer G-functions and show that they constitute a $(\theta + 1) \times (\theta + 1)$ matrix-valued function $\G$ solving a specified RH problem similar to the RH problem for $\Psi^{(\Mei)}$. For $k = 0, 1, \dotsc, \theta$, we set
\begin{equation}\label{def:tildepsik}
  \widetilde{\psi}_k(\zeta) := \frac{(-1)^{k+1}}{(2\pi)^{\theta}}i\zeta^{\alpha}\MeijerG{\theta + 1, 0}{0, \theta + 1}{-}{k, -\frac{\alpha}{\theta}, \frac{1 - \alpha}{\theta}, \dotsc, \frac{\theta - 1 - \alpha}{\theta}}{\zeta^{\theta}},
\end{equation}
which is analytic in $\compC \setminus \realR_-$, and these functions are then defined in two steps for \ref{enu:tildeMeijer_G_kernel_step_1} $j = 1, \dotsc, \theta$, and \ref{enu:tildeMeijer_G_kernel_step_2} $j = 0$ as follows.

\begin{enumerate}[label=(\arabic*), ref=(\arabic*)]
 \item $j = 1, \dotsc, \theta$: \label{enu:tildeMeijer_G_kernel_step_1}
 \begin{equation}\label{def:tildepsikj}
 \G_{kj}(\zeta)= \widetilde{\psi}_k(-e^{\frac{2(j - 1)\pi i}{\theta}} \zeta^{\frac{1}{\theta}}), \qquad \zeta \in \compC \setminus \realR_-.
 \end{equation}

\item $j = 0$: \label{enu:tildeMeijer_G_kernel_step_2}
  \begin{multline} \label{def:tildepsiMeik0}
    \G_{k0}(\zeta) = (-1)^k \MeijerG{1, 0}{0, \theta + 1}{-}{k, -\frac{\alpha}{\theta}, \frac{1 - \alpha}{\theta}, \dotsc, \frac{\theta - 1 - \alpha}{\theta}}{\zeta}\\
    +
    \begin{dcases}
      0, & \arg \zeta \in (\frac{\pi}{2}, \pi) \cup (-\pi, -\frac{\pi}{2}), \\
      -\zeta^{-\frac{\alpha}{\theta}}\widetilde{\psi}_k(-\zeta^{\frac{1}{\theta}}), & \arg \zeta \in (0, \frac{\pi}{2}), \\
      \zeta^{-\frac{\alpha}{\theta}}\widetilde{\psi}_k(-\zeta^{\frac{1}{\theta}}), & \arg \zeta \in (-\frac{\pi}{2}, 0).
    \end{dcases}
  \end{multline}
\end{enumerate}

In \eqref{def:tildepsiMeik0}, we note that the function $\MeijerG{1, 0}{0, \theta + 1}{-}{k, -\alpha/\theta, (1 - \alpha)/\theta, \dotsc, (\theta - 1 - \alpha)/\theta}{\zeta}$ is an entire function in $\zeta$. The model RH problem, analogous to the RH problem \ref{RHP:MeiG} for $\Psi^{(\Mei)}$, used in the approximation of $\U$ reads in the following way.
\begin{RHP}\label{RHP:tildeMeiG} \hfill
  \begin{enumerate}[label=\emph{(\arabic*)}, ref=(\arabic*)]
  \item \label{enu:thm:tildeMeiG_-1}
    $\G(\zeta)$ is a $(\theta+1)\times (\theta+1)$ matrix-valued function defined and analytic in $\compC \setminus \{\mathbb{R}\cup i\mathbb{R} \}$.
  \item \label{enu:thm:tildeMeiG_0}
    For $\zeta \in \mathbb{R}\cup i\mathbb{R} \setminus\{0\}$, we have
    \begin{equation}\label{eq:PsitildeMeijump}
      \G_+(\zeta) = \G_-(\zeta) J_{\G}(\zeta),
    \end{equation}
    where
    \begin{equation} \label{eq:defn_J_tildeMei}
     J_{\G}(\zeta)=\begin{cases}
    \begin{pmatrix}
      1 & 0 \\
      \zeta^{-\frac{\alpha}{\theta}} & 1
    \end{pmatrix}
    \oplus I_{\theta - 1}, & \zeta \in i\realR \setminus \{0\}, \\
    \begin{pmatrix}
      0 & \zeta^{\frac{\alpha}{\theta}}
      \\
      -\zeta^{-\frac{\alpha}{\theta}} & 0
    \end{pmatrix}
    \oplus I_{\theta - 1}, & \zeta > 0, \\
    \Mcyclic, & \zeta <0,
      \end{cases}
    \end{equation}
    and orientations of the real and imaginary axes are shown in Figure \ref{fig:jumps-MeiG}.

  \item \label{enu:thm:tildeMeiG_1}
    As $\zeta \to \infty$, we have, for $\zeta \in \compC_{\pm}$,
    \begin{equation} \label{eq:tildeG_asy_infty}
      \G(\zeta) = \frac{e^{\left( \frac{\alpha + 1/2}{\theta + 1} - \frac{2\alpha}{\theta} \right) \pi i}}{\sqrt{(\theta+1)(2\pi)^{\theta}}} \zeta^{\frac{\alpha}{\theta} - \frac{\alpha + 1/2}{\theta + 1}} \Upsilon(\zeta)\left( \Omega_{\pm} + \bigO(\zeta^{-\frac{1}{\theta + 1}}) \right) e^{\Lambda(\zeta)}\widetilde\Xi(\zeta),
    \end{equation}
    where  $\Upsilon(\zeta)$, $\Omega_{\pm}$ and $\Lambda(\zeta)$ are defined in \eqref{def:Upsilon}, \eqref{def:Omegapm} and \eqref{def:Lambda}, respectively, and
    \begin{multline}\label{def:tildexi}
      \widetilde\Xi(\zeta) =
      \left\{
        \begin{aligned}
          \begin{pmatrix}
           e^{\frac{2\alpha}{\theta} \pi i} \zeta^{-\frac{\alpha }{\theta}} & 0 \\
            0 & e^{2\left( \frac{\alpha}{\theta} - \frac{\alpha + 1/2}{\theta + 1} \right) \pi i}
          \end{pmatrix},
          && \zeta \in \compC_+, \\
          \begin{pmatrix}
           e^{2\left( \frac{\alpha}{\theta} - \frac{\alpha + 1/2}{\theta + 1} \right) \pi i}\zeta^{-\frac{\alpha }{\theta}} & 0 \\
            0 & -e^{\frac{2\alpha}{\theta} \pi i}
          \end{pmatrix},
          && \zeta \in \compC_-,
        \end{aligned}
      \right\} \\
      \oplus \diag\left(e^{2j\left( \frac{\alpha}{\theta} - \frac{\alpha + 1/2}{\theta + 1} \right) \pi i}\right)_{j=2}^{\theta}.
    \end{multline}

  \item \label{enu:thm:tildeMeiG_2}
    As $\zeta \to 0$ from $\compC \setminus (\mathbb{R}\cup i\mathbb{R})$, we have for $k = 0, 1, \dotsc, \theta$,
    \begin{enumerate}[label=\emph{(\alph*)}, ref=(\arabic*)]
    \item for $\arg \zeta \in (\pi/2, \pi) \cup (-\pi, -\pi/2)$,
      \begin{equation} \label{eq:tildepsiMeik0zero}
        \G_{k0}(\zeta) = \bigO(1),
      \end{equation}
    \item
      for $j=1,\ldots,\theta$,
      \begin{equation} \label{eq:tildepsiMeik1zero}
        \G_{kj}(\zeta) = \begin{cases}
          \bigO(1), & \alpha > 0, \\
          \bigO(\log \zeta), & \alpha = 0, \\
          \bigO(\zeta^{\frac{\alpha }{\theta}}), & -1 < \alpha < 0,
        \end{cases}
      \end{equation}
    \end{enumerate}
where $\G_{kj}(\zeta)$ stands for the $(k,j)$-th entry of $\G(\zeta)$.

  \item \label{enu:prop:dettildePsiMei_2}
    As $\zeta \to 0$, each entry of $\G(\zeta)^{-1}$ blows up at most as a power function near the origin. More precisely, we have, for $j=0,1,\ldots,\theta$ and $\arg \zeta \in (\pi/2,\pi)\cup (-\pi,-\pi/2)$
    \begin{enumerate}[label=\emph{(\alph*)}, ref=(\arabic*)]
    \item
      \begin{equation} \label{eq:tildePsiMeiInv0jzero}
        (\G(\zeta)^{-1})_{0j} = \begin{cases}
          \bigO(1), & \alpha \geq \theta-1, \\
          \bigO(\log \zeta), & \alpha = \theta-1, \\
          \bigO(\zeta^{\frac{\alpha+1-\theta}{\theta}}), & -1 < \alpha < \theta-1,
        \end{cases}
      \end{equation}
    \item for $k=1,\ldots,\theta$,
      \begin{equation} \label{eq:tildePsiMeiInv0jzero2}
        (\G(\zeta)^{-1})_{kj} = \bigO(\zeta^{\frac{1-\theta}{\theta}}),
      \end{equation}
    \end{enumerate}
    where $(\G(\zeta)^{-1})_{kj}$ stands for the $(k,j)$-th entry of $\G(\zeta)^{-1}$.
\end{enumerate}
\end{RHP}

In what follows, we show the above RH problem can be solved explicitly.
\begin{prop}\label{thm:tildeMeiG}
Let $\G_{kj}(\zeta)$ be the functions given in \eqref{def:tildepsikj} and \eqref{def:tildepsiMeik0}. The matrix-valued function $\G(\zeta)$ defined by
  \begin{equation}\label{eq:tildePsiMei}
    \G(\zeta)=\left(\widetilde\Psi^{(\Mei)}_{kj}(\zeta)\right)_{k,j=0}^{\theta},
  \end{equation}
solves the RH problem \ref{RHP:tildeMeiG} for $\G$. Moreover, we have
     \begin{equation}\label{eq:dettildePsiMei}
      \det(\G (\zeta))= (2\pi)^{-\frac{\theta(\theta + 1)}{2}} e^{\frac{(\theta + 1)(2 - \theta)}{4} \pi i}\zeta^{\frac{\theta - 1}{2}}.
    \end{equation}
\end{prop}
\begin{proof}
It is clear that $\G$ defined in \eqref{eq:tildePsiMei} is analytic in $\compC \setminus \{\mathbb{R}\cup i\mathbb{R} \}$. We will then show it indeed satisfies the other items of RH problem \ref{RHP:tildeMeiG}, and leave the proof of \eqref{eq:dettildePsiMei} at the end.

\paragraph{Item \ref{enu:thm:tildeMeiG_0}}
It is straightforward to show that the jumps for $\G(\zeta)$ on the negative real axis and on $i\mathbb{R} \setminus \{0\}$ are satisfied if it is defined by \eqref{eq:tildePsiMei}. For the jump on the positive real axis, we observe from \eqref{def:tildepsik} and the integral representation of Meijer G-function \eqref{def:Meijer} that for $k=0,1,\ldots,\theta$, and $\zeta>0$,
  \begin{equation}
    \begin{split}
      & \widetilde{\psi}_k(\zeta^{\frac1\theta}e^{-\pi i} )- \widetilde{\psi}_k( \zeta^{\frac1\theta}e^{\pi i})
      \\
      & = \frac{(-1)^k}{(2\pi)^\theta\pi i} \zeta^{\frac{\alpha}{\theta}} \int_{L}\prod_{l=0}^{\theta-1}\Gamma\left(u+\frac{l-\alpha}{\theta}\right)\Gamma(u+k)\sin( (\theta u-\alpha)\pi)\zeta^{-u}du \\
      & =  \frac{(-1)^k}{2^\theta\pi i} \zeta^{\frac{\alpha}{\theta}} \int_{L}\frac{\Gamma(u+k)\sin( (\theta u-\alpha)\pi)}{\prod_{l=0}^{\theta-1} \Gamma\left(1-\frac{l-\alpha}{\theta}-u\right)\sin \left(\pi (u+\frac{l-\alpha}{\theta})\right)}\zeta^{-u}du
      \\
      & = (-1)^k \frac{\zeta^{\frac{\alpha}{\theta}}}{2 \pi i} \int_{L}\frac{\Gamma(u+k)}{\prod_{l=0}^{\theta-1}\Gamma\left(1-\frac{l-\alpha-u}{\theta}\right)}\zeta^{-u}du
      =(-1)^k \zeta^{\frac{\alpha}{\theta}}\MeijerG{1, 0}{0, \theta + 1}{-}{k, -\frac{\alpha}{\theta}, \frac{1 - \alpha}{\theta}, \dotsc, \frac{\theta - 1 - \alpha}{\theta}}{\zeta},
    \end{split}
  \end{equation}
where we have made use of the reflection formula \eqref{eq:reflformula} in the second equality and the fact (cf. \cite[Equation 1.392.1]{Gradshteyn-Ryzhik07})
\begin{equation}\label{eq:prodsin}
\prod_{l=0}^{\theta-1} \sin \left(\pi (u+\frac{l-\alpha}{\theta})\right)=2^{1-\theta}\sin((\theta u-\alpha)\pi)
\end{equation}
in the third equality. An appeal to this relation gives us the jump of $\G$ on the positive real axis as indicated in \eqref{eq:PsitildeMeijump} and \eqref{eq:defn_J_tildeMei}.

\paragraph{Item \ref{enu:thm:tildeMeiG_1}}
By the definition of $\widetilde{\psi}_k$ given in \eqref{def:tildepsik} and \cite[Equations (2)--(5) in Section 5.10]{Luke69}, it follows that, for $k=0,1,\ldots,\theta$, and $\zeta \to \infty$,
\begin{multline}
 \widetilde{\psi}_k(\zeta) = \frac{(-1)^{k+1}i}{\sqrt{(\theta + 1)(2\pi)^{\theta}}} \zeta^{k - \frac{1}{2} + \frac{\alpha + 1/2 - k}{\theta + 1}}
  \\
  \times \exp(-(\theta + 1) \zeta^{\frac{\theta}{\theta + 1}}) (1 + \bigO(\zeta^{-\frac{\theta}{\theta + 1}})), \quad \arg \zeta \in (-\frac{\theta+2}{\theta}\pi, \frac{\theta+2}{\theta}\pi),
\end{multline}
and
\begin{multline}
  \MeijerG{1, 0}{0, \theta + 1}{-}{k, -\frac{\alpha}{\theta}, \frac{1 - \alpha}{\theta}, \dotsc, \frac{\theta - 1 - \alpha}{\theta}}{\zeta} = \frac{(-1)^k}{\sqrt{(\theta + 1)(2\pi)^{\theta}}}\zeta^{\frac{k - \alpha - 1/2}{\theta + 1}}  \\
  \times
  \begin{cases}
     e^{\frac{1/2 + \alpha - k}{\theta + 1} \pi i} \exp((\theta + 1) e^{-\frac{1}{\theta + 1} \pi i} \zeta^{\frac{1}{\theta + 1}}) (1 + \bigO(\zeta^{\frac{1}{\theta + 1}})), & \arg \zeta \in (0, \pi), \\
    e^{-\frac{1/2 + \alpha - k}{\theta + 1} \pi i} \exp((\theta + 1) e^{\frac{1}{\theta + 1} \pi i} \zeta^{\frac{1}{\theta + 1}}) (1 + \bigO(\zeta^{\frac{1}{\theta + 1}})), & \arg \zeta \in (-\pi, 0).
  \end{cases}
\end{multline}
Hence, we further obtain from \eqref{def:tildepsikj}, \eqref{def:tildepsiMeik0} and \eqref{eq:tildePsiMei} that, as $\zeta \to \infty$,
\begin{multline}
  \G_{k1}(\zeta) = \frac{\zeta^{\frac{\alpha}{\theta} + \frac{k - \alpha - 1/2}{\theta + 1}}}{\sqrt{(\theta + 1)(2\pi)^{\theta}}}  \\
  \times
  \begin{cases}
    e^{\frac{k - \alpha - 1/2}{\theta + 1} \pi i} \exp((\theta + 1) e^{\frac{1}{\theta + 1}\pi i} \zeta^{\frac{1}{\theta + 1}}) (1 + \bigO(\zeta^{-\frac{1}{\theta + 1}})), & \arg \zeta \in (0, \pi), \\
    -e^{-\frac{k - \alpha - 1/2}{\theta + 1} \pi i} \exp((\theta + 1) e^{-\frac{1}{\theta + 1}\pi i} \zeta^{\frac{1}{\theta + 1}}) (1 + \bigO(\zeta^{-\frac{1}{\theta + 1}})), & \arg \zeta \in (-\pi, 0),
  \end{cases}
\end{multline}
for $j = 2, \dotsc, \theta$,
\begin{multline}
  \G_{kj}(\zeta) = \frac{\zeta^{\frac{\alpha}{\theta} + \frac{k - \alpha - 1/2}{\theta + 1}}}{\sqrt{(\theta + 1)(2\pi)^{\theta}}}e^{\left( \frac{2(j - 1)\alpha}{\theta} + \frac{(2j - 1)(k - \alpha - 1/2)}{\theta + 1} \right) \pi i} \\
  \times  \exp((\theta + 1) e^{\frac{2j - 1}{\theta + 1})\pi i} \zeta^{\frac{1}{\theta + 1}}) (1 + \bigO(\zeta^{-\frac{1}{\theta + 1}})), \quad \arg \zeta \in (-\pi, \pi),
\end{multline}
and
\begin{multline}
 \G_{k0}(\zeta) = \frac{\zeta^{\frac{k - \alpha - 1/2}{\theta + 1}}}{\sqrt{(\theta + 1)(2\pi)^{\theta}}} \\
  \times
  \begin{cases}
    e^{\frac{1/2 + \alpha - k}{\theta + 1} \pi i} \exp((\theta + 1) e^{-\frac{1}{\theta + 1} \pi i} \zeta^{\frac{1}{\theta + 1}}) (1 + \bigO(\zeta^{\frac{1}{\theta + 1}})), & \arg \zeta \in (0,\frac{\pi}{2}) \cup (\frac{\pi}{2}, \pi), \\
    e^{-\frac{1/2 + \alpha - k}{\theta + 1} \pi i} \exp((\theta + 1) e^{\frac{1}{\theta + 1} \pi i} \zeta^{\frac{1}{\theta + 1}}) (1 + \bigO(\zeta^{\frac{1}{\theta + 1}})), & \arg \zeta \in (-\pi, -\frac{\pi}{2})\cup(-\frac{\pi}{2},0).
  \end{cases}
\end{multline}
Inserting the above individual asymptotics into \eqref{eq:tildePsiMei}, we arrive at \eqref{eq:tildeG_asy_infty} after straightforward calculations.

\paragraph{Items \ref{enu:thm:tildeMeiG_2} and \ref{enu:prop:dettildePsiMei_2}}
For convenience, we only consider the case that $\alpha \notin \intZ$, and leave the proof of other cases to the reader.

With $k \in \{ 0, \dotsc, \theta \}$ be fixed, it follows from \eqref{def:tildepsik}, \eqref{def:tildepsikj} and \eqref{def:Meijer} that, for $\arg \zeta \in (0, \pi)$,
\begin{equation}\label{eq:tildeGk1int}
\G_{k1}(\zeta)=\psi_k(-\zeta^{\frac{1}{\theta}})=\frac{(-1)^{k+1}}{(2\pi)^{\theta+1}}e^{-\alpha \pi i}\zeta^{\frac{\alpha}{\theta}}
\int_{L}\prod_{l=0}^{\theta-1}\Gamma\left(u+\frac{l-\alpha}{\theta}\right)\Gamma(u+k)(e^{-\theta\pi i}\zeta)^{-u}du.
\end{equation}
Thus, $\G_{k1}(\zeta)$ admits the following Puiseux series representation:
\begin{equation}\label{eq:tildeGk1Puis}
  \G_{k1}(\zeta) = (-i) \left( \sum^{\infty}_{l = 0} \a_{k, l} \zeta^{\frac{l}{\theta}} + e^{-\alpha \pi i} \zeta^{\frac{\alpha}{\theta}} \sum^{\infty}_{l = k} \b_{k, l} \zeta^l \right), \quad \arg \zeta \in (0, \pi),
\end{equation}
where $\a_{k, l}$ and $\b_{k, l}$ are real numbers which can be calculated explicitly by evaluating the residues of the integrand in \eqref{eq:tildeGk1int} at the poles $(\alpha-l)/\theta-m$ and $-k-m$ with $m\in \{0\}\cup \mathbb{N}$, respectively.  By setting
\begin{equation}\label{def:tildefg}
  \f_{k, j}(\zeta) := \sum^{\infty}_{l = 0} \a_{k, \theta l + j - 1} \zeta^l, \quad (j = 1, \dotsc, \theta), \quad \text{and} \quad \gg_k(\zeta): = \sum^{\infty}_{l = k} b_{k, l} \zeta^l,
\end{equation}
it is then readily seen from \eqref{eq:tildeGk1Puis} and \eqref{def:tildepsikj} that
\begin{align}
  \G_{k1}(\zeta) = {}& (-i) \times
  \begin{cases}
    \sum^{\theta}_{l = 1} \zeta^{\frac{l - 1}{\theta}} \f_{k, l}(\zeta) + e^{-\alpha \pi i}  \zeta^{\frac{\alpha}{\theta}} \gg_k(\zeta), & \arg \zeta \in (0, \pi), \\
    \sum^{\theta}_{l = 1} \zeta^{\frac{l - 1}{\theta}} \f_{k, l}(\zeta) + e^{\alpha \pi i} \zeta^{\frac{\alpha}{\theta}} \gg_k(\zeta), & \arg \zeta \in (-\pi, 0),
  \end{cases}
  \label{eq:tildeGk1exp}
  \\
  \G_{kj}(\zeta) = {}& (-i) \left( \sum^{\theta}_{l = 1} c_{l, j} \zeta^{\frac{l - 1}{\theta}} \f_{k, l}(\zeta) + e^{\frac{(2j - 2 - \theta)\alpha \pi i}{\theta}} \zeta^{\frac{\alpha}{\theta}} \gg_k(\zeta) \right), \quad \arg \zeta \in (-\pi, \pi),
  \label{eq:tildeGkjexp}
\end{align}
for $j = 2, \dotsc, \theta$, where $c_{l, j}$ is given in \eqref{eq:expr_tilde_G_kj}. The estimates \eqref{eq:tildepsiMeik1zero} follow directly from the above two formulas.

To show the local behaviour of $\G_{k0}(\zeta)$ near the origin, we observe from \eqref{def:tildepsiMeik0}, \eqref{def:Meijer}, \eqref{eq:reflformula} and \eqref{eq:prodsin} that, for $\arg \in (-\pi,- \pi/2)\cup (\pi /2, \pi )$,
\begin{equation}
\G_{k 0}(\zeta) = \frac{(-1)^k}{(2\pi)^\theta\pi i}\int_{L}\prod_{l=0}^{\theta-1}\Gamma\left(u+\frac{l-\alpha}{\theta}\right)\Gamma(u+k)\sin( (\theta u-\alpha)\pi)\zeta^{-u}du.
\end{equation}
This, together with \eqref{eq:tildeGk1int} and \eqref{eq:tildeGk1exp}, implies that
\begin{equation}\label{eq:tildepsik0zero}
  \G_{k0}(z) = -2 \sin(\alpha \pi) \gg_k(z),
\end{equation}
which gives us \eqref{eq:tildepsiMeik0zero}.

If $\arg \zeta \in (\pi/2,\pi)$, it is readily seen from \eqref{eq:tildeGk1exp}, \eqref{eq:tildeGkjexp} and \eqref{eq:tildepsik0zero} that
  \begin{equation} \label{eq:tildePsiMeiExp}
    \G(\zeta) = \widetilde\Theta(\zeta) \diag(\zeta^{\frac{\alpha}{\theta}}, 1, \zeta^{\frac1\theta}, \dotsc, \zeta^{\frac{\theta - 1}{\theta}}) \widetilde{\mathsf{C}}_{\alpha} \diag(\zeta^{-\frac{\alpha}{\theta}}, 1, \dotsc, 1),
  \end{equation}
  where
  \begin{equation}
    \widetilde\Theta(\zeta)=
     -i
  \begin{pmatrix}
    \gg_{0}(\zeta) & \f_{0, 1}(\zeta) & \cdots & \f_{0, \theta}(\zeta) \\
    \gg_{1}(\zeta) & \f_{1, 1}(\zeta) & \cdots & \f_{2, \theta}(\zeta) \\
    \vdots & \vdots & \cdots & \vdots \\
    \gg_{\theta}(\zeta) & \f_{\theta, 1}(\zeta) & \cdots & \f_{\theta, \theta}(\zeta) \\
  \end{pmatrix}
  \end{equation}
  with $\gg_k(\zeta)$ and $\f_{k,j}(\zeta)$, $k=0,1,\ldots,\theta$, $j=1,\ldots,\theta$, given in \eqref{def:tildefg},
  and
  \begin{equation}
    \widetilde{\mathsf{C}}_{\alpha}=
    \begin{pmatrix}
      -2\sin(\alpha \pi)i & \vec{d}_{\alpha} \\
       0_{\theta \times 1}& C_{\alpha}
    \end{pmatrix},
    \quad \text{where}
    \quad
    \vec{d}_{\alpha} =
    \begin{pmatrix}
      e^{-\alpha \pi i} &  e^{(\frac2\theta-1)\alpha \pi i} & \cdots &  e^{(\frac{2(\theta-1)}{\theta}-1) \alpha \pi i}
      \end{pmatrix},
  \end{equation}
and the $\theta \times \theta$ constant matrix $C_\alpha$ is given in \eqref{eq:psimeizero}. Note that since $\det C_{\alpha} \neq 0$, it is easily seen that
  \begin{equation}
    \widetilde{\mathsf{C}}^{-1}_{\alpha} =
    \begin{pmatrix}
      \frac{i}{2\sin(\alpha \pi)} & -\frac{i}{2\sin(\alpha \pi)}\vec{d}_{\alpha}C_\alpha^{-1}
      \\
      0 & C^{-1}_{\alpha}
    \end{pmatrix}.
  \end{equation}
Also, by assuming \eqref{eq:dettildePsiMei}, we obtain from \eqref{eq:tildePsiMeiExp} that $\det(\widetilde\Theta(\zeta))$ is a non-zero constant and
  \begin{equation}\label{eq:tildeThetaInvZero}
    \widetilde\Theta(\zeta)^{-1}=\bigO(1),
    \qquad \zeta \to 0,
  \end{equation}
which is understood in an entry-wise manner.
  Since
  \begin{align} \label{eq:intermediate_tildePsi_inverse}
      \G(\zeta)^{-1} = {}& \diag(\zeta^{\frac{\alpha}{\theta}}, 1, \dotsc, 1) \widetilde{\mathsf{C}}_{\alpha}^{-1} \diag(\zeta^{-\frac{\alpha}{\theta}}, 1, \zeta^{-\frac1\theta}, \dotsc, \zeta^{-\frac{\theta - 1}{\theta}}) \widetilde\Theta(\zeta)^{-1} \nonumber \\
      = {}&
      \begin{pmatrix}
         \frac{i}{2\sin(\alpha \pi)} & \bigO(\zeta^{\frac{\alpha}{\theta}}) & \bigO(\zeta^{\frac{\alpha - 1}{\theta}}) & \cdots & \bigO(\zeta^{\frac{\alpha + 1 - \theta}{\theta}})
         \\
        0 & \bigO(1) & \bigO(\zeta^{-\frac{1}{\theta}}) & \cdots & \bigO(\zeta^{-\frac{\theta - 1}{\theta}}) \\
        \vdots & \vdots & \vdots & \cdots & \vdots \\
        0 & \bigO(1) & \bigO(\zeta^{-\frac{1}{\theta}}) & \cdots & \bigO(\zeta^{-\frac{\theta - 1}{\theta}})
      \end{pmatrix}
      \widetilde\Theta(\zeta)^{-1},
  \end{align}
we then obtain \eqref{eq:tildePsiMeiInv0jzero} and \eqref{eq:tildePsiMeiInv0jzero2} for $\arg \zeta \in (\pi/2, \pi)$. If $\arg \zeta \in (-\pi, -\pi/2)$, the proof is similar and we omit the details here.

\paragraph{Evaluation of $\det(\G(\zeta))$}
From \eqref{def:tildexi}, it is readily seen that
  \begin{equation}
    \det (\widetilde\Xi(\zeta)) = \pm e^{(\frac{2\alpha}{\theta}+\alpha-\frac{\alpha}{2})\pi i}\zeta^{-\frac{\alpha}{\theta}},\qquad \zeta \in \compC_{\pm}.
  \end{equation}
This, together with \eqref{eq:detUpsilon} and \eqref{eq:tildeG_asy_infty}, implies that, as $\zeta \to \infty$,
\begin{equation}
  \det(\G(\zeta)) =  (2\pi)^{-\frac{\theta(\theta + 1)}{2}} e^{\frac{(\theta + 1)(2 - \theta)}{4} \pi i} \zeta^{\frac{\theta - 1}{2}} (1 + \bigO(\zeta^{-\frac{1}{\theta + 1}})).
\end{equation}
Since it can be shown that
$
  \det(\G(\zeta)) = \bigO(\zeta^{(\theta - 1)/2}),
$
as $\zeta \to 0$, we then obtain \eqref{eq:dettildePsiMei} by combining \eqref{eq:PsitildeMeijump}, \eqref{eq:defn_J_tildeMei} and Liouville's theorem.
\end{proof}

We are now ready to construct the local parametrix near the origin.

\subsubsection*{Construction of the local parametrix around $0$}
Analogous to the definition of $N(z)$ given in \eqref{def:N}, we first set
\begin{equation} \label{def:tildeN}
  \N(z) = \diag (\n_0(z), \n_1(z), \dotsc, \n_{\theta}(z)), \quad z\in\compC \setminus (-\infty,b^{\theta}],
\end{equation}
where
\begin{equation}\label{def:tildeni}
  \n_0(z) = \P^{(\infty)}_1(z^{\frac1 \theta}), \quad \n_j(z) = \P^{(\infty)}_2(e^{\frac{2(j - 1)}{\theta} \pi i} z^{\frac1 \theta}), \quad j = 1, \dotsc, \theta,
\end{equation}
with $\P^{(\infty)} =(\P^{(\infty)}_1, \P^{(\infty)}_2)$ defined in \eqref{eq:tildeFs}--\eqref{eq:tildeP2infty}.

In view of \eqref{eq:asy0tildeP1infty} and  \eqref{eq:asy0tildeP1infty2}, the following local behaviour of $\N$ near the origin is immediate.
\begin{prop}\label{prop:tildeNzero}
Let $\n_i(z)$, $i=0,1,\ldots,\theta$, be the functions defined in \eqref{def:tildeni}. As $z \to 0$, we have, with the constant $c$ given in \eqref{def:Jcs},
\begin{align}
  \n_0(z) = {}&
                \begin{cases}
                  \frac{c^{\frac{(\alpha+1/2)\theta}{1+\theta}}}{\sqrt{1+\theta}}e^{\frac{\alpha+1/2}{1+\theta}\pi i}z^{-\frac{\alpha+1/2}{1+\theta}}(1 + \bigO(z^{\frac{1}{1+\theta}})), & z \in \compC_+, \\
                  \frac{c^{\frac{(\alpha+1/2)\theta}{1+\theta}}}{\sqrt{1+\theta}}e^{-\frac{\alpha+1/2}{1+\theta}\pi i}z^{-\frac{\alpha+1/2}{1+\theta}}(1 + \bigO(z^{\frac{1}{1+\theta}})), & z \in \compC_-, \\
                \end{cases} && \\
  \n_1(z) = {}&
                \begin{cases}
                  \frac{c^{\frac{(\alpha+1/2)\theta}{1+\theta}}}{\sqrt{1+\theta}}e^{-\frac{\alpha+1/2}{1+\theta}\pi i} z^{\frac{\alpha - \theta/2}{\theta(1+\theta)}} (1 + \bigO(z^{\frac{1}{1+\theta}})), & z \in \compC_+, \\
                  -\frac{c^{\frac{(\alpha+1/2)\theta}{1+\theta}}}{\sqrt{1+\theta}}e^{\frac{\alpha+1/2}{1+\theta}\pi i} z^{\frac{\alpha - \theta/2}{\theta(1+\theta)}} (1 + \bigO(z^{\frac{1}{1+\theta}})), & z \in \compC_-,
                \end{cases}
  \end{align}
  and for $j=2,\ldots,\theta$,
  \begin{align}
  \n_j(z) = {}& \frac{c^{\frac{(\alpha+1/2)\theta}{1+\theta}}}{\sqrt{1+\theta}}e^{-\frac{\alpha+1/2}{1+\theta}\pi i} e^{\frac{(2\alpha - 1)(j - 1)}{\theta(1 + \theta)} \pi i} z^{\frac{\alpha - \theta/2}{\theta(1+\theta)}}(1 + \bigO(z^{\frac{1}{1+\theta}})).
\end{align}
\end{prop}

With $r=r_n$ given in \eqref{def:rn}, we then define  a $(\theta + 1) \times (\theta + 1)$ matrix-valued function for  $z \in D(0,r_n^{\theta})\setminus (\mathbb{R} \cup i \mathbb{R})$ as follows:
\begin{multline}\label{def:tildesfP0}
  \widetilde{\mathsf{P}}^{(0)}(z) = (2\pi)^{\frac{\theta}{2}} c^{\frac{(\alpha + 1/2)\theta}{1+\theta}} (\rho n)^{\frac{1}{2} - \frac{\alpha}{\theta}} \diag(1,(\rho n)^{-1},\ldots,(\rho n)^{-\theta}) \\
  \times \G((\rho n)^{\theta + 1} z) ((\rho n)^{\frac{\alpha(\theta + 1)}{\theta}} \oplus I_{\theta}) \N(z)^{-1} e^{-nM(z)},
\end{multline}
where $\rho$, $M(z)$ and $\N(z)$ are defined in \eqref{eq:defn_rho}, \eqref{def:M} and \eqref{def:tildeN}, respectively, and $\G$ is the Meijer G-parametrix solving the RH problem \ref{RHP:tildeMeiG}.

\begin{prop}\label{RHP:tildesfP0}
The function $\widetilde{\mathsf{P}}^{(0)}(z)$ defined in \eqref{def:tildesfP0} has the following properties.
  \begin{enumerate}[label=\emph{(\arabic*)}, ref=(\arabic*)]
  \item
    $ \widetilde{\mathsf{P}}^{(0)}(z)$ is analytic in $D(0,r_n^{\theta}) \setminus (\mathbb{R} \cup i \mathbb{R})$.
  \item
    For $z\in D(0,r_n^{\theta}) \cap (\mathbb{R} \cup i \mathbb{R})$, we have
    \begin{equation}
      \widetilde{\mathsf{P}}^{(0)}_{+}(z) = \widetilde{\mathsf{P}}^{(0)}_{-}(z) J_{\U}(z),
    \end{equation}
    where $J_{\U}(z)$ is defined in \eqref{eq:defn_J_tildeU}.
  \item
    For $z\in \partial D(0,r_n^{\theta})$, we have, as $n\to \infty$,
    \begin{equation}\label{eq:tildesfP0asy}
      \widetilde{\mathsf{P}}^{(0)}(z) = \Upsilon(z)\Omega_{\pm} (I+\bigO(n^{-\frac{1}{2\theta+1}})), \qquad z \in \compC_{\pm},
    \end{equation}
    where $\Upsilon(z)$ and $\Omega_{\pm}$ are defined in \eqref{def:Upsilon} and \eqref{def:Omegapm}, respectively.
  \end{enumerate}
\end{prop}
\begin{proof}
Since the proof is similar to that of Proposition \ref{RHP:sfP0}, we only give a sketch here. While it is straightforward to check the
jump of $\widetilde{\mathsf{P}}^{(0)}$ on $(0,ir_n^\theta)\cup(0,-ir_n^\theta)$, one needs to use the relations among $m_i(z)$, $i=0,1,\ldots,\theta$, established in \eqref{eq:relmi1}, \eqref{eq:relmi12}, and the facts that for $z<0$,
\begin{equation}\label{eq:tilden1}
\n_{0,+}(z)=\n_{0,-}(z), \quad \n_{\theta,+}(z)=\n_{1,-}(z), \quad \n_{j-1,+}(z)=\n_{j,-}(z), \quad j=2, \dotsc, \theta,
\end{equation}
and for $z\in (0,r_n^\theta)$,
\begin{equation}\label{eq:tilden2}
\n_{0,+}(z) = -\n_{1,-}(z) z^{-\frac{\alpha}{\theta}}, \quad \n_{1,+}(z) = \n_{0,-}(z)z^{\frac{\alpha}{\theta}}, \quad \n_{j,+}(z)=\n_{j,-}(z), \quad j=2,\ldots,\theta,
\end{equation}
to see its jump on $(-r_n^\theta, 0)$ and $(0, r_n^\theta)$. The relations in \eqref{eq:tilden1} and \eqref{eq:tilden2} are easily seen from the definitions of $\n_i(z)$, $i=0,1,\ldots,\theta$, given in \eqref{def:tildeni} and the RH problem \ref{rhp:global2} for $\P^{(\infty)}$.

The large $n$ asymptotics of $\widetilde{\mathsf{P}}^{(0)}$ \eqref{eq:tildesfP0asy} on the boundary of the disc follows directly from \eqref{def:tildesfP0}, \eqref{eq:tildeG_asy_infty}, \eqref{eq:enMeLambda} and
\begin{multline}
    \left( (\rho n)^{\frac{\alpha(\theta + 1)}{\theta}} \oplus I_{\theta} \right) \widetilde \Xi((\rho n)^{\theta+1}z)\N(z)^{-1} = \\
    \sqrt{1+\theta}c^{-\frac{(\alpha+1/2)\theta}{\theta+1}}e^{(\frac{2\alpha}{\theta}-\frac{\alpha+1/2}{\theta+1})\pi i} z^{\frac{\alpha+1/2}{\theta+1}-\frac{\alpha}{\theta}}(I+\bigO(n^{-\frac{2\theta}{2\theta+1}})), \quad z\in \partial D(0,r_n^\theta),
  \end{multline}
where the function $\widetilde \Xi$ is defined in \eqref{def:tildexi}.
\end{proof}

With $\widetilde{\mathsf{P}}^{(0)}$ given in \eqref{def:tildesfP0}, we next define
\begin{equation}\label{def:tildeP0}
  \P^{(0)}(z) = \Omega_{\pm}^{-1}\Upsilon(z)^{-1}\widetilde{\mathsf{P}}^{(0)}(z), \quad z \in \compC_{\pm}.
\end{equation}
In view of \eqref{eq:UpsilonOmegajump}, \eqref{def:tildeP0} and Proposition \ref{RHP:tildesfP0}, the following RH problem for $\P^{(0)}$ is then immediate.
\begin{RHP}\label{rhp:tildep0}
\hfill
  \begin{enumerate}[label=\emph{(\arabic*)}, ref=(\arabic*)]
  \item
    $\P^{(0)}(z)$ is analytic in $D(0,r_n^{\theta}) \setminus (\mathbb{R} \cup i \mathbb{R})$.
  \item
    For $z\in D(0,r_n^{\theta}) \cap (\mathbb{R} \cup i \mathbb{R})$, we have
    \begin{equation}
      \P^{(0)}_+(z) =
      \begin{cases}
        \P^{(0)}_-(z) J_{\U}(z), & z \in (0, ir_n^{\theta}) \cup (0, -ir_n^{\theta}), \\
        \left(\begin{pmatrix}
            0 & 1
            \\
            1 & 0
          \end{pmatrix} \oplus I_{\theta - 1} \right)
        \P^{(0)}_-(z) J_{\U}(z), & z \in (0, r_n^{\theta}), \\
        \Mcyclic^{-1} \P^{(0)}_-(z) J_{\U}(z), & z \in (-r_n^{\theta}, 0),
      \end{cases}
    \end{equation}
    where $J_{\U}(z)$ is defined in \eqref{eq:defn_J_tildeU}.
  \item
    For $z\in \partial D(0,r_n^{\theta})$, we have, as $n\to \infty$,
    \begin{equation}\label{eq:tildeP0asy}
      \P^{(0)}(z) = I+\bigO(n^{-\frac{1}{2\theta+1}}).
    \end{equation}
  \end{enumerate}
\end{RHP}

Finally, as in the definition of $V^{(0)}$ in \eqref{def:V0}, we set
\begin{equation} \label{def:tildeV0}
  \V^{(0)}(z) = \U(z) \P^{(0)}(z)^{-1}, \quad z \in D(0,r_n^{\theta}) \setminus (\mathbb{R} \cup i \mathbb{R}),
\end{equation}
where $\U$ is the solution of the RH problem \ref{rhp:tildeU}. It is then straightforward to check the following RH problem for $\V^{(0)}$.
\begin{RHP}\label{rhp:tildeV0}
\hfill
\begin{enumerate}[label=\emph{(\arabic*)}, ref=(\arabic*)]
  \item \label{enu:rhp:tildeV0:1}
    $\V^{(0)}=(\V^{(0)}_0, \V^{(0)}_1, \dotsc, \V^{(0)}_{\theta})$ is analytic in $D(0,r_n^{\theta}) \setminus \mathbb{R} $, or equivalently, its definition can be analytically extended onto $(0, ir^{\theta}_n) \cup (-ir^{\theta}_n, 0)$.
  \item \label{enu:rhp:tildeV0:2}
    For $z\in (-r_n^\theta,r_n^{\theta})\setminus \{0\}$, we have
    \begin{equation} \label{eq:enu:rhp:tildeV0:2}
      \V^{(0)}_+(z) = \V^{(0)}_-(z)
      \begin{cases}
        \begin{pmatrix}
          0 & 1 \\
          1 & 0
        \end{pmatrix}
        \oplus I_{\theta - 1}, & z \in (0, r_n^{\theta}), \\
        \Mcyclic^{-1}, & z \in (-r_n^{\theta}, 0).
      \end{cases}
    \end{equation}
  \item \label{enu:rhp:tildeV0:3}
    For $z\in \partial D(0,r_n^{\theta})$, we have, as $n \to \infty$,
    \begin{equation}\label{eq:tildeV0asy}
      \V^{(0)}(z) = \U(z)(I+\bigO(n^{-\frac{1}{2\theta+1}})).
    \end{equation}
  \item
    As $z \to 0$, we have
    \begin{equation}\label{eq:tildeV0kzero}
      \V^{(0)}_k(z)=\bigO(1), \qquad k=0,1,\ldots,\theta.
    \end{equation}
  \end{enumerate}
\end{RHP}

\subsection{Final transformation}
With the aid of the functions $\Q(z)$, $\V^{(b)}(z)$ and $\V^{(0)}(z)$ given in previous sections, our final transformation is the following definition of a $1\times 2$ vector-valued function $\R(z) = (\R_1(z), \R_2(z))$ such that $\R_1(z)$ is analytic in $\halfH \setminus \Sigma^R$ and $\R_2(z)$ is analytic in $\compC \setminus \Sigma^R$:
\begin{align}
   \R_1(z) = {}&
  \begin{cases}
    \Q_1(z), & \hbox{$z\in \halfH \setminus \{D(b,\epsilon) \cup D(0,r_n) \cup \Sigma^R \}$,} \\
    \V_1^{(b)}(z), & \hbox{$z\in D(b,\epsilon) \setminus [b-\epsilon, b]$,} \\
    \V_0^{(0)}(z^\theta), & \hbox{$z \in W_1 \setminus [0,r_n]$,} \\
  \end{cases} \label{def:tildeR1}
  \\
  \R_2(z) = {}&
  \begin{cases}
    \Q_2(z), & \hbox{$z\in \mathbb{C} \setminus \{D(b,\epsilon) \cup D(0,r_n) \cup \Sigma^R \}$,} \\
    \V_2^{(b)}(z), & \hbox{$z\in D(b,\epsilon) \setminus [b-\epsilon, b]$,} \\
    \V_1^{(0)}(z^\theta), & \hbox{$z\in W_1 \setminus [0,r_n]$,} \\
    \V_k^{(0)}(z^\theta), & \hbox{$z\in W_k$, $k=2,\ldots,\theta$,}
  \end{cases} \label{def:tildeR2}
\end{align}
where $\Sigma^R$ and $W_k$, $k=1,\ldots,\theta$, are defined in \eqref{def:sigmaR} and \eqref{def:Wk}, respectively. In view of the RH problems \ref{RHP:tildeSvar}, \ref{RHP:tildeVb} and \ref{rhp:tildeV0} for $\Q$, $\V^{(b)}$ and $\V^{(0)}$, it is straightforward to check that $\R$ satisfies the following shifted RH problem.
\begin{RHP} \label{rhp:tildeR}
\hfill
  \begin{enumerate}[label=\emph{(\arabic*)}, ref=(\arabic*)]
  \item
    $\R=(\R_1,\R_2)$ is analytic in $(\halfH  \setminus \Sigma^{R}, \compC \setminus \Sigma^{R})$.

  \item \label{enu:rhp:tildeR:2}
    $\R$ satisfies the following jump conditions:
    \begin{equation} \label{eq:rhp:tildeR:2}
      \R_+(z)=\R_-(z)
      \begin{cases}
        J_{\Q}(z), & \hbox{$z\in \Sigma_1^R \cup \Sigma_2^R \cup (b+\epsilon, +\infty)$,} \\
        \P^{(b)}(z), & \hbox{$z\in \partial D(b,\epsilon)$,} \\
        \begin{pmatrix}
          0 & 1 \\
          1 & 0
        \end{pmatrix}, & \hbox{$z\in (0,b) \setminus \{r_n, b-\epsilon\}$,}
      \end{cases}
    \end{equation}
    where $J_{\Q}(z)$ and $\P^{(b)}(z)$ are defined in \eqref{def:Jtildecals} and \eqref{def:tildePb}, respectively, and with $\Gamma_1$ defined in \eqref{def:Gammak}, we have  for $z\in\Gamma_1$,
    \begin{equation}
      \R_{1,+}(z) = \R_{1,-}(z)\P_{00}^{(0)}(z^{\theta}) + \sum^{\theta}_{j = 1} \R_{2,-}( e^{\frac{2(j - 1)\pi i}{\theta}}z) \P^{(0)}_{j0}(z^{\theta}),\quad \arg z \neq 0, \pm \frac{\pi}{2\theta},\label{eq:jumptildeRgamma1}
    \end{equation}
    and for $k=1,\ldots,\theta$,
    \begin{equation}
    \R_{2,+}(e^{\frac{2(k - 1)\pi i}{\theta}}z) = \R_{1,-}(z) \P^{(0)}_{0k}(z^{\theta}) + \sum^{\theta}_{j = 1} \R_{2,-}(e^{\frac{2(j - 1)\pi i}{\theta}}z) \P^{(0)}_{jk}(z^{\theta}), \label{eq:jumptildeRgammai}
    \end{equation}
     where $\P^{(0)}(z)=(\P^{(0)}_{jk}(z))_{j,k=0}^\theta$ defined in \eqref{def:tildeP0} is the solution of the RH problem \ref{rhp:tildep0}.
    The orientations of the curves in $\Sigma^{R}$ are shown in Figure \ref{fig:Sigma_R}, and the orientations of the circles $\partial D(0, r_n)$ and $\partial D(b, \epsilon)$ are particularly taken in a clockwise manner.
  \item
    As $z \to \infty$ in $\halfH$, $\R_1$ behaves as $\R_1(z) = 1 + \bigO(z^{-\theta})$.
  \item \label{enu:rhp:tildeR:4}
    As $z \to \infty$ in $\compC$, $\R_2$ behaves as $\R_2(z) = \bigO(1)$.
  \item \label{enu:rhp:tildeR:5}
    As $z \to b$ or $z \to 0$, we have $\R_1(z) = \bigO(1)$ and $\R_2(z) = \bigO(1)$.
  \item
    For $x > 0$, we have the boundary condition $\R_1(e^{\pi i/\theta}x) = \R_1(e^{-\pi i/\theta}x)$.
\end{enumerate}
\end{RHP}

To estimate $\R$ for large $n$, we again transform the RH problem for $\R$ to a scalar one on the complex-$s$ plane. For this purpose, we set
\begin{equation}
\omega(z):=\frac1z,
\end{equation}
which maps $\compC\cup \{\infty\}$ to $\compC\cup \{\infty\}$, and denote by
\begin{equation}
\J(s):=J_c \circ \omega(s)=J_c(1/s)
\end{equation}
the composition of the function $J_c$ given in \eqref{def:Jcs} with $\omega$. Similar to the definition of $\tR$ in \eqref{eq:scalar_R_defn}, the transformation is then defined as follows:
\begin{equation} \label{eq:scalar_tildeR_defn}
  \widetilde{\tR}(s) =
  \begin{cases}
  \R_1(\J(s)), & \text{$s \in \omega(D \setminus[-1, 0])$ and $s \notin \omega(I_2(\Sigma^R))$}, \\
  \R_2(\J(s)), & \text{$s \in \omega(\compC \setminus \overline{D})$ and $s \notin \omega(I_1(\Sigma^R ))$,}
  \end{cases}
\end{equation}
where $D$ is the region bounded by the curve $\gamma=\gamma_1\cup\gamma_2$, $\Sigma^R$ is given in \eqref{def:sigmaR}, and $I_1: \compC \setminus [0, b] \to \compC \setminus \overline{D}$ and $I_2: \halfH \setminus [0, b] \to D$ are defined in \eqref{eq:inverse1} and \eqref{eq:inverse2}, respectively.
We emphasize that it is crucial to introduce the mapping $\omega$ here, which ensures $\widetilde{\tR}$ is normalized at infinity, as indicated in item \ref{enu:rhp:tildetR:3} of the RH problem \ref{rhp:tildetR} for $\widetilde{\tR}$ below.

To state the RH problem for $\widetilde{\tR}$, we recall the curves $\SigmaR^{(\cdot)}$ given in \eqref{def:SigmaRi}, and define
\begin{equation}\label{def:tildeSigmaR}
  \widetilde{\SigmaR}: = \widetilde{\SigmaR}^{(1')} \cup \widetilde{\SigmaR}^{(2)} \cup \widetilde{\SigmaR}^{(3)} \cup \widetilde{\SigmaR}^{(3')}  \cup \widetilde{\SigmaR}^{(4)} \cup \bigcup_{k=1}^{\theta} \widetilde{\SigmaR}^{(5)}_k ,
\end{equation}
with $\widetilde{\SigmaR}^{(\cdot)}:=\omega(\SigmaR^{(\cdot)})$, which is the images of the union of the solid and the dashed curves in Figure \ref{fig:jump_R_scalar} under the mapping $\omega$. We also define the following functions on each curve constituting $\widetilde{\SigmaR}$:
\begin{align}
  J_{\widetilde{\SigmaR}^{(1')}}(s) = {}& \frac{\P^{(\infty)}_2(z)}{z^{\alpha} P^{(\infty)}_1(z)} e^{-n\phi(z)}, & s \in {}& \widetilde{\SigmaR}^{(1')}, \\
  \intertext{where $z = \J(s) \in \Sigma^R_1 \cup \Sigma^R_2$,}
  J_{\widetilde{\SigmaR}^{(2)}}(s) = {}& \frac{z^{\alpha} \P^{(\infty)}_1(z)}{\P^{(\infty)}_2(z)} e^{n\phi(z)}, & s \in {}& \widetilde{\SigmaR}^{(2)}, \\
  \intertext{where $z = \J(s) \in (b + \epsilon, +\infty)$.}
  J^{1}_{\widetilde{\SigmaR}^{(3)}}(s) = {}& \P^{(b)}_{22}(z) - 1, \quad J^{2}_{\widetilde{\SigmaR}^{(3)}}(s) = \P^{(b)}_{12}(z), & s \in {}& \widetilde{\SigmaR}^{(3)}, \label{def:Jtildesigma3}\\
  \intertext{where $z = \J(s) \in \partial D(b,\epsilon)$,}
  J^{1}_{\widetilde{\SigmaR}^{(3')}}(s) = {}& \P^{(b)}_{11}(z) - 1, \quad J^{2}_{\widetilde{\SigmaR}^{(3')}}(s) = \P^{(b)}_{21}(z), & s \in {}& \widetilde{\SigmaR}^{(3')}, \label{def:Jtildesigma32}\\
  \intertext{where $z = \J(s) \in \partial D(b,\epsilon)$,}
  J^{0}_{\widetilde{\SigmaR}^{(4)}}(s) = {}& \P^{(0)}_{00}(z^{\theta}) - 1, \quad J^{j}_{\widetilde{\SigmaR}^{(4)}}(s) = \P^{(0)}_{j0}(z^{\theta}), & s \in {}& \widetilde{\SigmaR}^{(4)}, \\
  \intertext{where $z = \J(s) \in \Gamma_1$ and $j=1,\ldots,\theta$,}
  J^k_{\widetilde{\SigmaR}^{(5)}_k}(s) = {}& \P^{(0)}_{kk}(z^{\theta}) - 1, \quad J^j_{\widetilde{\SigmaR}^{(5)}_k}(s) = \P^{(0)}_{j k}(z^{\theta}),\quad j\neq k, & s \in {}& \widetilde{\SigmaR}^{(5)}_k,
\end{align}
where $z = \J(s) \in \Gamma_k$, $k=1,\ldots,\theta$, and $j=0,1,\ldots,\theta$. In \eqref{def:Jtildesigma3} and \eqref{def:Jtildesigma32}, $\P^{(b)}(z)=(\P^{(b)}_{jk}(z))_{j,k=1}^2$ is defined in \eqref{def:tildePb}. With the aid of these functions, analogue to the definition of $\DeltaR$ in \eqref{def:DeltaR}, we further define an operator $\DeltatildeR$ that acts on functions defined on $\widetilde\SigmaR$ by
\begin{equation}\label{def:DeltatildeR}
  \DeltatildeR f(s) =
  \begin{cases}
    J_{\widetilde{\SigmaR}^{(1')}}(s) f(\s), & \text{$s \in \widetilde{\SigmaR}^{(1')}$ and $\s = \omega(I_1(z))$}, \\
    J_{\widetilde{\SigmaR}^{(2)}}(s) f(\s), & \text{$s \in \widetilde{\SigmaR}^{(2)}$ and $\s = \omega(I_2(z))$}, \\
    J^{1}_{\widetilde{\SigmaR}^{(3)}}(s) f(s) + J^{2}_{\widetilde{\SigmaR}^{(3)}}(s) f(\s), & \text{$s \in \widetilde{\SigmaR}^{(3)}$ and $\s = \omega(I_2(z))$}, \\
    J^{1}_{\widetilde{\SigmaR}^{(3')}}(s) f(s) + J^{2}_{\widetilde{\SigmaR}^{(3')}}(s) f(\s), & \text{$s \in \widetilde{\SigmaR}^{(3')}$ and $\s = \omega(I_1(z))$}, \\
    J^{0}_{\widetilde{\SigmaR}^{(4)}}(s) f(s) + \sum\limits ^{\theta}_{j = 1} J^{j}_{\widetilde{\SigmaR}^{(4)}}(s) f(\s_j), & \text{$s \in \widetilde{\SigmaR}^{(4)}$ and $\s_j = \omega(I_1(z e^{\frac{2(j - 1)\pi i}{\theta}})) \in \omega(I_1(\Gamma_j))$}, \\
    \sum\limits^{\theta}_{j = 0} J^j_{\widetilde{\SigmaR}^{(5)}_k}(s) f(\s_j), & \text{$s \in \widetilde{\SigmaR}^{(5)}_k$ and  $\s_0 = \omega(I_2(z e^{\frac{2(1 - k)\pi i}{\theta}})) \in \omega(D)$,} \\
    & \text{$\s_j = \omega(I_1(z e^{\frac{2(j - k)\pi i}{\theta}})) \in \omega(I_1(\Gamma_j)) \subseteq \omega(\compC \setminus \overline{D})$} \\
    & \text{for $j = 1, \dotsc, \theta$,  such that $\s_k = s$},
  \end{cases}
\end{equation}
where $z=\J(s)$ and $f$ is a complex-valued function defined on $\widetilde{\SigmaR}$. Similar to the proof of Proposition \ref{prop:uniquess}, we have that the function $\widetilde{\tR}(s)$ defined in \eqref{eq:scalar_tildeR_defn} is the unique solution of the following RH problem after trivial analytical extension.

\begin{RHP} \label{rhp:tildetR} \hfill
  \begin{enumerate}[label=\emph{(\arabic*)}, ref=(\arabic*)]
  \item
    $\widetilde{\tR}(s)$ is analytic in $\compC \setminus \widetilde{\SigmaR}$, where the contour $\widetilde{\SigmaR}$ is defined in \eqref{def:tildeSigmaR}.
  \item \label{enu:rhp:tildetR:2}
    For $s\in \widetilde{\SigmaR}$, we have
    \begin{equation}
      \widetilde{\tR}_+(s) - \widetilde{\tR}_-(s) = \DeltatildeR \widetilde{\tR}_-(s),
    \end{equation}
    where $\DeltatildeR$ is the operator defined in \eqref{def:DeltatildeR}.
  \item \label{enu:rhp:tildetR:3}
    As $s \to \infty$, we have
    \begin{equation}
      \widetilde{\tR}(s)=1+\bigO(s^{-1}).
    \end{equation}
  \item \label{enu:rhp:tildetR:4}
    As $s \to 0$, we have $\widetilde{\tR}(s)=\bigO(1)$.
  \end{enumerate}
\end{RHP}

As $n\to \infty$, we have the asymptotic estimate of the operator $\DeltatildeR$ analogous to \eqref{eq:estOperator}
\begin{equation}\label{eq:estOperator2}
\lVert \DeltatildeR \rVert_{L^2(\widetilde{\SigmaR})} \leq M_{\widetilde{\SigmaR}}n^{-\frac{1}{2\theta + 1}},
\end{equation}
for some positive constant $M_{\widetilde{\SigmaR}}$. Thus, by using similar arguments leading to Lemma \ref{lem:tRest}, we finally conclude the following estimate of $\widetilde{\tR}$.
\begin{lemma}\label{lem:tildetRest}
As $n\to \infty$, we have
\begin{equation}\label{eq:esttildetR}
\widetilde{\tR}(s)=1+\bigO(n^{-\frac{1}{2\theta + 1}}),
\end{equation}
uniformly for $\lvert s+1 \rvert < \epsilon r_n^{\theta/(1+\theta)}$, where $\epsilon$ is a small positive constant and $r_n$ is defined in \eqref{def:rn}.
\end{lemma}

\section{Proofs of main results}\label{sec:provmainresult}
In this section, we prove Theorems \ref{thm:pqkappa} and \ref{thm:kernel} by using the asymptotic results obtained in Sections \ref{sec:asyY} and \ref{sec:asytildeY}.

\subsection{Proof of Theorem \ref{thm:pqkappa}}

\paragraph{Proof of \eqref{eq:p_thm}}

By tracing back the transformations $Y \to T \to S \to Q$ given in \eqref{eq:Y_to_T}, \eqref{def:S} and \eqref{def:thirdtransform}, it is readily seen from \eqref{def:Y} that
\begin{equation} \label{eq:Y_1_in_Z_i}
  p_n(z)=Y_1(z) =
  \begin{cases}
    Z_1(z), & \text{$z$ outside the lens}, \\
    Z_1(z) + Z_2(z), & \text{$z$ in the upper part of the lens}, \\
    Z_1(z) - Z_2(z), & \text{$z$ in the lower part of the lens},
  \end{cases}
  \end{equation}
where
\begin{equation} \label{def:Zi}
  Z_1(z) = Q_1(z) P^{(\infty)}_1(z) e^{ng(z)}, \quad Z_2(z) = Q_2(z) P^{(\infty)}_2(z) \theta z^{-\alpha - 1 + \theta} e^{n(V(z) - \g(z) + \ell)},
\end{equation}
with $g(z)$ and $\g(z)$ defined in \eqref{def:g} and \eqref{def:tildeg}, respectively.

If $z\in D(0,r_n)$, we see from \eqref{def:U0}--\eqref{def:Uk} that $Q_i(z)$, $i=1,2$, are determined by $U_k(z^{\theta})$, $k = 0, 1, \dotsc, \theta$, and moreover, by \eqref{def:V0},
\begin{equation}\label{eq:UkinVP0}
U_k(z) = \sum^{\theta}_{i = 0} V^{(0)}_i(z) P^{(0)}_{ik}(z),
\end{equation}
where $V^{(0)}(z)=(V^{(0)}_0(z), V^{(0)}_1(z), \dotsc, V^{(0)}_{\theta}(z))$ and $P^{(0)}(z)=(P^{(0)}_{ik}(z))_{i,k=0}^\theta$ are defined in \eqref{def:V0} and \eqref{def:P0}, respectively. We next represent $V^{(0)}$ via $R = (R_1, R_2)$ through \eqref{def:R1} and \eqref{def:R2} in the vicinity of $0$. To this end, we define, for $0 < \lvert s \rvert < \varepsilon r_n^{\theta/(\theta+1)}=\varepsilon n^{-2\theta/(2\theta + 1)}$ with $\arg s \neq \pm \pi/(\theta + 1)$ and $\varepsilon$ being a small enough positive constant,
\begin{equation}\label{def:v}
  v(s) =
  \begin{cases}
    R_2(e^{-\frac{\pi i}{\theta}} s^{1 + \frac{1}{\theta}}), & \arg s \in (0, \frac{\pi}{\theta + 1}), \\
    R_2(e^{\frac{\pi i}{\theta}} s^{1 + \frac{1}{\theta}}), & \arg s \in (-\frac{\pi}{\theta + 1}, 0), \\
    R_1(-(-s)^{1 + \frac{1}{\theta}}), & \arg s \in (\frac{\pi}{\theta + 1}, \pi] \cup (-\pi, -\frac{\pi}{\theta + 1}).
  \end{cases}
\end{equation}
In view of the RH problem \ref{rhp:R} for $R$, it is easily seen that $v(s)$ can be extended analytically in the disc $\lvert s \rvert \leq \varepsilon n^{-2\theta/(2\theta + 1)}$, and thus admits the Taylor expansion $v(s) = \sum^{\infty}_{k = 0} c_k s^k$ there. Moreover, as $\varepsilon$ is small enough, we see from \eqref{def:v}, \eqref{eq:scalar_R_defn} and Lemma \ref{lem:tRest} that uniformly for $s$ in this disk
\begin{equation}\label{eq:estv}
v(s)=1+\bigO(n^{-\frac{1}{2\theta+1}}).
\end{equation}
A combination of \eqref{def:R1}, \eqref{def:R2} and \eqref{def:v} shows that, for $\lvert z \rvert < \varepsilon^{\theta + 1} n^{-2\theta(\theta+1)/(2\theta + 1)}$,
\begin{multline}\label{eq:Vinv}
  (V^{(0)}_0(z), V^{(0)}_1(z), \dotsc, V^{(0)}_{\theta}(z)) = \\
    \begin{cases}
      (v(z^{\frac{1}{\theta + 1}} e^{-\frac{\pi i}{\theta + 1}}), v(z^{\frac{1}{\theta + 1}} e^{\frac{\pi i}{\theta + 1}}), v(z^{\frac{1}{\theta + 1}} e^{\frac{3\pi i}{\theta + 1}}), \dotsc, v(z^{\frac{1}{\theta + 1}} e^{\frac{(2\theta - 1)\pi i}{\theta + 1} })), & z \in \compC_+, \\
      (v(z^{\frac{1}{\theta + 1}} e^{\frac{\pi i}{\theta + 1}}), v(z^{\frac{1}{\theta + 1}} e^{-\frac{\pi i}{\theta + 1}}), v(z^{\frac{1}{\theta + 1}} e^{\frac{3\pi i}{\theta + 1}}), \dotsc, v(z^{\frac{1}{\theta + 1}} e^{\frac{(2\theta - 1)\pi i}{\theta + 1} })), & z \in \compC_-.
    \end{cases}
\end{multline}
We then define for $\lvert s \rvert < \varepsilon^{\theta + 1} n^{-2\theta(\theta+1)/(2\theta + 1)}$,
\begin{equation} \label{eq:defn_v_k(s)}
  v_k(s) = \frac{1}{2\pi i} \oint_{\lvert \zeta \rvert = \varepsilon n^{-2\theta/(2\theta + 1)}} \frac{v(\zeta)}{\zeta^{k + 1}(1 - s/\zeta^{\theta + 1})} d\zeta = \sum^{\infty}_{l = 0} c_{k + l(\theta + 1)} s^{l}, \quad k = 0, 1, \dotsc, \theta,
\end{equation}
which are analytic near the origin and $v(s)=\sum_{k=0}^\theta s^k v_k(s^{\theta+1})$ for $\lvert s \rvert < \varepsilon n^{-2\theta/(2\theta + 1)}$. Hence, we rewrite \eqref{eq:Vinv} as
\begin{equation}
  (V^{(0)}_0(z), V^{(0)}_1(z), \dotsc, V^{(0)}_{\theta}(z)) = (v_0(-z), v_1(-z), \dotsc, v_{\theta}(-z)) \Upsilon(z) \Omega_{\pm}, \quad z\in \compC_{\pm},
\end{equation}
where $\Upsilon(z)$  and $\Omega_{\pm}$  are defined in \eqref{def:Upsilon} and \eqref{def:Omegapm}, respectively. Inserting the above formula into \eqref{eq:UkinVP0}, we obtain from \eqref{def:P0} that, for $k=0,1,\ldots,\theta$,
\begin{equation} \label{eq:U_j_asy_simplified}
  \begin{split}
    U_k(z) = {}& \sum^{\theta}_{i = 0} V^{(0)}_i(z) P^{(0)}_{ik}(z) = \sum^{\theta}_{i = 0} v_i(-z) \mathsf{P}^{(0)}_{ik}(z) \\
    = {}& \frac{(2\pi)^{1 - \frac{\theta}{2}}}{\sqrt{\theta}} c^{\frac{2(\alpha+1)-\theta}{2(1+\theta)}} (\rho n)^{\frac{\alpha + 1}{\theta} - \frac{1}{2}} \left( \sum^{\theta}_{l = 0} \frac{v_l(-z)}{(\rho n)^l} \Psi^{(\Mei)}_{lk}((\rho n)^{\theta + 1} z) \right) \\
    & \times
    \begin{cases}
      (\rho n)^{-\frac{\theta + 1}{\theta}(\alpha + 1 - \theta)} (n_0(z))^{-1} e^{n m_0(z)}, & k = 0, \\
      (n_k(z))^{-1} e^{n m_k(z)}, & k = 1, 2, \dotsc, \theta,
    \end{cases}
  \end{split}
\end{equation}
where $\mathsf{P}^{(0)}(z)=(\mathsf{P}^{(0)}_{ik}(z))_{i,k=0}^\theta$, $m_k(z)$ and $n_k(z)$, $k = 0, 1, \dotsc, \theta$, are defined in \eqref{def:sfP0}, \eqref{def:m0}--\eqref{def:mj} and \eqref{def:ni}, the constants $c$ and $\rho$ are given in \eqref{def:Jcs} and \eqref{eq:defn_rho}, respectively.

From now on, we assume that $z \neq 0$ belongs to a compact subset $K$ of $\compC$. Additionally, it is also assumed that $z \notin \Sigma$ and $\arg z \neq (2k + 1)\pi/\theta$ for any integer $k$. If $n$ is large enough, it is clear that there exists $k\in\{1,\ldots,\theta\}$, such that $(\rho n)^{-1 - 1/\theta} z \in W_k \subset D(0,r_n)$ and $|(\rho n)^{-1 - \theta} z^\theta| < \varepsilon^{\theta + 1} n^{-2\theta(\theta+1)/(2\theta + 1)}$, where the domain $W_k$ is defined in \eqref{def:Wk}. Thus, it follows from \eqref{def:Zi}, \eqref{def:U0}--\eqref{def:Uk} and \eqref{eq:U_j_asy_simplified} that
\begin{multline} \label{eq:Z_j_asy}
  Z_j \left( \frac{z}{(\rho n)^{1 + 1/\theta}} \right) = \frac{(2\pi)^{1 - \theta/2}}{\sqrt{\theta}} c^{\frac{2(\alpha+1)-\theta}{2(1+\theta)}} (\rho n)^{\frac{\alpha + 1}{\theta} - \frac{1}{2}} \exp \left[ n \left( V \left( \frac{z}{(\rho n)^{1 + 1/\theta}} \right) + \ell - \g_+(0) \right) \right] \\
  \times \left( \sum^{\theta}_{l = 0} (\rho n)^{-l} v_l \left( -\frac{z^{\theta}}{(\rho n)^{\theta + 1}} \right) \Psi^{(\Mei)}_{l,i(j, z)}(z^{\theta}) \right) \times
  \begin{cases}
    1, & j = 1, \\
    \theta z^{\theta - \alpha - 1}, & j = 2,
  \end{cases}
\end{multline}
where
\begin{equation}
  i(j, z) =
  \begin{cases}
    0, & j = 2, \\
    k, & j = 1 \text{ and } (\rho n)^{-1 - 1/\theta} z \in W_k.
  \end{cases}
\end{equation}

If $(\rho n)^{-1 - 1/\theta} z \in W_k$ with $k = 2, \dotsc, \theta$, or $(\rho n)^{-1 - 1/\theta} z \in W_1$ but outside the lens, it follows from  \eqref{eq:Y_1_in_Z_i} that
\begin{equation} \label{eq:final_p_n:1}
  p_n \left( \frac{z}{(\rho n)^{1 + 1/\theta}} \right) = Z_1 \left( \frac{z}{(\rho n)^{1 + 1/\theta}} \right).
\end{equation}
To estimate $Z_1$ for large $n$, we note that uniformly for all $z \in K$,
\begin{equation}\label{eq:Vzeroexp}
  V \left( \frac{z}{(\rho n)^{1 + 1/\theta}} \right)=V(0)+\bigO(n^{-1-\frac1\theta}).
\end{equation}
This, together with \eqref{eq:gequal}, implies that uniformly
\begin{align}\label{eq:estenV}
    \exp \left[ n \left( V \left( \frac{z}{(\rho n)^{1 + 1/\theta}} \right) + \ell - \g_+(0) \right) \right] = {}& e^{n g_-(0)}(1+\bigO(n^{-\frac1\theta}))
    \nonumber \\
    = {}& (-1)^ne^{n \Re g(0)}(1+\bigO(n^{-\frac1\theta})),
\end{align}
where we use that $\Re g_-(0)=\Re g_+(0)$ and $\Im g_\pm(0)=\pm \pi i$ in the second equality, as can be seen directly from its definition in \eqref{def:g}. Furthermore, on account of \eqref{eq:estv}, the functions $v_k(s)$ defined in \eqref{eq:defn_v_k(s)} satisfy the estimates
\begin{equation} \label{eq:est_v_k}
  v_0(s) = 1 + \bigO(n^{-\frac{1}{2\theta + 1}}), \quad v_k(s) = \bigO(n^{\frac{2k\theta - 1}{2\theta + 1}}), \quad k = 1, \dotsc, \theta,
\end{equation}
for $\lvert s \rvert < \varepsilon^{\theta + 1} n^{-2\theta(\theta + 1)/(2\theta + 1)}$. Substituting the above two formulas into \eqref{eq:Z_j_asy}, it is readily seen from \eqref{eq:final_p_n:1} that, with $C_n$ defined in \eqref{def:Cn} and uniformly in $z$,
\begin{equation} \label{eq:p_n_outside}
p_n \left( \frac{z}{(\rho n)^{1 + 1/\theta}} \right) = (-1)^n \theta^{-1} C_n \left( \Psi^{(\Mei)}_{0k}(z^{\theta}) + \bigO ( n^{-\frac{1}{2\theta + 1}} ) \right),
\end{equation}
which is \eqref{eq:p_thm} by \eqref{eq:defn_tilde_G_kj} and \eqref{def:psiMeik1}. On the other hand, if $(\rho n)^{-1 - 1/\theta} z \in W_1$ and locates inside the upper/lower part of the lens, we obtain from \eqref{eq:Y_1_in_Z_i}, \eqref{eq:Z_j_asy}, \eqref{eq:estenV} and \eqref{eq:est_v_k} that, for $z\in \compC_{\pm}$,
\begin{equation} \label{eq:p_n_inside}
  \begin{split}
    p_n \left( \frac{z}{(\rho n)^{1 + 1/\theta}} \right) = {}& Z_1 \left( \frac{z}{(\rho n)^{1 + 1/\theta}} \right) \pm Z_2 \left( \frac{z}{(\rho n)^{1 + 1/\theta}} \right) \\
    = {}& (-1)^n \theta^{-1} C_n \left( \Psi^{(\Mei)}_{01}(z^{\theta}) \pm \theta z^{\theta - \alpha - 1} \Psi^{(\Mei)}_{00}(z^{\theta}) + \bigO ( n^{-\frac{1}{2\theta + 1}} ) \right),
\end{split}
\end{equation}
which again gives us \eqref{eq:p_thm}, in view of \eqref{def:psiMeik1}, \eqref{def:PsiMeik0} and the relation \eqref{eq:Grelation}. We note the error terms $\bigO(n^{-1/(2\theta + 1)})$ in both \eqref{eq:p_n_outside} and \eqref{eq:p_n_inside} are uniform for all $z \in K \setminus \{ z = 0 \text{ or } z \in \Sigma \text{ or } \arg z = (2k + 1)\pi/\theta, \, k \in \intZ \}$.

Finally, we note that it is straightforward obtain \eqref{eq:p_thm} by analytic continuation if $(\rho n)^{-1 - 1/\theta} z $ belongs to the boundaries of $W_k$ or the lens, especially when $z \in [0,\infty)$.

\paragraph{Proof of \eqref{eq:q_thm}}
The proof of \eqref{eq:q_thm} is analogous to that of \eqref{eq:p_thm}, and we will omit some details here. Moreover, we will prove \eqref{eq:q_thm} for $z \in \halfH$ or its boundary.

By tracing back the transformations $\Y \to \T \to \St \to \Q$ given in \eqref{def:tildeT}, \eqref{def:tildeS} and \eqref{eq:thirdtransform_tilde}, it is readily seen from \eqref{def:Ytilde} that for $z\in \halfH$,
\begin{equation}\label{eq:tildeY_1_in_Z_i}
  q_n(z^\theta)=\Y_1(z) =
  \begin{cases}
    \Z_1(z), & \text{$z$ outside the lens,} \\
    \Z_1(z) + \Z_2(z), & \text{$z$ in the upper part of the lens,} \\
    \Z_1(z) - \Z_2(z), & \text{$z$ in the lower part of the lens,}
  \end{cases}
  \end{equation}
  where
\begin{equation}\label{def:tildeZi}
  \Z_1(z) = \Q_1(z) \P^{(\infty)}_1(z) e^{n\g(z)}, \quad \Z_2(z) = \Q_2(z) \P^{(\infty)}_2(z) z^{-\alpha} e^{n(V(z) - g(z) + \ell)}.
\end{equation}

If $z \in \halfH \cap D(0,r_n)$, we see from \eqref{def:tildeU0} and \eqref{def:tildeU1} that $\Q_i(z)$, $i=1,2$, are determined by $\U_{i - 1}(z^{\theta})$, and moreover, by \eqref{def:tildeV0},
\begin{equation}\label{eq:tildeUkinVP0}
\U_k(z) = \sum^{\theta}_{i = 0} \V^{(0)}_i(z) \P^{(0)}_{ik}(z),
\end{equation}
where $\V^{(0)}(z)=(\V^{(0)}_0(z), \V^{(0)}_1(z), \dotsc, \V^{(0)}_{\theta}(z))$ and $\P^{(0)}(z)=(\P^{(0)}_{ik}(z))_{i,k=0}^\theta$ are defined in \eqref{def:tildeV0} and \eqref{def:tildeP0}, respectively. We next represent $\V^{(0)}$ via $\R = (\R_1, \R_2)$ through \eqref{def:tildeR1} and \eqref{def:tildeR2}. For $0 < \lvert s \rvert < \varepsilon n^{-2\theta/(2\theta + 1)}$ with $\arg s \neq \pm \pi/(\theta + 1)$ and $\varepsilon$ being a small enough positive constant, we define
\begin{equation}\label{def:tildev}
  \v(s) =
  \begin{cases}
    \R_1(e^{-\frac{\pi i}{\theta}} s^{1 + \frac{1}{\theta}}), & \arg s \in (0, \frac{\pi}{\theta + 1}), \\
    \R_1(e^{\frac{\pi i}{\theta}} s^{1 + \frac{1}{\theta}}), & \arg s \in (-\frac{\pi}{\theta + 1}, 0), \\
    \R_2(-(-s)^{1 + \frac{1}{\theta}}), & \arg s \in (\frac{\pi}{\theta + 1}, \pi] \cup (-\pi, -\frac{\pi}{\theta + 1}),
  \end{cases}
\end{equation}
which is an analogue of the function $v$ given in \eqref{def:v}. In view of the RH problem \ref{rhp:tildeR} for $\R$, we also have that $\v(s)$ can be extended analytically in the disk $\lvert s \rvert \leq \varepsilon n^{-2\theta/(2\theta + 1)}$ and admits a Taylor expansion $\v(s) = \sum^{\infty}_{k = 0} \c_k s^k$ there. Since $\varepsilon$ is small enough, we see from \eqref{def:tildev}, \eqref{eq:scalar_tildeR_defn} and Lemma \ref{lem:tildetRest} that uniformly for $s$ in this disk,
\begin{equation}\label{eq:esttildev}
\v(s)=1+\bigO(n^{-\frac{1}{2\theta+1}}).
\end{equation}
A combination of \eqref{def:tildeR1}, \eqref{def:tildeR2} and \eqref{def:tildev} shows that, for $\lvert z \rvert < \varepsilon^{\theta + 1} n^{-2\theta(\theta+1)/(2\theta + 1)}$,
\begin{multline}\label{eq:tildeVinv}
  (\V^{(0)}_0(z), \V^{(0)}_1(z), \dotsc, \V^{(0)}_{\theta}(z)) = \\
  \begin{cases}
    (\v(z^{\frac{1}{\theta + 1}} e^{-\frac{\pi i}{\theta + 1}}), \v(z^{\frac{1}{\theta + 1}} e^{\frac{\pi i}{\theta + 1}}), \v(z^{\frac{1}{\theta + 1}} e^{\frac{3\pi i}{\theta + 1}}), \dotsc, \v(z^{\frac{1}{\theta + 1}} e^{\frac{(2\theta - 1) \pi i}{\theta + 1} })), & z \in \compC_+ \cap \halfH, \\
    (\v(z^{\frac{1}{\theta + 1}} e^{\frac{\pi i}{\theta + 1}}), \v(z^{\frac{1}{\theta + 1}} e^{-\frac{\pi i}{\theta + 1}}), \v(z^{\frac{1}{\theta + 1}} e^{\frac{3\pi i}{\theta + 1}}), \dotsc, \v(z^{\frac{1}{\theta + 1}} e^{\frac{(2\theta - 1)\pi i}{\theta + 1} })), & z \in \compC_-\cap \halfH.
  \end{cases}
\end{multline}
We then define for $\lvert s \rvert < \varepsilon^{\theta + 1} n^{-2\theta(\theta+1)/(2\theta + 1)}$,
\begin{equation}\label{eq:defn_tildev_k(s)}
  \v_k(s) = \frac{1}{2\pi i} \oint_{\lvert \zeta \rvert = \varepsilon n^{-2\theta/(2\theta + 1)}} \frac{\v(\zeta)}{\zeta^{k + 1}(1 - s/\zeta^{\theta + 1})} d\zeta  = \sum^{\infty}_{l = 0} \c_{k + l(\theta + 1)} s^{l}, \quad k = 0, 1, \dotsc, \theta,
\end{equation}
which are analytic near the origin and $\v(s)=\sum_{k=0}^\theta s^k \v_k(s^{\theta+1})$ for $\lvert s \rvert < \varepsilon n^{-2\theta/(2\theta + 1)}$. Hence, we rewrite \eqref{eq:tildeVinv} as
\begin{equation}
  (\V^{(0)}_0(z), \V^{(0)}_1(z), \dotsc, \V^{(0)}_{\theta}(z)) = (\v_0(-z), \v_1(-z), \dotsc, \v_{\theta}(-z)) \Upsilon(z) \Omega_{\pm},\quad z\in \compC_{\pm}\cap \halfH,
\end{equation}
where $\Upsilon(z)$  and $\Omega_{\pm}$  are defined in \eqref{def:Upsilon} and \eqref{def:Omegapm}, respectively. Inserting the above formula into \eqref{eq:tildeUkinVP0}, we obtain from \eqref{def:tildeU0}, \eqref{def:tildeU1} and \eqref{def:tildeP0} that
\begin{equation} \label{eq:tildeQ}
  \begin{split}
    \Q_j(z) = \U_{j - 1}(z^{\theta}) = {}& \sum^{\theta}_{i = 0} \V^{(0)}_i(z^{\theta}) \P^{(0)}_{i, j - 1}(z^{\theta}) = \sum^{\theta}_{i = 0} \v_i(-z^{\theta}) \Pmodeltilde^{(0)}_{i, j - 1}(z^{\theta}) \\
    = {}& (2\pi)^{\frac{\theta}{2}} c^{\frac{(\alpha + 1/2)\theta}{1+\theta}} (\rho n)^{\frac{1}{2} - \frac{\alpha}{\theta}} \frac{e^{n\g_+(0)}}{\P^{(\infty)}_{j}(z)} \sum^{\theta}_{i = 0} \frac{\v_i(-z^{\theta})}{(\rho n)} \G_{i, j - 1}((\rho n)^{\theta + 1} z^{\theta}) \\
    & \times
    \begin{cases}
      (\rho n)^{\frac{\alpha(\theta + 1)}{\theta}} e^{-n\g(z)}, & j = 1, \\
      e^{n(g(z) - V(z) - \ell)}, & j = 2,
    \end{cases}
  \end{split}
\end{equation}
where $\Pmodeltilde^{(0)}(z)=\Pmodeltilde^{(0)}_{ik}(z))_{i,k=0}^\theta$, $c$ and $\rho$ are given in \eqref{def:tildesfP0}, \eqref{def:Jcs} and \eqref{eq:defn_rho}, respectively.

We now assume that $z \in K \cap \halfH$ and $(\rho n)^{-1 - 1/\theta} z \in W_1 \setminus \Sigma \subset D(0,r_n)\cap \halfH$, where $K$ is a compact subset of $\overline{\halfH}$, $W_1$ and $\Sigma$ are defined in \eqref{def:Wk} and \eqref{def:Sigma}, respectively. Since $|(\rho n)^{-1 - \theta} z^\theta| < \varepsilon^{\theta + 1} n^{-2\theta(\theta+1)/(2\theta + 1)}$ for $n$ large enough, it is then readily seen from \eqref{def:tildeZi} and \eqref{eq:tildeQ} that for $j = 1, 2$,
\begin{multline}\label{eq:tildeZj}
  \Z_j \left( \frac{z}{(\rho n)^{1 + 1/\theta}} \right) = (2\pi)^{\frac{\theta}{2}} c^{\frac{(\alpha + 1/2)\theta}{(\theta + 1)}} (\rho n)^{\frac{1}{2} + \alpha} e^{n \g_+(0)} \\
  \times \sum^{\theta}_{i = 0} (\rho n)^{-i} \v_i \left( -\frac{z^{\theta}}{(\rho n)^{\theta + 1}} \right) \Psitilde^{(\Mei)}_{i, j - 1}(z^{\theta}) \times
  \begin{cases}
    1, & j = 1, \\
    z^{-\alpha}, & j = 2.
  \end{cases}
\end{multline}
The asymptotic formula \eqref{eq:q_thm} for $q_n(z^\theta)$ then follows from \eqref{eq:tildeY_1_in_Z_i}, \eqref{eq:tildeZj}, \eqref{def:tildepsikj}, \eqref{def:tildepsiMeik0} and direct calculations with the error bound $\bigO(n^{-1/(2\theta + 1)})$ uniformly valid for all such $z$. To that end, one needs the estimates
\begin{equation}
  \v_0(s) = 1 + \bigO(n^{-\frac{1}{2\theta + 1}}), \quad \v_k(s) = \bigO(n^{\frac{2k\theta - 1}{2\theta + 1}}), \quad k = 1, \dotsc, \theta,
\end{equation}
for $\lvert s \rvert < \varepsilon n^{-2\theta/(2\theta + 1)}$, and the facts that $\Re \g_-(0)=\Re \g_+(0)$ and $\Im \g_\pm(0)=\pm \pi i$. We leave the details to the interested readers, and also remark that the asymptotics can be extended to $\Sigma$ and the boundary of $\halfH$.

\paragraph{An auxiliary lemma for Theorem \ref{thm:kernel}}
Although the proof of Theorem \ref{thm:kernel} is given in Section \ref{subsec:proof_kernel_formula} below, we prove the following lemma for the use therein, since it is closely related to the asymptotic formulas \eqref{eq:p_thm} and \eqref{eq:q_thm} proved above.
\begin{lemma}\label{lem:asyknnV}
 As $n\to\infty$, we have,
 \begin{equation}\label{eq:k_n_asy}
    \frac{1}{\kappa_n} \frac{e^{-nV(x/(\rho n)^{1 + 1/\theta})}}{(\rho n)^{\alpha(1 + 1/\theta) + 1/\theta}} p_n \left( \frac{x}{(\rho n)^{1 + 1/\theta}} \right) q_n \left( \frac{y^{\theta}}{(\rho n)^{1 + \theta}} \right) = \theta c^{-\frac{\theta}{\theta + 1}} k^{(\alpha, \theta)}(x, y) + \bigO (n^{-\frac{1}{2\theta + 1}}),
  \end{equation}
uniformly for $x,y$ in compact subsets of $[0,\infty)$, where the constants $\kappa_n$, $\rho$, $c$ are given in \eqref{eq:pqbioOP}, \eqref{eq:defn_rho}, \eqref{def:Jcs}, and
\begin{multline}\label{def:kalphatheta}
    k^{(\alpha, \theta)}(x, y)= x^{\theta-\alpha-1} \MeijerG{\theta, 0}{0, \theta + 1}{-}{\frac{\alpha - \theta + 1}{\theta}, \frac{\alpha - \theta + 2}{\theta}, \dotsc, \frac{\alpha - 1}{\theta}, \frac{\alpha}{\theta}, 0}{x^{\theta}}
\\ \times
\MeijerG{1, 0}{0, \theta + 1}{-}{0, -\frac{\alpha}{\theta}, \frac{1 - \alpha}{\theta}, \dotsc, \frac{\theta - 1 - \alpha}{\theta}}{y^{\theta}}.
  \end{multline}
\end{lemma}

\begin{proof}
In view of \eqref{eq:p_thm}, \eqref{eq:q_thm}, \eqref{eq:Vzeroexp} and \eqref{eq:gequal}, it suffices to show
\begin{equation} \label{eq:kappa_thm}
  \kappa_n = 2\pi \theta^{-1/2} c^{\alpha + 1} e^{n\ell} ( 1 + \bigO ( n^{-\frac{1}{2\theta + 1}} ) ),
\end{equation}
where $\ell$ is the constant given in \eqref{eq:EL1}.

By \eqref{def:Y}, we note that
\begin{equation}\label{eq:kapparep}
  Y_2(z) = \frac{i}{2\pi} \kappa_n z^{-(n + 1)\theta} + \bigO(z^{-(n + 2)\theta}), \quad z \to +\infty.
\end{equation}
On the other hand, we see from \eqref{eq:Y_to_T}, \eqref{def:S}, \eqref{def:thirdtransform} and \eqref{def:R2} that for $z \in \halfH \setminus [0,\infty)$ large enough,
\begin{equation}\label{eq:Y2rep}
  Y_2(z) = R_2(z) P^{(\infty)}_2(z) e^{n(\ell - \g(z))},
\end{equation}
where the functions $R_2(z)$, $P^{(\infty)}_2(z)$ and $\g(z)$ are given in \eqref{def:R2}, \eqref{eq:P2} and \eqref{def:tildeg}, respectively.
As $z \to \infty$, it is readily seen from  \eqref{eq:gtilde_asy}, \eqref{def:Jcs},\eqref{eq:Fs}, \eqref{eq:P2} and \eqref{eq:inverse2} that
\begin{equation} \label{eq:entglim}
  \lim_{z \to \infty} \frac{e^{n\g(z)}}{z^{n\theta}} = 1, \qquad  \lim_{z \to \infty} z^{\theta} P^{(\infty)}_2(z) = \lim_{s \to 0}(J_c(s))^{\theta} \tP(s) = \frac{c^{\alpha + 1} i}{\sqrt{\theta}},
\end{equation}
and from \eqref{eq:scalar_R_defn}, \eqref{def:Jcs}, \eqref{eq:inverse2} and Lemma \ref{lem:tRest} that
\begin{equation}\label{eq:R2lim}
  \lim_{z \to \infty} R_2(z) = \lim_{z \to \infty} \tR(I_2(z)) =\lim_{s \to 0} \tR(s) = 1 + \bigO ( n^{-\frac{1}{2\theta + 1}} ), \quad n\to \infty.
\end{equation}
Inserting \eqref{eq:entglim} and \eqref{eq:R2lim} into \eqref{eq:Y2rep}, we then obtain \eqref{eq:kappa_thm} from \eqref{eq:kapparep} and conclude the proof.
\end{proof}

\subsection{Proof of Theorem \ref{thm:kernel}} \label{subsec:proof_kernel_formula}
Throughout this subsection, we will use the notations $c^{(V)}, \rho^{(V)}, \kappa_{n,j}^{(V)}, p_{n,j}^{(V)}$ and $q_{n,j}^{(V)}$ to emphasize their dependence on the external field $V$ and the relevant parameters, since we need to consider different $V$ at certain  stage of the proof.

Recall the correlation kernel $K_n$ defined in \eqref{eq:sum_form_K}, it is easily seen that
\begin{equation}\label{def:knj}
  K_n(x,y)= x^{\alpha} \sum_{j=0}^n k^{(V)}_{n, j}(x, y),  \qquad  k^{(V)}_{n, j}(x, y) := \frac{e^{-nV(x)}}{\kappa^{(V)}_{n, j}}  p^{(V)}_{n, j}(x)q^{(V)}_{n, j}(y^{\theta}),
\end{equation}
where $\kappa^{(V)}_{n, j}$ is given \eqref{eq:pqbioOP}. To proceed, we assume that, without loss of generality,
$$V(x) = x^r + \bigO(x^{r + 1}), \qquad x\to 0,$$
for some positive integer $r$, due to the analyticity of $V$ and the assumption \eqref{eq:one-cut_regular}. Then we define a family of functions $V_{\tau}$ indexed by a continuous parameter $\tau \in [0, 1]$ as follows:
\begin{equation}\label{def:vtau}
  V_{\tau}(x): =
  \begin{cases}
    \tau^{-1} V(\tau^{1/r} x), & \tau \in (0, 1], \\
    x^r, & \tau = 0.
  \end{cases}
\end{equation}
Clearly, $V_{\tau}(x)$ is continuous in both $x$ and $\tau$, and our assumption on the external field $V$ implies that Theorem \ref{thm:pqkappa} still holds with $V$ replaced by $V_{\tau}$. Moreover, we have the following technical lemma:
\begin{lemma} \label{prop:generalized}
  Suppose $V$ satisfies \eqref{eq:loggrowth} and \eqref{eq:one-cut_regular}. With $V_{\tau}$ defined in \eqref{def:vtau}, we have, for any $M, \epsilon > 0$, there exists a positive integer $N_{M, \epsilon}$ such that if $n > N_{M, \epsilon}$, then
\begin{multline}
  \left\lvert n^{-\alpha(1 + \frac{1}{\theta}) - \frac{1}{\theta}} k^{(V_{\tau})}_{n, n} \left( \frac{x}{n^{1 + 1/\theta}}, \frac{y}{n^{1 + 1/\theta}} \right) \right. \\
  - \left. \theta (c^{(V_{\tau})})^{-\frac{\theta}{\theta + 1}} (\rho^{(V_{\tau})})^{\alpha(1 + \frac{1}{\theta}) + \frac{1}{\theta}} k^{(\alpha, \theta)}((\rho^{(V_{\tau})})^{1 + \frac{1}{\theta}} x, (\rho^{(V_{\tau})})^{1 + \frac{1}{\theta}} y) \right\rvert < \epsilon,
\end{multline}
uniformly for all $\tau \in [0, 1]$ and $x, y \in [0, M]$, where $k^{(V_{\tau})}_{n, n}(x,y)$ and $k^{(\alpha, \theta)}(x,y)$ are defined in \eqref{def:knj} and \eqref{def:kalphatheta}, respectively.
\end{lemma}
Lemma \ref{prop:generalized} is a straightforward generalization of Lemma \ref{lem:asyknnV} which deals with the special case $\tau = 1$. Note that if $V$ satisfies \eqref{eq:loggrowth} and \eqref{eq:one-cut_regular}, then for all $\tau \in [0, 1]$, $V_{\tau}$ also satisfy \eqref{eq:loggrowth} and \eqref{eq:one-cut_regular}, so they are also one-cut regular, due to the discussion below \eqref{eq:one-cut_regular}. For a proof of Lemma \ref{prop:generalized}, one needs to show an extra result that the estimate holds uniformly for all $\tau \in [0, 1]$, which can be done by using the continuity of $V_{\tau}$ with respect to $\tau$ and we omit the details here.

The strategy of proof now is to split the summation in the correlation kernel $K_n$ into two parts, and estimate each part separately. To this end, we see from the uniqueness of biorthogonal polynomials satisfying \eqref{eq:pqbioOP} that for $j=1,\ldots,n$,
\begin{equation}
  p^{(V)}_{n, j}(x) = \left( \frac{j}{n} \right)^{\frac{j}{r}} p^{(V_{j/n})}_{j, j} ( (n/j)^{\frac{1}{r}} x ), \quad q^{(V)}_{n, j}(y^{\theta}) = \left( \frac{j}{n} \right)^{\frac{j \theta }{r}} q^{(V_{j/n})}_{j, j} ( (n/j)^{\frac{\theta}{r}} y^{\theta} ),
\end{equation}
and
\begin{equation}
  k^{(V)}_{n, j}(x, y) = \left( \frac{n}{j} \right)^{\frac{\alpha + 1}{r}} k^{(V_{j/n})}_{j, j} ( (n/j)^{\frac{1}{r}} x, (n/j)^{\frac{1}{r}} y ).
\end{equation}
Thus, we have, for any fixed positive integer $N$ and $n > N$,
\begin{multline}
  n^{-(\alpha + 1)(1 + \frac{1}{\theta}) +1 } \sum^n_{j = N + 1} k^{(V)}_{n, j} \left( \frac{x}{n^{1 + 1/\theta}}, \frac{y}{n^{1 + 1/\theta}} \right) =
 \\
  \sum^n_{j = N + 1} \left( \frac{j}{n} \right)^{(\alpha + 1)(\frac{1}{\theta} - \frac{1}{r}) + \alpha} j^{-\alpha(1 + \frac{1}{\theta}) - \frac{1}{\theta}} k^{(V_{j/n})}_{j, j} \left( \left( \frac{j}{n} \right)^{1 + \frac{1}{\theta} - \frac{1}{r}} \frac{x}{j^{1 + 1/\theta}}, \left( \frac{j}{n} \right)^{1 + \frac{1}{\theta} - \frac{1}{r}} \frac{y}{j^{1 + 1/\theta}} \right).
\end{multline}
Given $M, \epsilon > 0$, if we further take $N > N_{M, \epsilon}$ in the above formula with $N_{M, \epsilon}$ given in Lemma \ref{prop:generalized}, it then follows that
\begin{multline} \label{eq:est_K_first}
 \left\lvert \frac{1}{n^{(\alpha + 1)(1 + 1/\theta)}} \sum^n_{j = N + 1} k^{(V)}_{n, j} \left( \frac{x}{n^{1 + 1/\theta}}, \frac{y}{n^{1 + 1/\theta}} \right) - \frac{1}{n} \sum^n_{j = N + 1} f \left( \frac{j}{n} \right) \right\rvert
\\
\leq \frac{\epsilon }{n} \sum^n_{j = N + 1} \left( \frac{j}{n} \right)^{(\alpha + 1)(\frac{1}{\theta} - \frac{1}{r}) + \alpha}
< \epsilon',
\end{multline}
for $x, y \in [0, M]$, where
\begin{multline}
f(t) = t^{(\alpha + 1)(\frac{1}{\theta} - \frac{1}{r}) + \alpha} \theta (c^{(V_{t})})^{-\frac{\theta}{\theta + 1}} (\rho^{(V_{t})})^{\alpha(1 + \frac{1}{\theta}) + \frac{1}{\theta}}
\\
\times k^{(\alpha, \theta)} \left( t^{1 + \frac{1}{\theta} - \frac{1}{r}} (\rho^{(V_{t})})^{1 + \frac{1}{\theta}} x, t^{1 + \frac{1}{\theta} - \frac{1}{r}} (\rho^{(V_{t})})^{1 + \frac{1}{\theta}} y \right)
\end{multline}
and $\epsilon' = \epsilon/[(\alpha + 1)(1 + \frac{1}{\theta} - \frac{1}{r})]$.
Since $k^{(\alpha, \theta)}(x, y)$ is a continuous function in $x, y \in [0, \infty)$, and $c^{(V_t)}$ and $\rho^{(V_t)}$ are continuous functions in $t$ with values in a compact subset of $(0, \infty)$ as $t \in [0, 1]$, we have that $f(t)$ is continuous for $t \in (0, 1]$ with $f(t) = \bigO(t^{(\alpha + 1)(\frac{1}{\theta} - \frac{1}{r}) + \alpha})$ as $t \to 0_+$. Thus, we observe that the summation involving $f(j/n)$ in \eqref{eq:est_K_first} is a Riemann sum of a definite integral as $n\to \infty$, that is,
\begin{equation} \label{eq:est_k_sum_main}
  \lim_{n \to \infty} \frac{1}{n} \sum^n_{j = N + 1} f \left( \frac{j}{n} \right) = \int^1_0 f(t)dt.
\end{equation}
Also, it is not hard to see that for a fixed $N$, uniformly for $x, y \in [0, M]$,
\begin{multline} \label{eq:est_k_sum_remainder}
\lim_{n \to \infty} \frac{1}{n^{(\alpha + 1)(1 + \frac{1}{\theta})}} \sum^N_{k = 0} k^{(V)}_{n, k}\left( \frac{x}{n^{1 + 1/\theta}}, \frac{y}{n^{1 + 1/\theta}} \right)
\\
= \lim_{n \to \infty} \frac{1}{n^{(\alpha + 1)(1 + \frac{1}{\theta})}} \sum^N_{k = 0} k^{(V_0)}_{n, k}\left( \frac{x}{n^{1 + 1/\theta}}, \frac{y}{n^{1 + 1/\theta}} \right) = 0,
\end{multline}
where the second identity requires a scaling argument. As a consequence, a combination of  \eqref{def:knj}, \eqref{eq:defn_kernel_integral}, and the estimates \eqref{eq:est_K_first}, \eqref{eq:est_k_sum_main} and \eqref{eq:est_k_sum_remainder} implies that
we only need to show the equality
\begin{equation} \label{eq:integral_for_kernel}
  \int^1_0 f(t) dt = \theta^2 \int^{(\rho^{(V)})^{1 + 1/\theta}}_0 u^{\alpha} k^{(\alpha, \theta)}(ux, uy) du,
\end{equation}
to conclude Theorem \ref{thm:kernel}, which will be the task for the rest of the proof.

To show the integral identity \eqref{eq:integral_for_kernel}, by the change of variable
\begin{equation}\label{def:u}
u =  t^{1 + \frac{1}{\theta} - \frac{1}{r}} (\rho^{(V_{t})})^{1 + \frac{1}{\theta}},
\end{equation}
it suffices to show that
\begin{equation} \label{eq:last_main_id}
  \theta du = t^{\frac{1}{\theta} - \frac{1}{r}} (c^{(V_{t})})^{-\frac{\theta}{\theta + 1}} (\rho^{(V_{t})})^{\frac{1}{\theta}} dt.
\end{equation}
To this end, we introduce an intermediate variable $\c = c^{(V/t)}$. Note that by \eqref{eq:defn_c}, $\c=t^{1/r}c^{(V_{t})}$ depends on $t \in (0, 1]$ analytically. Thus, we can view $t$ and $u$ as functions of $\c$ locally, and the identity \eqref{eq:last_main_id} is reduced to
\begin{equation} \label{eq:alt_last}
  \theta \frac{du}{d\c} = t^{\frac{1}{\theta} - \frac{1}{r}} (c^{(V_{t})})^{-\frac{\theta}{\theta + 1}} (\rho^{(V_{t})})^{\frac{1}{\theta}} \frac{dt}{d\c}.
\end{equation}
We next evaluate the functions $dt/d\c$ and $du/d\c$ in \eqref{eq:alt_last} explicitly one by one. For $dt/d\c$, it follows from  \eqref{eq:defn_c} that
\begin{equation} \label{eq:dtau_dc}
  \frac{1}{2\pi i} \oint_{\gamma} \frac{\widehat{U}(J_{\c}(s))}{s} ds = (1 + \theta)t, \qquad  \widehat{U}(z) = zV'(z),
\end{equation}
where $\gamma$ is the counterclockwise oriented closed contour consisting of the union of $\gamma_1$ and $\gamma_2$ in Figure \ref{fig:mapJ}. Hence,
\begin{equation}\label{eq:dtaudc}
  \frac{dt}{d\c} = \frac{1}{\theta + 1} \frac{1}{2\pi i} \oint_{\gamma} \widehat{U}'(J_{\c}(s)) \left( \frac{s + 1}{s} \right)^{1 + \frac{1}{\theta}} ds.
\end{equation}
To evaluate $du/d\c$, we note that
\begin{align} \label{eq:simplify_integral_tau}
&\int^{b^{(V_{t})}}_0 (V''_{t}(y)y + V'_{t}(y)) \Im \left( \frac{-1}{1 + I^{(V_{t})}_+(y)} \right) dy \nonumber \\
    &=  \frac{1}{2i} \oint_{\gamma} \frac{V''_{t}(J_{c^{(V_{t})}}(s)) J_{c^{(V_{t})}}(s) + V'_{t}(J_{c^{(V_{t})}}(s))}{1 + s} d J_{c^{(V_{t})}}(s) \nonumber \\
    &=  \frac{1}{2i} \oint_{\gamma} \frac{\widehat{V}''_{t}(J_{\c}(s)) J_{\c}(s) + \widehat{V}'_{t}(J_{\c}(s))}{1 + s} d J_{\c}(s) = \frac{1}{2t i}\oint_{\gamma} \frac{V''(J_{\c}(s)) J_{\c}(s) + V'(J_{\c}(s))}{1 + s} d J_{\c}(s),
\end{align}
where $\widehat{V}_{t}(x) =  V(x)/t$ and we have made use of the fact that  $\c=t^{1/r}c^{(V_{t})}$ in the second equality. It then follows from \eqref{def:u}, \eqref{eq:defn_rho}, \eqref{def:d1} and \eqref{eq:simplify_integral_tau} that
\begin{equation} \label{eq:intro_U}
  u = \theta^{-1 - \frac{1}{\theta}} \c^{-1} \left( \frac{1}{2\pi i} \oint_{\gamma} \frac{\widehat{U}'(J_{\c}(s))}{1 + s} d J_{\c}(s) \right)^{1 + \frac{1}{\theta}},
\end{equation}
where $\widehat{U}(z)$ is defined in \eqref{eq:dtau_dc}. By taking derivatives on both sides of \eqref{eq:intro_U}, we obtain
\begin{multline}\label{eq:dudc}
  \frac{d u}{d \c} = \theta^{-1 - \frac{1}{\theta}} \c^{-2} \left( \frac{1}{2\pi i} \oint_{\gamma} \frac{\widehat{U}'(J_{\c}(s))}{1 + s} d J_{\c}(s) \right)^{\frac{1}{\theta}} \\
  \times \left( \frac{1}{\theta} \frac{1}{2\pi i} \oint_{\gamma} \frac{\widehat{U}'(J_{\c}(s))}{1 + s} d J_{\c}(s) + \frac{\theta + 1}{\theta} \frac{1}{2\pi i} \oint_{\gamma} \frac{\widehat{U}''(J_{\c}(s)) J_{\c}(s)}{1 + s} d J_{\c}(s) \right).
\end{multline}

Finally, substituting \eqref{eq:dtaudc}, \eqref{eq:intro_U} and \eqref{eq:dudc} into \eqref{eq:alt_last}, we find that the identity \eqref{eq:alt_last} reduces to
\begin{equation} \label{eq:final_identity_kernel}
  \frac{\c}{\theta + 1}  \oint_{\gamma} \widehat{U}'(J_{\c}(s)) \left( \frac{s + 1}{s} \right)^{1 + \frac{1}{\theta}} ds =   \frac{1}{\theta} \oint_{\gamma} \frac{\widehat{U}'(J_{\c}(s))}{1 + s} d J_{\c}(s) + \frac{\theta + 1}{\theta} \oint_{\gamma} \frac{\widehat{U}''(J_{\c}(s)) J_{\c}(s)}{1 + s} d J_{\c}(s).
\end{equation}
Since the left-hand side of \eqref{eq:final_identity_kernel} is equal to
\begin{equation}
  \c  \oint_{\gamma} \widehat{U}'(J_{\c}(s)) \left( \frac{s + 1}{s} \right)^{\frac{1}{\theta}} ds - \frac{\theta}{\theta + 1} \oint_{\gamma} \widehat{U}'(J_{\c}(s)) d J_{\c}(s) = \c \oint_{\gamma} \widehat{U}'(J_{\c}(s)) \left( \frac{s + 1}{s} \right)^{\frac{1}{\theta}} ds,
\end{equation}
we only need to show that
\begin{equation}\label{eq:finalreduc}
  \theta \c  \oint_{\gamma} \widehat{U}'(J_{\c}(s)) \left( \frac{s + 1}{s} \right)^{\frac{1}{\theta}} ds -  \oint_{\gamma} \frac{\widehat{U}'(J_{\c}(s))}{1 + s} d J_{\c}(s) - (\theta + 1)  \oint_{\gamma} \frac{\widehat{U}''(J_{\c}(s)) J_{\c}(s)}{1 + s} d J_{\c}(s)
  = 0.
\end{equation}
As one can check that the left-hand side of the formula above can be written as
\begin{equation}
  \frac{1}{2\pi i} \oint_{\gamma} \frac{d}{ds} \left( \theta \widehat{U}(J_{\c}(s)) - (\theta + 1) \c \widehat{U}'(J_{\c}(s)) \left( \frac{s + 1}{s} \right)^{\frac{1}{\theta}} \right) ds,
\end{equation}
the equality \eqref{eq:finalreduc} holds automatically.

\appendix
\section{The Meijer G-functions}\label{appendix:Meijer}
By definition, the Meijer G-function is given by the
following contour integral in the complex plane:
\begin{equation}\label{def:Meijer}
\MeijerG{m, n}{p, q}{a_1,\ldots,a_p}{b_1,\ldots,b_q}{z}
=\frac{1}{2\pi i}\int_L
\frac{\prod_{j=1}^m\Gamma(b_j+u)\prod_{j=1}^n\Gamma(1-a_j-u)}
{\prod_{j=m+1}^q\Gamma(1-b_j-u)\prod_{j=n+1}^p\Gamma(a_j+u)}z^{-u}
d u,
\end{equation}
where $\Gamma$ denotes the usual gamma function and the branch cut
of $z^{-u}$ is taken along the negative real axis. It is also
assumed that
\begin{itemize}
  \item $0\leq m\leq q$ and $0\leq n \leq p$, where $m,n,p$ and $q$
  are integer numbers;
  \item The real or complex parameters $a_1,\ldots,a_p$, and
  $b_1,\ldots,b_q$, satisfy the conditions
  \begin{equation*}
  a_k-b_j \neq 1,2,3, \ldots, \quad \textrm{for $k=1,2,\ldots,n$ and $j=1,2,\ldots,m$,}
  \end{equation*}
  i.e., none of the poles of $\Gamma(b_j+u)$, $j=1,2,\ldots,m$ coincides
  with any poles of $\Gamma(1-a_k-u)$, $k=1,2,\ldots,n$.
\end{itemize}
The contour $L$ is chosen in such a way that all the poles of
$\Gamma(b_j+u)$, $j=1,\ldots,m$ are on the left of the path, while
all the poles of $\Gamma(1-a_k-u)$, $k=1,\ldots,n$ are on the right,
which is usually taken to go from $-i\infty$ to $i\infty$. Most of the known special functions can be viewed as special cases of the Meijer G-functions. For more details, we refer to the references \cite{Beals-Szmigielski13,Luke69,Boisvert-Clark-Lozier-Olver10}.

\section{The Airy parametrix}\label{app:Airy}
Let $y_0$, $y_1$ and $y_2$ be the functions defined by
\begin{equation}
  y_0(\zeta) = \sqrt{2\pi}e^{-\frac{\pi i}{4}} \Ai(\zeta), \quad y_1(\zeta) = \sqrt{2\pi}e^{-\frac{\pi i}{4}} \omega\Ai(\omega \zeta), \quad y_2(\zeta) = \sqrt{2\pi}e^{-\frac{\pi i}{4}} \omega^2\Ai(\omega^2 \zeta),
\end{equation}
where $\Ai$ is the usual Airy function (cf. \cite[Chapter 9]{Boisvert-Clark-Lozier-Olver10}) and $\omega=e^{2\pi i/3}$. We then define a $2\times 2$ matrix-valued function $\Psi^{(\Ai)}$ by
\begin{equation}
\Psi^{(\Ai)}(\zeta)
= \left\{
    \begin{array}{ll}
      \begin{pmatrix}
      y_0(\zeta) &  -y_2(\zeta) \\
      y_0'(\zeta) & -y_2'(\zeta)
    \end{pmatrix}, & \hbox{$\arg \zeta \in (0,\frac{2\pi}{3})$,} \\
       \begin{pmatrix}
      -y_1(\zeta) &  -y_2(\zeta) \\
      -y_1'(\zeta) & -y_2'(\zeta)
    \end{pmatrix}, & \hbox{$\arg \zeta \in (\frac{2\pi}{3},\pi)$,} \\
       \begin{pmatrix}
      -y_2(\zeta) &  y_1(\zeta) \\
      -y_2'(\zeta) & y_1'(\zeta)
    \end{pmatrix}, & \hbox{$\arg \zeta \in (-\pi,-\frac{2\pi}{3})$,} \\
      \begin{pmatrix}
      y_0(\zeta) &  y_1(\zeta) \\
      y_0'(\zeta) & y_1'(\zeta)
    \end{pmatrix}, & \hbox{$\arg \zeta \in  (-\frac{2\pi}{3},0)$.}
    \end{array}
  \right.
\end{equation}
It is well-known that $\det (\Psi^{(\Ai)}(z))=1$ and $\Psi^{(\Ai)}(\zeta)$ is the unique solution of the following RH problem; cf. \cite{Deift99}.
\begin{RHP}\label{rhp:Ai}
\textrm{}
\begin{enumerate}
  \item[\rm (1)]   $ \Psi^{(\Ai)}(\zeta)$ is analytic in
  $\mathbb{C} \setminus \Gamma_{\Ai}$, where the contour $\Gamma_{\Ai}$ is defined in \eqref{def:AiryContour}
 with the orientation shown in Figure \ref{fig:jumps-Psi-A}.
\item [\rm (2)]
  For $z \in \Gamma_{\Ai}$, we have
  \begin{equation}
     \Psi^{(\Ai)}_+(\zeta) =  \Psi^{(\Ai)}_-(\zeta)
    \begin{cases}
      \begin{pmatrix}
        1 & 1 \\
        0 & 1
      \end{pmatrix},
      & \arg \zeta =0, \\
      \begin{pmatrix}
        1 & 0 \\
        1 & 1
      \end{pmatrix},
      & \arg \zeta = \pm \frac{2\pi }{3}, \\
      \begin{pmatrix}
        0 & 1  \\
        -1 & 0
      \end{pmatrix},
      & \arg \zeta = \pi.
    \end{cases}
  \end{equation}
\item [\rm (3)]
  As $\zeta \to \infty$, we have
    \begin{equation}
    \Psi^{(\Ai)}(\zeta)=\frac{1}{\sqrt{2}}\zeta^{-\frac14 \sigma_3}
\begin{pmatrix}
1 & 1
\\
-1 & 1
\end{pmatrix}e^{-\frac{\pi i}{4} \sigma_3}(I+\bigO(\zeta^{-\frac32}))e^{-\frac23 \zeta^{\frac32}\sigma_3}.
\end{equation}
\item [\rm (4)]
  As $\zeta \to 0$, we have $\Psi^{(\Ai)}(\zeta) = \bigO(1)$ and $\Psi^{(\Ai)}(\zeta)^{-1} = \bigO(1)$, which is understood in an entry-wise manner.
\end{enumerate}
\end{RHP}

\begin{figure}[t]
\begin{center}
   \setlength{\unitlength}{1truemm}
   \begin{picture}(100,70)(-5,2)
       \put(40,40){\line(-2,-3){16}}
       \put(40,40){\line(-2,3){16}}
       \put(40,40){\line(-1,0){30}}
       \put(40,40){\line(1,0){30}}

       \put(30,55){\thicklines\vector(2,-3){1}}
       \put(30,40){\thicklines\vector(1,0){1}}
       \put(50,40){\thicklines\vector(1,0){1}}
       \put(30,25){\thicklines\vector(2,3){1}}

        \put(39,36.3){$0$}
       \put(40,40){\thicklines\circle*{1}}
\end{picture}
   \caption{The jump contour $\Gamma_{\Ai}$ for the RH problem \ref{rhp:Ai} for $\Psi^{(\Ai)}$.}
   \label{fig:jumps-Psi-A}
\end{center}
\end{figure}
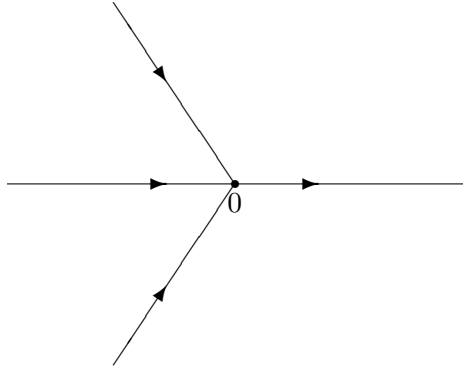

\section*{Acknowledgements}
We thank the anonymous referees for their careful reading and constructive suggestions. Dong Wang was partially supported by Singapore AcRF Tier 1 Grant R-146-000-217-112, National Natural Science Foundation of China under grant number 11871425, and the University of Chinese Academy of Sciences start-up grant 118900M043. Lun Zhang was partially supported by National Natural Science Foundation of China under grant number 11822104, ``Shuguang Program'' supported by Shanghai Education Development Foundation and Shanghai Municipal Education Commission, and The Program for Professor of Special Appointment (Eastern Scholar) at Shanghai Institutions of Higher Learning. We thank Arno Kuijlaars and Leslie Molag for helpful discussions at the very early stage of the project.


\end{document}